\documentclass[]{article}
\RequirePackage[OT1]{fontenc}
\RequirePackage{float}
\usepackage{amsmath}
\usepackage{amssymb}
\usepackage{bm}
\usepackage{setspace}
\usepackage[margin=1.1in]{geometry}
\usepackage[utf8]{inputenc}
\usepackage[english]{babel}
\usepackage{mathrsfs}
\usepackage{subfigure}
\usepackage{amsthm}
\usepackage[authoryear]{natbib}
\usepackage{lipsum}
\usepackage{multirow}
\usepackage{graphicx}
\newtheorem{theorem}{Theorem}[section]
\newtheorem{cor}{Corollary}[section]
\newtheorem{lemma}[theorem]{Lemma}
\newtheorem{assumption}{Assumption}
\newtheorem{remark}{Remark}

\DeclareMathOperator*{\argmax}{arg\,max}
\title{Consistent Bayesian Sparsity Selection for High-dimensional Gaussian DAG Models with Multiplicative and 
Beta-mixture Priors}
\author{Xuan Cao \\University of Cincinnati
	\and 
	Kshitij Khare \\ University of Florida
\and Malay Ghosh \\ University of Florida}
\begin{document}
	\allowdisplaybreaks
	\doublespacing
	\noindent
	\maketitle
			\begin{abstract}
			Estimation of the covariance matrix for high-dimensional multivariate datasets is a challenging and important 
			problem in modern statistics. In this paper, we focus on high-dimensional Gaussian DAG models where 
			sparsity is induced on the Cholesky factor $L$ of the inverse covariance matrix. In recent work,  
			(\citep*{CKG:2017}), we established high-dimensional sparsity selection consistency for a hierarchical 
			Bayesian DAG model, where an Erdos-Renyi prior is placed on the sparsity pattern in the Cholesky factor  
			$L$, and a DAG-Wishart prior is placed on the resulting non-zero Cholesky entries. In this 
			paper we significantly improve and extend this work, by (a) considering more diverse and effective priors on 
			the sparsity pattern in $L$, namely the beta-mixture prior and the multiplicative prior, and (b) establishing 
			sparsity selection consistency under significantly relaxed conditions on $p$, and the sparsity pattern of the 
			true model. We demonstrate the validity of our theoretical results via numerical 
			simulations, and also use further simulations to demonstrate that our sparsity selection approach is 
			competitive with existing state-of-the-art methods including both frequentist and Bayesian approaches in 
			various settings. 	
	\end{abstract}
	\section{Introduction}\label{sec1}
	
\noindent	
Covariance estimation and selection is a fundamental problem in multivariate statistical inference, and plays a crucial role in 
many data analytic methods. In high-dimensional settings, where the number of variables is much larger than 
the number of samples, the sample covariance matrix (traditional estimator for the population covariance matrix) can perform 
rather poorly. See \citep*{Bickel:Levina:2008,bickel:2008:thres,elkaroui:2007} for example. To address the challenge posed 
by high-dimensionality, several promising methods have been proposed in the literature. In particular, methods inducing 
sparsity in the covariance matrix $\Sigma$, its inverse $\Omega$, or the Cholesky factor of the inverse, have proven to be 
very effective in applications. In this paper, we focus on imposing sparsity on the Cholesky factor of the inverse covariance 
(precision) matrix. These models are also referred to as Gaussian DAG models. 
	
Consider a case when we have i.i.d. observations ${\bf Y}_1, {\bf Y}_2, \cdots, {\bf Y}_n$ obeying a $p$-variate normal 
distribution with mean vector ${\bf 0}$ and covariance matrix $\Sigma$. Let $\Omega = LD^{-1}L^T$ be the unique modified 
Cholesky decomposition of the inverse covariance matrix $\Omega = \Sigma^{-1}$, where $L$ is a lower triangular matrix 
with unit diagonal entries, and $D$ is a diagonal matrix with positive diagonal entries. A given sparsity pattern on $L$ 
corresponds to certain conditional independence relationships, which can be encoded in terms of a directed acyclic graph 
$\mathscr D$ on the set of $p$ variables as follows: if the variables $i$ and $j$ do not share an edge in $\mathscr D$, then 
$L_{ij} = 0$ (see Section \ref{sec2} for more details). 

There are two major approaches in the literature for sparse estimation of $L$. The first approach is based on regularized 
likelihood/pseudolikelihood using $\ell_1$ penalization. See \citep*{HLPL:2006, Rutimann:Buhlmann:2009, 
Shojaie:Michailidis:2010, RLZ:2010, AAZ:2015, Yu:Bien:2016, KORR:2017}. Some of these frequentist approaches assume 
$L$ is banded, i.e., the elements of $L$ that are far from the diagonal are taken to be zero. The other methods put 
restrictions on the maximum number of non-zero entries in $L$. 

On the Bayesian side, when the underlying graph is known, literature exists that explores the posterior convergence rates for 
Gaussian concentration graph models, which induce sparsity in the inverse covariance matrix $\Omega$. See 
\citep*{Banerjee:Ghosal:2014, Banerjee:Ghosal:2015, XKG:2015, LL:posterior} for example. Gaussian concentration graph 
models and Gaussian DAG models studied in this paper intersect only at perfect DAG models, which are equivalent to 
decomposable concentration graphical models. For general Gaussian DAG models, comparatively fewer works have tackled 
with asymptotic consistency properties. Recently, \citet*{CKG:2017} establish both strong model selection consistency and 
posterior convergence rates for sparse Gaussian DAG models  in a high-dimensional regime. In particular, the authors 
consider a hierarchical Gaussian DAG model with DAG-Wishart priors introduced in \citep*{BLMR:2016} on the Cholesky 
parameter space and independent Bernoulli$(q)$ priors for each edge in the DAG (the so-called Erdos-Renyi prior). 
However, the sparsity assumptions on the true model required to establish consistency are rather restrictive. In addition, as a 
result of the extremely small value of the edge probability $q$ in the Bernoulli prior, the simulations studies always tend to 
favor more sparse models under smaller values of $p$. \citet*{LLL:2018} also explore the Cholesky factor selection 
consistency under the empirical sparse Cholesky (ESC) prior and $\alpha$-posteriors. Compared with \citep*{CKG:2017}, 
under relaxed conditions in terms of the dimensionality, sparsity and lower bound of the non-zero elements in the Cholesky 
factor, \citet*{LLL:2018} establish strong model selection consistency with the $\alpha$-posterior distribution. 

It recently came to our attention that two more flexible alternative priors compared to the Erdos-Renyi prior have been considered in the 
undirected graphical models literature: (a) the multiplicative prior \citep*{Tan:2017}, and (b) the beta-mixture prior 
\citep*{Carvalho:Scott:2009}. Both priors are more diverse than the Erdos-Renyi prior (the Erdos-Renyi prior can be obtained 
as a degenerate version of these priors), and have various attractive properties. For example, the multiplicative model prior 
can account for greater variability in the degree distribution as compared to the Erods-Renyi model, while the beta-mixture 
prior allows for stronger control over the number of spurious edges and corrects for multiple hypothesis testing automatically. 
We provide the algebraic forms of these priors in Section 3 and Section 5 respectively, and refer the reader to 
\citep*{Carvalho:Scott:2009, Tan:2017} for a detailed discussion of their properties. 

To the best of our knowledge, a rigorous investigation of high-dimensional posterior consistency properties with the 
multiplicative prior or the beta-mixture prior has not been undertaken for either undirected graphical models or Gaussian 
DAG models. Hence, our goal was to investigate if high-dimensional consistency results could be established under these 
two more diverse and algebraically complex class of prior distributions in the Gaussian DAG model setting. Another goal was 
to investigate if these high-dimensional posterior consistency results can be obtained under much weaker conditions as 
compared to \citep*{CKG:2017}, particularly  conditions similar to those in \citep*{LLL:2018}. This was a challenging goal, 
particularly for the multiplicative model prior, as the prior mass function is not available in closed form (note that the mass 
functions for the Erdos-Renyi, ESC and beta-mixture priors are available in closed form). 

As the main contributions of this paper, we establish high-dimensional posterior consistency results for Gaussian DAG 
models with spike and slab priors on the Cholesky factor $L$, under both the multiplicative prior as well as the beta-mixture 
prior on the sparsity pattern in $L$ (Theorems \ref{thm4} to \ref{thm3}), using assumptions similar to those in \citep*{LLL:2018} (where 
a different setting of ESC priors and $\alpha$-posteriors is used). Also, through simulation studies, we demonstrate that the  
models studied in this paper can outperform existing state-of-the-art methods including both penalized likelihood and 
Bayesian approaches in different settings.
	
The rest of paper is organized as follows. Section \ref{sec2} provides background material regarding Gaussian DAG 
model and introduce the spike and slab prior on the Choleksy factor. In Section \ref{sec:modelspecification}, we revisit the multiplicative prior, and  present 
our hierarchical Bayesian model and the parameter class for the inverse covariance matrices. Model selection 
consistency results for both the multiplicative prior and the beta-mixture prior are stated in Section \ref{sec:modelconsistency} and Section \ref{sec:beta_mixture} with proofs provided in Section 
\ref{sec:modelselectionproofs}. In Section \ref{sec:experiments} we use simulation experiments to illustrate the 
posterior ratio consistency result, and demonstrate the benefits of our Bayesian approach and computation procedures 
for Choleksy factor selection vis-a-vis existing Bayesian and penalized likelihood approaches. We end our paper with a 
discussion session in Section \ref{sec:discussion}.
	
	\section{Preliminaries}\label{sec2}
	
	\noindent
	In this section, we provide the necessary background material from graph theory, 
	Gaussian DAG models, and also introduce our spike and slab prior on the Cholesky parameter. 
	
	\subsection{Gaussian DAG Models} \label{sec2.1}
	
	\noindent
	We consider the multivariate Gaussian distribution 
	\begin{equation} \label{mgd}
	\bm Y \sim N_p(0, \Omega^{-1}),
	\end{equation} 
	where $\Omega$ is a $p \times p$ inverse covariance matrix. Any positive definite matrix $\Omega$ can be uniquely 
	decomposed as $\Omega = 
	LD^{-1}L^T$, where $L$ is a lower triangular matrix with unit diagonal entries, and $D$ 
	is a diagonal matrix with positive diagonal entries. This decomposition is known as the 
	modified Cholesky decomposition of $\Omega$ (see for example 
	\cite{Pourahmadi:2007}). In particular, the model (\ref{mgd}) can be interpreted as a Gaussian DAG model depending 
	on the sparsity pattern of $L$. 
	
	A directed acyclic graph (DAG) $\mathscr{D} = (V,E)$ consists 
	of the vertex set $V = \{1,\ldots,p\}$ and an edge set $E$ such that there is no directed 
	path starting and ending at the same vertex. As in \citep*{BLMR:2016, CKG:2017}, we will without 
	loss of generality assume a parent ordering, where that all the edges are directed from 
	larger vertices to smaller vertices. For several applications in genetics, finance, and 
	climate sciences, a location or time based ordering of variables is naturally available. For example, in genetic datasets, 
	the variables can be genes or SNPs located contiguously on a chromosome, and their spatial location 
	provides a natural ordering.  More examples can be found in 
	\citep*{HLPL:2006, Shojaie:Michailidis:2010, Yu:Bien:2016, KORR:2017}. The set of parents of $i$, denoted 
	by $pa_i(\mathscr D)$, is the 
	collection of all vertices which are larger than $i$ and share an edge with $i$. 
	Similarly, the set of children of $i$, denoted by $chi_i(\mathscr D)$, is the collection of all vertices 
	which are smaller than $i$ and share an edge with $i$. 
	
	A Gaussian DAG model over a given DAG $\mathscr{D}$, denoted by 
	$\mathscr{N}_{\mathscr{D}}$, consists of all multivariate Gaussian distributions which 
	obey the directed Markov property with respect to a DAG $\mathscr{D}$. In particular, 
	if $\bm{Y}=(Y_1, \ldots, Y_p)^T \sim N_p(0,\Sigma)$ and $N_p(0,\Sigma = \Omega^{-1}) \in 
	\mathscr{N}_{\mathscr{D}}$, then $$Y_i \perp \bm{Y}_{\{i+1,\ldots,p\}\backslash pa_i(\mathscr D)}|
	\bm{Y}_{pa_i(\mathscr D)},$$for each $1\le i \le p$. Furthermore, it is well-known that if $\Omega = LD^{-1}L^T$ is the 
	modified Cholesky decomposition of $\Omega$, then $N_p(0,\Omega^{-1}) \in 
	\mathscr{N}_{\mathscr{D}}$ if and only if $L_{ij} = 0$ whenever $i \notin pa_j (\mathscr D)$. In other 
	words, the structure of the DAG $\mathscr{D}$ is uniquely reflected in the sparsity pattern of the Cholesky factor $L$. 
	In light of this, it is often more convenient to 
	reparametrize the inverse covariance matrix in terms of the Cholesky parameter $(L,D)$.

	\subsection{Notations} \label{sec2.2}
	
	\noindent
	Consider the modified cholesky decomposition $\Omega = LD^{-1}L^T$, where $L$ is a lower triangular matrix with all 
	the unit diagonals and $D = \mbox{Diag }\{d_1, d_2, \ldots, d_p\}$, where $d_i$'s are all positive. We suggest to impose 
	spike and slab priors on the lower diagonal of $L$ to recover the sparse structure of the Cholesky factor. To facilitate this 
	purpose, we introduce latent binary variables $Z = \left\{Z_{21}, \ldots, Z_{kj}, \ldots, Z_{p,p-1}\right\}$ for $1 \le j < k \le p$ to indicate whether $L_{kj}$ is active, i.e., $Z_{kj} = 1$  if 
	$L_{kj} \neq 0$ and 0, otherwise. We can view the binary variable $Z_{kj}$ as the indicator for the sparsity pattern of $L$. In other words, for each $1 \le j \le p-1$, let $Z_j$, a subset of $\left\{j+1, j+2, 
	\ldots, p\right\}$, be the index set of all non-zero components in $\left\{Z_{j+1,j}, \ldots, Z_{p,j}\right\}$. $Z_j$ explicitly
	 gives the support of the Cholesky factor and the sparsity pattern of the underlying DAG. Denote $|Z_j| = \sum_{k=j+1}^p 
	 Z_{kj}$ as the cardinality of set $Z_j$ for $ 1 \le j \le p-1$.
	 
	Following the definition of $Z$, for any $p \times p$ matrix $A$, 
	denote the column vectors $A_{Z.j}^> = (A_{kj})_{k \in Z_j}$ and 
	$A_{Z.i}^{\ge} = (A_{ii}, (A_{Z.i}^>)^T)^T.$ Also, let $A_{Z}^{>j} = (A_{ki})_{k,i \in Z_j}$, $$ A_{Z}^{ \ge i} = 
	\left[ \begin{matrix}
	A_{ii} & (A_{Z.i}^>)^T \\
	A_{Z.i}^> & A_{Z}^{>i}
	\end{matrix} \right]. 
	$$
	In particular, $A_{Z.p}^{\ge} = A_{Z}^{ \ge q} = A_{pp}$. 
	
	Next, we provide some additional required notation. For $x \in \mathbb{R}^p$, let $\lVert x \rVert_r = \left(\sum_{j=1}^{p} 
	|x_j|^r\right)^{\frac1 r}$ and $\lVert x \rVert_\infty = \max_j|x_j|$ represent the standard $l_r$ and $l_\infty$ norms. For a 
	$p \times p$ matrix $A$, let $eig_1(A) \le eig_2(A) \ldots eig_p(A)$ be the ordered eigenvalues of $A$ and denote
	\begin{align*}
	\lVert A \rVert_{\max} = \max_{1\le i,j \le p} |A_{ij}|,
	\end{align*}
	\begin{align*}
	\lVert A \rVert_{(r,s)} = \mbox{sup} \left\{\lVert Ax \rVert_s:\lVert x \rVert = 1 \right\}, \mbox{ for } 1 \le r,s < \infty.
	\end{align*}
	In particular, 
	$$\lVert A \rVert_{(1,1)} = \max_j\sum_i|A_{ij}|,\mbox{ } \lVert A \rVert_{(\infty, \infty)} = \max_i\sum_j|A_{ij}| \mbox{ and }
	\lVert A \rVert_{(2,2)} =  eig_p(A)^{\frac 1 2}.$$
	\subsection{Spike and Slab Prior on Cholesky Parameter}
	In this section, we specify our spike and slab prior on the Cholesky factor as follows.
	\begin{eqnarray}
	& & L_{kj} \mid d_j, Z_{kj} \overset{ind}{\sim} Z_{kj} N\left(0, \tau^2 d_j\right) + (1-Z_{kj}) \delta_0(L_{kj}),\quad 1 \le j 
	< k \le p, \label{model_spike_slab}\\
	& & d_j  \overset{ind}{\sim} \mbox{Inverse-Gamma}(\lambda_1, \lambda_2), \quad j = 1,2,\ldots,p \label{model_d},
	\end{eqnarray}

	\noindent
	for some constants $\tau, \lambda_1, \lambda_2 \ge 0$, where $\delta_0(L_{kj})$ denotes a point mass at $0$. We refer 
	to (\ref{model_spike_slab}) and (\ref{model_d}) as our spike and slab Cholesky (SSC) prior.  
	$Z_{jk} = 1$ implies $L_{jk}$ being the “signal” (i.e., from the slab component), and $Z_{jk} = 0$ implies $L_{jk}$ being 
	the noise (i.e., from the spike component). Note that to obtain 
	our desired asymptotic consistency results, appropriate conditions for these hyperparameters will be introduced in 
	Section \ref{sec:assumption}. \cite{Xu:Ghosh:2015} also impose this type of priors on the regression factors. Further 
	comparisons and discussion are provided in Remark \ref{spikeslabregression}.

	\begin{remark} \label{spikeslab:dagwishart}
	Note that in (\ref{model_d}), we are allowing the hyperparameters for the inverse-gamma prior to be zero. In 
	\citep*{CKG:2017}, the DAG-Wishart prior with multiple shape parameters introduced in \citep*{BLMR:2016} is placed on 
	the Cholesky parameter. As indicated in Theorem 7.3 in \citep*{BLMR:2016}, the DAG-Wishart distribution defined on the 
	Cholesky parameter space given a DAG yields the independent inverse-gamma distribution with strictly positive shape 
	and scale parameters on $d_j$  and multivariate Gaussian distribution on the non-zero elements in each column of $L$ 
	given $d_j$. Hence, for given DAG structures, there are some difference and connection between the DAG-Wishart prior 
	and our spike and slab prior. 
	\end{remark}

\section{Model Specification} \label{sec:modelspecification}
In this section, we revisit the multiplicative prior introduced in \citep*{Tan:2017} over space of graphs, and specify our hierarchical model. 
\subsection{Multiplicative Prior} \label{sec:multiplicative}
In the context of Gaussian graphical model, \citet*{Tan:2017} allow the probability of a link between nodes $k$ and $j$, $q_{kj}$ to vary with $i,j$ by taking $q_{kj} = \omega_k\omega_j$ and $0 < \omega_j < 1$ for each $1 \le j \le p$. The authors further treat each $\omega_i$ as a variable with a beta prior to adopt a fully Bayesian approach. The authors further utilize Laplace approximations, and through simulation studies, show that the proposed multiplicative model (following the nomenclature in \citep*{Tan:2017}) facilitates the purpose to encourage sparsity or graphs that exhibit particular degree patterns based on prior knowledge. Adapted to our framework, we consider the following multiplicative prior over the space of sparsity variation for the Cholesky factor. 
\begin{eqnarray} 
& &  \pi(Z \mid \omega_1, \ldots, \omega_{p}) = \prod_{1 \le j < k \le p}(\omega_k\omega_j)^{Z_{kj}}(1-\omega_k\omega_j)^{1-Z_{kj}}, \label{model7} \\ 
& & \omega_j \sim \mbox{Beta}(\alpha_1,\alpha_2), \qquad 1 \le j \le p,\label{model8}
\end{eqnarray}
where $\alpha_1, \alpha_{2}$ are positive constants. Compared with the universal indicator probability $q$ in an Erdos-Renyi prior, here we allow the variation attainable in the degree structure of each node through different values of $\omega_j$. Note that under the multiplicative prior, the marginal posterior for $Z$ can not be obtained in closed form, which leads to further challenges not only in the theoretical analysis, but also in the computational strategy. We will elaborate on this matter in Section \ref{sec:experiments}. 
\subsection{Hierarchical Model Formulation } \label{sec:model_formula}
Let $\bm{Y}_1, \bm{Y}_2, \ldots, \bm{Y}_n $ be 
independent and identically distributed $p$-variate Gaussian vectors with mean 
$0$ and true covariance matrix $\Sigma_0  = (\Omega_0 )^{-1}$, where $\Omega_0  = 
L_0 (D_0 )^{-1}(L_0 )^T$ is the modified Cholesky decomposition of 
$\Omega_0 $. Let $S = \frac1n\sum_{i=1} \bm{Y}_i \bm{Y}_i^T$ 
denotes the sample covariance matrix. The sparsity pattern of the true Choleksy factor $L_0 $ is uniquely encoded in the true binary variable denoted as $Z_0 $. Similar to \citep*{CKG:2017}, we also denote $d$ as the maximum number of 
non-zero entries in any column of $L_0 $, and $s = \min_{1 \leq j,i \leq p, i \in Z_j} |(L_0)_{ji}|$. For sequences  $a_n$ and $b_n$, $a_n \sim b_n$ means $\frac{a_n}{b_n} \rightarrow c$ for some constant $c > 0$, as $n \rightarrow \infty$. Let $a_n =o(b_n)$ represent $\frac{a_n}{b_n} \rightarrow 0$ as $n \rightarrow \infty$. 

The class of spike and slab Cholesky distributions in Section \ref{sec2} and the multiplicative priors in Section \ref{sec:multiplicative} can be used for Bayesian model selection of the Cholesky factor through the following hierarchical model,
\begin{eqnarray} \label{model:spec}
& & \bm{Y} \mid (D, L), Z \sim N_p \left( \bm 0, (LD^{-1}L^T)^{-1}\right)), \label{model1}\\
& & L_{kj} \mid d_j, Z_{kj} \overset{ind}\sim Z_{kj} N\left(\bm 0, \tau^2 d_j\right) + (1-Z_{kj}) \delta_0(L_{kj}), \quad 1 \le j < k \le p, \label{model2}\\
& & d_j  \overset{ind}{\sim} \mbox{Inverse-Gamma}(\lambda_1, \lambda_2), \quad j = 1,2,\ldots,p, \label{model4} \\ 
& &  \pi(Z \mid \omega_1, \ldots, \omega_{p}) = \prod_{1 \le j < k \le p}(\omega_k\omega_j)^{Z_{kj}}(1-\omega_k\omega_j)^{1-Z_{kj}}, \label{model5} \\ 
& & \omega_j \sim \mbox{Beta}(\alpha_1,\alpha_2), \qquad 1 \le j \le p,\label{model6}
\end{eqnarray}
where $\mbox{Beta}(\alpha_1,\alpha_2)$ represents the beta distribution with shape parameters $\alpha_1, \alpha_2$. The proposed hierarchical model now has five hyperparameters: the scale parameter $\tau>0$ in model (\ref{model2}) controlling the variance of the spike part in the spike and slab prior on each $L_{kj}$, the shape parameter $\lambda_1$ and scale parameter $\lambda_2$ in model (\ref{model4}), and the two positive shape parameters in the beta distribution in model (\ref{model6}). Further restrictions on these hyperparameters to ensure desired consistency will be specified in Section \ref{sec:assumption_hyper}.

The intuition behind this set-up with latent variables is that the elements in the Cholesky factor $L$ with zero or very small values will be identified with zero $Z$ values, while the active entries will be classified as $Z = 1$. We use the posterior probabilities of all the $\frac{p(p-1)}2$ latent variables $Z$ to identify the active elements in $L$. In particular, the following lemmas help specify the upper bound for the marginal probability ratio and the marginal posterior ratio for any ``non-true" model $Z$ compared with the true model $Z$ under the multiplicative prior. The proof will be provided in Section \ref{sec:proof_thm4}.
\begin{lemma}\label{graph_ratio_lemma}
	If the hyperparameter $\alpha_{2}$ in model (\ref{model6}) satisfies $\alpha_{2} \sim \max\left\{p^c, d^{\frac{2c}{c-2}}\right\}$, for $c > 2$, we have
\begin{align} 
\frac{\pi(Z)}{\pi(Z_0)} \le e^{2\alpha_1^2  + 2\alpha_1 + \frac 2 {\alpha_{1}}} \prod_{j=1}^{p}\frac{B(\alpha_1+ |Z_j|, \alpha_2)}{B(\alpha_1+ |{Z_0}_j|, \alpha_2)},
\end{align}
for $p \ge 4+\frac{4}{\alpha_1} + 2\sqrt{\alpha_1}$.	
\end{lemma}
\noindent
Lemma \ref{graph_ratio_lemma} further enables the marginalized posterior likelihood ratio to be upper bounded by decomposed prior terms absorbed into the product of items as follows. 
\begin{lemma} \label{newlemma1}
		If $\alpha_{2} \sim \max\left\{p^c, d^{\frac{2c}{c-2}}\right\}$, for $c > 2$, the marginal posterior ratio between any ``non-true" model $Z$ and the true model $Z_0$ under the multiplicative prior in (\ref{model7}) and (\ref{model8}) satisfies
	\begin{align} \label{posterior_multiplicative}
	&\frac{\pi({Z}|\bm{Y})}{\pi({Z}_0|\bm{Y})} \nonumber\\
	\le & M_1\prod_{j=1}^{p-1} (n\tau^2)^{-\frac{|Z_j| - |{Z_0}_j|}2} \frac{B(\alpha_1+ |Z_j|, \alpha_2)}{B(\alpha_1+ |{Z_0}_j|, \alpha_2)}\nonumber \\
	&\times \frac{|\tilde{S}_{Z_0}^{\ge j}|^{\frac12}}{|\tilde{S}_{Z}^{\ge j}|^{\frac12}}\left(\frac{\tilde{S}_{j|{Z_0}_j}}{\tilde{S}_{j|Z_j}}\right)^{\frac 1 2} \left(\frac{\tilde S_{j|{Z_0}_j}- \frac 1 {n\tau_{n,p}^2}+ \frac{2\lambda_2}{n}}{\tilde S_{j|{Z}_j} - \frac 1 {n\tau_{n,p}^2} + \frac{2\lambda_2}{n}}\right)^{\frac n 2 + \lambda_1} \nonumber\\
	\triangleq& M_1\times \prod_{j = 1}^{p-1}PR^\prime_j(Z,Z_0),
	\end{align} where $M_1 = e^{2\alpha_1^2  + 2\alpha_1 + \frac 2 {\alpha_{1}}}$, $\tilde{S} = S+\frac 1 {n\tau_{n,p}^2} I_p$ and $\tilde{S}_{j|Z_j} = \tilde{S}_{jj} - (\tilde{S}_{Z \cdot j}^>)^T(\tilde{S}_Z^{>j})^{-1} \tilde{S}_{Z \cdot j}^>$.
\end{lemma}

\section{Model Selection Consistency} \label{sec:modelconsistency}
In this section we will explore the high-dimensional asymptotic properties of the 
Bayesian model selection approach for the Cholesky factor specified in Section \ref{sec:model_formula}. 
For this purpose, we will work in a setting where the dimension $p = p_n$ of the data 
vectors, and the hyperparameters vary with the sample size $n$ and $p_n \ge n$. Assume that the data is actually being generated from a true model specified as follows. Let $\bm{Y}_1^n, \bm{Y}_2^n, \ldots, \bm{Y}_n^n$ be 
independent and identically distributed $p_n$-variate Gaussian vectors with mean 
$0$ and true covariance matrix $\Sigma_0^n = (\Omega_0^n)^{-1}$, where $\Omega_0^n = 
L_0^n(D_0^n)^{-1}(L_0^n)^T$ is the modified Cholesky decomposition of 
$\Omega_0^n$. The sparsity pattern of the true Choleksy factor $L_0^n$ is reflected in $Z_0^n$. Recall the definition in Section \ref{sec:model_formula} that $d_n$ is the maximum number of 
non-zero entries in any column of $L_0^n$, and $s_n = \min_{1 \leq j,i \leq p, i \in Z_j} |(L_0^n)_{ji}|$. In order to establish our asymptotic consistency results, we need the following mild assumptions with respective discussion/interpretation. 
\subsection{Assumptions} \label{sec:assumption}
\subsubsection{Assumptions on the True Parameter Class}
\begin{assumption}
	There exists $\epsilon_{0} \le 1$, such that for every $n \ge 1,$ $0 < \epsilon_{0} \le eig_1({\Omega}_0^n) \le eig_{p_n}({\Omega}_0^n) \le \epsilon_{0}^{-1}$.
\end{assumption}
\noindent
This assumption ensures that the eigenvalues of the true precision matrices are bounded by fixed constants, which has been commonly used for establish high dimensional covariance asymptotic properties. See for example \citep*{Bickel:Levina:2008, ElKaroui:2008, Banerjee:Ghosal:2014, XKG:2015, Banerjee:Ghosal:2015}. Previous work \citep*{CKG:2017} relaxes this assumption by allowing the lower and upper bounds on the eigenvalues to depend on $p$ and $n$. 
\begin{assumption} \label{assumption:d_n}
	$d_n\sqrt{\frac{\log p_n}{n}}\rightarrow 0$  as $n \rightarrow \infty$.
\end{assumption}
\noindent
This is a much weaker assumption for high dimensional covariance asymptotic than for example, \citep*{XKG:2015, Banerjee:Ghosal:2014, Banerjee:Ghosal:2015, CKG:2017}. Here we essentially allow the number of variables $p_n$ to grow slower than $e^{n/d_n^2}$ compared to previous literatures with rate $e^{n/d_n^4}$. 
\begin{assumption} \label{assumption:s_n}
	$\frac{d_n\log p_n}{s_n^2n} \rightarrow 0$ as $n \rightarrow \infty$.
\end{assumption}
\noindent
Recall that $s_n$ is the smallest (in absolute value) non-zero off-diagonal entry in 
$L_0^n$. Hence, this assumption also known as the ``beta-min" condition also provides 
a lower bound for the ``slab" part of $L_0^n$ that is needed for establishing consistency. This type of condition has been used for the exact support recovery of the high-dimensional linear regression models as well as Gaussian DAG models. See for example \citep*{LLL:2018, yang:2016, KORR:2017, CKG:2017, Yu:Bien:2016}.
\begin{remark} \label{remark:LLL}
	It is worthwhile to point out that our assumptions on the true Cholesky factor are weaker compared to \citep*{ LLL:2018}. In particular, \citet*{LLL:2018} introduce conditions A(2) and A(4) on the sparsity pattern of the true Cholesky factor such that the number of non-zero elements in each row as well as each column of $L_{0}^n$ to be smaller than some constant $s_0$, while in this paper, we are allowing the maximum number of non-zero entries in any column of $L_0^n$ to grow at a smaller rate than $\sqrt{\frac{n}{\log p_n}}$ (Assumption \ref{assumption:d_n}).
\end{remark}
\subsubsection{Assumptions on the Prior Hyperparameters} \label{sec:assumption_hyper}
\begin{assumption} \label{assumption:Z}
	$\pi(Z) = 0$ for all $Z$ satisfying $\max_{1 \le j \le p-1}|Z_j| \ge R_{n}$, where $R_{n} \sim n\left(\log n\right)^{-1}$.
\end{assumption}
\noindent
This assumption essentially states that the prior on the space of the $2^{\binom{p_n}{2}}$ possible models, places zero mass on unrealistically large models (see similar assumptions in \citep*{Johnson:Rossell:2012,Shin.M:2015,Narisetty:He:2014} in the context of regression). Assumption \ref{assumption:Z} is also more relaxed compared with Condition (P) in \citep*{LLL:2018} where $R_n \sim n(\log p)^{-1}\{(\log n)^{-1}\vee c_3\}$ for some constant $c_3$. Note that this condition is for the hyperparameter of the prior distribution on the latent variables only, which does not affect the true parameter space.
\begin{assumption} \label{assumption:tau}
	The hyperparameter $\tau_{n,p_n}$ in (\ref{model5}) satisfies $\frac{d_n}{\tau_{n,p_n}^2\log p_n} \rightarrow 0$ and $\frac{\sqrt{\frac n {\tau_{n,p_n}^2}}}{p_n^{\frac {(1 - 1/\kappa)c}2 }\log n} \rightarrow 0$, as $n \rightarrow \infty$, for some constant $\kappa > 1$.
\end{assumption}
\noindent
This assumption provides the rate at which the variance of the slab prior is required to grow to guarantee desired model selection consistency. Similar conditions on the hyperparameter can be seen in \citep*{Narisetty:He:2014, Shin.M:2015, Johnson:Rossell:2012}.
\begin{assumption} \label{assumption:alpha_2_multi}
	There exists a constant $c > 0$, such that the hyperparameters in model (\ref{model4}) satisfy $0 \le \lambda_{1n}, \lambda_{2n} < c$ and the shape parameters in model (\ref{model6}) satisfies $0 < \alpha_{1n} < c$, $\alpha_{2} \sim \max\left\{p_n^c, d_n^{\frac{2c}{c-2}}\right\}$, for $c > 2\kappa$, for some $\kappa > 1$.
\end{assumption}
\noindent
This assumption provides the rate at which the shape parameter needs to grow to ensure desired consistency. Previous literature with Erdos-Renyi priors puts restrictions on the rate of the edge probability. In particular,  previous work \citep*{CKG:2017} assumes $q = e^{-\eta_nn}$, where $\eta_n = d_n (\frac{\log p_n}{n})^{\frac{1/2}{1+k/2}}$ for some $k > 0$ to penalize large models. Similar assumptions on the hyperparameters can be also found in \citep*{yang:2016,Narisetty:He:2014} under regression setting. In Section \ref{sec:simulation:model:selection}, we will see the proposed model without specifying particular values for $q$ helps avoiding the potential computation limitation such as simulation results always favor the most sparse model.

For the rest of this paper, $p_n$, ${\Omega}_0^n$, $\Sigma_0^n$,$ L_0^n, D_0^n, 
Z_0^n, Z^n, d_n, \tau_n, s_n$, $\alpha_{1n}$, $\alpha_{2n}$ will be denoted as $p$, $ {\Omega}_0$, $\Sigma_0$, $L_0$, $D_0$, $Z_0, Z, d, \tau, s$, $\alpha_1, \alpha_2$ by leaving out the superscript for notational convenience. 
\subsection{Posterior Ratio Consistency}
We now state and prove the main model selection consistency results. The proofs for all the theorems will be provided in Section \ref{sec:proof_thm4} and Section \ref{sec:proof_thm5}. Our first result 
establishes what we refer to as ``posterior ratio consistency"  (following the terminology in \citep*{CKG:2017}). This notion of consistency implies that the true model will be the mode of the posterior distribution among all the 
models with probability tending to $1$ as $n \rightarrow \infty.$
\begin{theorem} \label{thm4}
	Under Assumptions 1-6, the following holds: 
	$$
	\max_{Z \ne Z_0} \frac{\pi(Z|\bm{Y})}
	{\pi(Z_0|\bm{Y})} \stackrel{\bar{P}}{\rightarrow} 0, \mbox{ as } n \rightarrow \infty. 
	$$
\end{theorem}
\noindent
\subsection{Model Selection Consistency for Posterior Mode}
If one was interested in a point estimate of $Z$ which reflects the sparsity pattern of $L_0$, the most apparent choice would be the posterior mode defined as\begin{equation} \label{a4}
\hat{Z} =  \argmax_{Z} \pi(Z|\bm{Y}).
\end{equation} From a frequentist point of view, it would be natural to obtain if we have 
model selection consistency for the posterior mode, which follows immediately from posterior ratio 
consistency established in Theorem \ref{thm4}, by noting that 
$
\max_{Z \ne Z_0} \frac{\pi(Z|\bm{Y})}
{\pi(Z_0|\bm{Y})} < 1 \Rightarrow \hat{Z}= Z_0. 
$ Therefore, we have the following corollary. 
\begin{cor} \label{cor1}
	Under Assumptions 1-6, the posterior mode $\hat{Z}$ is equal to the 
	true model $Z_0$ with probability tending to $1$, i.e., 
	$$
	\bar{P}(\hat{Z} = Z_0) \rightarrow 1, \mbox{ as } n \rightarrow \infty. 
	$$ 
\end{cor}

\begin{remark} \label{spikeslabregression}
	In the context of linear regression, \citet{Xu:Ghosh:2015} tackle the Bayesian group lasso problem. In particular, the authors propose the following hierarchical Bayesian model:
	\begin{eqnarray*}
		& & \bm Y \mid X, \bm{\beta}, \sigma^2 \sim N(X \bm{\beta}, \sigma^2 I)\\
		& & \bm \beta_g  \overset{ind}{\sim} \sigma^2, \tau_g^2 \sim (1-\pi_0)N(0, \sigma^2 \tau_{g}^2I) + \pi_0 \delta_0(\bm \beta_g), \quad g = 1,2,\ldots, G,\\ 
		& & \tau_g^2 \overset{ind}{\sim} \mbox{Gamma }(\frac{m_g + 1}{2}, \frac{\lambda^2}{2}), \quad g = 1,2,\ldots, G,\\
		& & \sigma^2 \overset{ind}{\sim} \mbox{Inverse-Gamma }(\alpha_1, \alpha_2). 
	\end{eqnarray*}
	\noindent
	In particular, they impose an independent spike and slab type prior on each factor $\bm \beta_g$ (conditional on the variance parameter $\sigma^2$), and an inverse Gamma prior on the variance. Each regression factor is explicitly present in the model with a probability $\pi_0$. In 
	this setting under an orthogonal design, the authors in \citep*{Xu:Ghosh:2015} establish oracle property and variable selection consistency for the median thresholding estimator of the regression coefficients on the group level. Note that with parent ordering, the off-diagonal entries in the $i^{th}$ 
	column of $L$ can be interpreted as the linear regression coefficients corresponding to fitting the $i^{th}$ 
	variable against all variables with label greater than $i$. Hence, there are similarities with respect to the model and consistency results between \citep{Xu:Ghosh:2015} and this work. However, despite these similarities, fundamental differences exist in these models and the corresponding analysis. Firstly, the number of groups (or factors) is considered to be fixed in \citep{Xu:Ghosh:2015}, while we allow the number of predictors to grow at an exponential rate of $n$ in a ultra high-dimensional setting, which creates more theoretically challenges. Secondly, the `design' matrices corresponding to the regression coefficients in each column of $L$ which can be represented as functions of the sample covariance matrix $S$ are random and correlated with each other, while \citep{Xu:Ghosh:2015} only considers the orthogonal design where $X^TX = I$ with no correlation introduced. Thirdly, the consistency result in \citep{Xu:Ghosh:2015} focuses only on group level selection only and is tailored for problems that only require group level sparsity,  while our model can induce sparsity in each individual element of $L$. The authors also propose a Bayesian hierarchical model referred to as Bayesian sparse group lasso to enable shrinkage both at the group level and within a group. However, no consistency results are addressed regarding this model. Lastly, in our model, each coefficient is present independently with multiplicative prior that incorporates information that $L$ is sparse, which is not the case in \citep{Xu:Ghosh:2015} as each factor is present with $\pi_0 = 0.5$. In particular, all the aspects discussed above lead to major differences and further 
	challenges in analyzing the ratio of posterior probabilities.
\end{remark}
\noindent

\subsection{Strong Model Selection Consistency}
\noindent
Next we establish another stronger result (compared to Theorem \ref{thm4}) which implies that the posterior mass assigned to 
the true model $Z_0$ converges to 1 in probability (under the true model). Following \citep*{Narisetty:He:2014, CKG:2017}, we refer to this notion of consistency as strong selection consistency.
\begin{theorem}\label{thm5}
	Under Assumptions 1-6, the following holds:
	$$
	\pi(Z_0 | \bm{Y}) \stackrel{\bar{P}}{\rightarrow} 1, \mbox{ as } n \rightarrow 
	\infty. 
	$$
\end{theorem}

\begin{remark} \label{edgerestrictassumption:multi}
	We would like to point out that our posterior ratio consistency and strong model selection consistency do not require any additional assumptions on bounding the maximum number of edges. In particular, \citet*{CKG:2017} consider only the DAGs with the total number of edges at most $ \frac{1}{8} d \left( \frac{n}{\log p}
	\right)^{\frac{1+k}{2+k}}$ for $k > 0$. By the assumptions in the previous work, it follows that the DAGs in the analysis do not include the models where the Cholesky factor has one or more non-zero elements for each column, since $p/ \frac{1}{8} d \left( \frac{n}{\log p}
	\right)^{\frac{1+k}{2+k}} \rightarrow \infty$, as $n \rightarrow \infty$, while in our result, each row can have at most $R_{n} \sim \frac n {\log n}$ number of non-zero entries as indicated in Assumption \ref{assumption:Z}. Hence, our strong model selection consistency results is more general than \citep*{CKG:2017, LLL:2018} in the sense that the consistency holds for a larger class of DAGs.
\end{remark}
\section{Results for Beta-mixture Prior} \label{sec:beta_mixture}
Though the multiplicative prior could allow variation among the indicator probabilities, the intractable marginal posteriors remain problematic in practice. The authors in \citep*{Tan:2017} address this issue via Laplace approximation. However, the computational cost for that will become extensive as $p$ increases. To obtain the marginal posterior probabilities in closed form and for ease of computation, we consider the following beta-mixture prior over the space of $Z$ introduced in \citep*{Carvalho:Scott:2009},
\begin{eqnarray}
& & Z_{kj} \mid q \overset{i.i.d}{\sim} \mbox{Bern} (q), \quad 1 \le j < k \le p, \label{model9} \\ 
& & q \sim \mbox{Beta}(\alpha_1,\alpha_2), \label{model10}
\end{eqnarray}
 where $\mbox{Bern} (q)$ denotes the Bernoulli distribution with probability $q$, and $\mbox{Beta}(\alpha_1,\alpha_2)$ represents the beta distribution with shape parameters $\alpha_1, \alpha_2$. We refer to model (\ref{model9}) and (\ref{model10}) as the beta-mixture prior over the space of latent variables indicating the sparsity structure for the Cholesky factor. 
 \begin{remark} \label{remark:beta_mixture}
 	\citet*{CKG:2017, Banerjee:Ghosal:2015} introduce an Erdos-Renyi type of 
 	distribution on the space of DAGs as the prior distribution for DAGs, where each directed edge is present with 
 	probability $q$ independently of the other edges. In particular, they define $\gamma_{ij} = \mathbb{I} \{(i,j) \in 
 	E(\mathscr D)\}$, $1 \le i < j \le p$ to be the edge indicator and let $\gamma_{ij}$, $1 \le 
 	i < j < p$ be independent identically distributed Bernoulli($q$) random variables. \citet*{CKG:2017} establish the DAG selection consistency under suitable assumptions. while \cite{Banerjee:Ghosal:2015} address the estimation
 	consistency, and provide high-dimensional Laplace approximations for the marginal 
 	posterior probabilities for the graphs. In our framework, we extend the previous work by putting a beta distribution on the edge probability $q$. The beta-mixture type of priors have previously been placed on graphs for simulation purpose in \citep{Carvalho:Scott:2009}, but the theoretical properties have yet to be investigated. A clear advantage of such an approach as indicated in \citep{Carvalho:Scott:2009} is that treating the previous fixed tuning constant $q$ as a model parameter shrinks the graph size to a data-determined value of $q$, and allows strong control over the number of spurious edges. 
 \end{remark}
\noindent
In order to obtain the posterior consistency for $Z$, we need the following lemma, which specifies the closed form for the marginal posterior density of $Z$ with proof provided in Section \ref{sec:proof_thm13}.
\begin{lemma} \label{newlemma}
	The marginal posterior density $\pi(Z|Y)$ under the beta-mixture prior satisfies
	\begin{align} \label{posterior_propto}
	&\pi(Z | \bm Y) \nonumber\\
	\propto& B\left(\alpha_1(p-1)+ \sum_{j = 1}^{p-1}|Z_j|, \alpha_2(p-1) + \frac{p(p-1)}{2}-
	\sum_{j = 1}^{p-1}|Z_j|\right) \nonumber\\
	&\times \prod_{j = 1}^{p-1}  \left(\frac{n\tilde S_{j|Z_j}}{2}- \frac 1 {2\tau^2} + \lambda_2\right)^{-\frac n 2 - \lambda_1}  \frac{|\tilde S_Z^{>j}|^{-\frac 1 2}}{(n\tau^2)^{|Z_j|/2}}\nonumber\\ 
	= & B\left(\alpha_1(p-1)+ \sum_{j = 1}^{p-1}|Z_j|, \alpha_2(p-1) + \frac{p(p-1)}{2}-
	\sum_{j = 1}^{p}|Z_j|\right) \nonumber\\
	&\times \prod_{j = 1}^{p-1} \left(\frac{n\tilde S_{j|Z_j}}{2} - \frac 1 {2\tau^2} + \lambda_2\right)^{-\frac n 2 - \lambda_1} \frac{\left(|\tilde S_{Z}^{\ge i}|\tilde{S}_{j|Z_j}\right)^{-\frac 1 2}}{(n\tau^2)^{|Z_j|/2}},
	\end{align}
	in which $\tilde{S} = S+\frac 1 {n\tau^2} I_p$, $\tilde{S}_{j|Z_j} = \tilde{S}_{jj} - (\tilde{S}_{Z \cdot j}^>)^T 
	(\tilde{S}_Z^{>j})^{-1} \tilde{S}_{Z \cdot j}^>$ and $B(a,b) = \frac{\Gamma(a)\Gamma(b)}{\Gamma(a+b)}.$ The second equation follows from $|\tilde{S}_Z^{>j}| = |\tilde S_{Z}^{\ge i}|\left(\tilde{S}_{jj} - (\tilde{S}_{Z \cdot j}^>)^T 
	(\tilde{S}_Z^{>j})^{-1} \tilde{S}_{Z \cdot j}^>\right) = |\tilde S_{Z}^{\ge i}|\tilde{S}_{j|Z_j}.$
\end{lemma}
\noindent
In particular, these posterior 
probabilities can be used to select a model representing the sparsity pattern of $L$ by computing the posterior mode that maximize the posterior densities.
The convenient closed form for the marginal posterior in (\ref{posterior_propto}) also yields nice posterior ratio consistency under the following weaker assumption on $\alpha_{2}$ compared with Assumption \ref{assumption:alpha_2_multi}.
\begin{assumption}  \label{assumption:alpha_2}
	There exists a constant $c > 0$, such that the hyperparameters in model (\ref{model4}) satisfy $0 \le \lambda_{1n}, \lambda_{2n} < c$ and the shape parameters in model (\ref{model6}) satisfies $0 < \alpha_{1n} < c$, $\alpha_{2n} \sim p^c$.
\end{assumption}
\begin{theorem} \label{thm1}
	Under Assumptions 1-5 and \ref{assumption:alpha_2}, the following holds under the beta-mixture prior: 
	$$
	\max_{Z \ne Z_0} \frac{\pi(Z|\bm{Y})}
	{\pi(Z_0|\bm{Y})} \stackrel{\bar{P}}{\rightarrow} 0, \mbox{ as } n \rightarrow \infty. 
	$$
\end{theorem}
\noindent
The next theorem establishes the strong selection consistency under the beta-mixture prior. See proofs for Theorem \ref{thm1} and Theorem \ref{thm3} in Section \ref{sec:proof_thm13}.
\begin{theorem} \label{thm3}
	Under Assumptions 1-6, for the beta-mixture prior, the following holds:
	$$
	\pi(Z_0 | \bm{Y}) \stackrel{\bar{P}}{\rightarrow} 1, \mbox{ as } n \rightarrow 
	\infty. 
	$$
\end{theorem}
\noindent
\begin{remark} \label{edgerestrictassumption}
	We would like to point out that posterior ratio consistency 
	(Theorem \ref{thm1} does not require any restriction on $c$ (the rate of the shape parameter in the beta distribution (\ref{model10})) that will be growing, this requirement is only needed for strong selection consistency 
	(Theorem \ref{thm3}). Similar restrictions on the hyperparameters have been 
	considered for establishing consistency properties in the regression setup. See \citep*{yang:2016, LLL:2018, CKG:nonlocal} for example. 
\end{remark}
\noindent
The closed form for the marginal posterior probability in (\ref{posterior_propto}) is convenient for showing the consistency. However, when it comes to simulation, the beta term in (\ref{posterior_propto}) pertaining to the beta-mixture prior is often too large, and could sometimes blow up when $p$ is relatively large. In addition, for the beta-mixture prior, probability $q$ is assumed to be universal across all indicators, which seems not flexible and diverse enough. In the following section, we will take on the task to investigate and evaluate the simulation performance for both the multiplicative model and the beta-mixture model.
\section{Simulation Studies} \label{sec:experiments}
In this section, we demonstrate our main results through simulation studies. First recall from (\ref{posterior_propto}) that the marginal posterior distributions for $Z$ under the beta-mixture prior can be derived analytically in closed form (up to a constant)  in (\ref{posterior_propto}). Therefore, we can evaluate the parameter space more clearly with this naturally assigned ``score", that is the posterior probability. 

For the multiplicative prior, the $\omega_j$ ($1 \le j \le p$) can not be integrated out, thus the closed form for the marginal distribution of $Z$ can not be conveniently acquired. As indicated in \citep*{Tan:2017}, evaluating the marginal densities via Monte Carlo becomes more computationally intensive as the dimension increases. Therefore, the authors propose to estimate these quantities efficiently through Laplace approximation instead. Detailed functional and Hessian expressions can be found in the supplemental material in \citep*{Tan:2017}. Here we adopt the same Laplace approximation for estimating the marginal densities for $Z$. However, as we will see in Figure \ref{run_time}, though the multiplicative prior could potentially lead to better model selection performance, the additional procedure when evaluating each individual posterior probability could be quite time consuming. In particular, the Newton-type algorithm used for obtaining the mode of the log-likelihood runs extremely slow in higher dimensions. 

\subsection{Simulation I: Illustration of Posterior Ratio Consistency} \label{sec:illustration:ratio:consistency}
\noindent
In this section, we illustrate the consistency result in Theorem \ref{thm4} and Theorem \ref{thm1} using a simulation experiment. 
Our goal is to show that the log of the posterior ratio for any ``non-true" model compared to the true model will converge to negative infinity. 
To serve this purpose, we consider $10$ different values of $p$ ranging from 
$150$ to $1500$, and choose $n = p/3$. Next, for each fixed $p$, a $p \times p$ 
lower triangular matrix with diagonal entries $1$ and off-diagonal entries $0.5$ is constructed. In particular, unlike in previous work \citep{CKG:2017} where the expected value of non-zero entries in each column of $L_0$ does not exceed 3, here we randomly chose 3\% or 5\% of the lower triangular entries of the Cholesky factor
and set them to be 0.5. The remaining entries were set to zero. 

The purpose of this setting is to show our consistency requires more relaxed sparsity assumptions on the true model compared to \citep*{CKG:2017}. We refer to this matrix as 
$L_0$. The matrix $L_0$ also reflects the true underlying DAG structure encoded in $Z_0$. Next, we generate 
$n$ i.i.d. observations from the $N(0_p, (L_0^{-1})^T L_0^{-1})$ distribution, and set the 
hyperparameters as $c = 2$, $\tau_{n,p} = \sqrt{n}$, $\lambda_1 = \lambda_2 = 0.05$, $\alpha_1 = 0.05$ for 
$i = 1,2,\ldots,p$. The above process ensures all the assumptions are satisfied. We then examine 
posterior ratio consistency under four different cases by computing the log posterior ratio of 
a ``non-true"model $Z$ and $Z_0$ as follows. 
\begin{enumerate}
	\item Case $1$: Model $Z$ is a submodel of $Z_0$ and the number of total non-zero entries of 
	$Z$ is exactly half of $Z_0$, i.e. $\sum Z = \frac 1 2 \sum Z_0$. 
	\item Case $2$: $Z_0$ is a submodel of $Z$ and the number of total non-zero entries of 
	$Z$ is exactly twice of $Z_0$, i.e. $\sum Z = 2 \sum Z_0$. 
	\item Case $3$: $Z$ is not necessarily a submodel of $Z_0$, but satisfying the number 
	of total non-zero entries in $Z$ is half the number of non-zero entries in $Z_0$. 
	\item Case $4$: $Z_0$ is not necessarily a submodel of $Z$, but the number 
	of total non-zero entries in $Z$ is twice the number of non-zero elements in $Z_0$. 
\end{enumerate}
\begin{figure}[htbp]
	\centering
	\begin{subfigure}
		{\includegraphics[width=51.5mm]{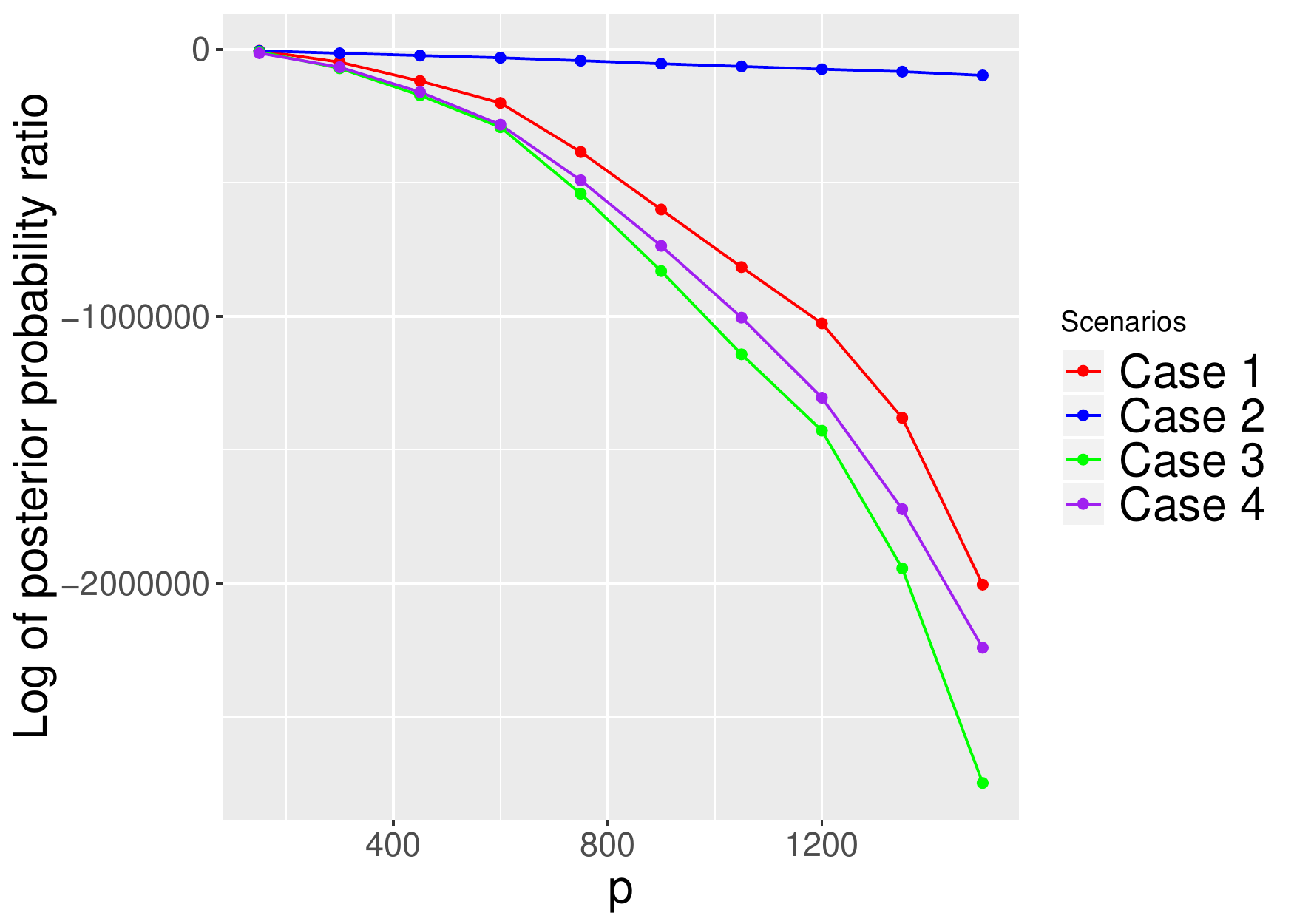}}
	\end{subfigure}
	\qquad
	\begin{subfigure}
		{\includegraphics[width=51.5mm]{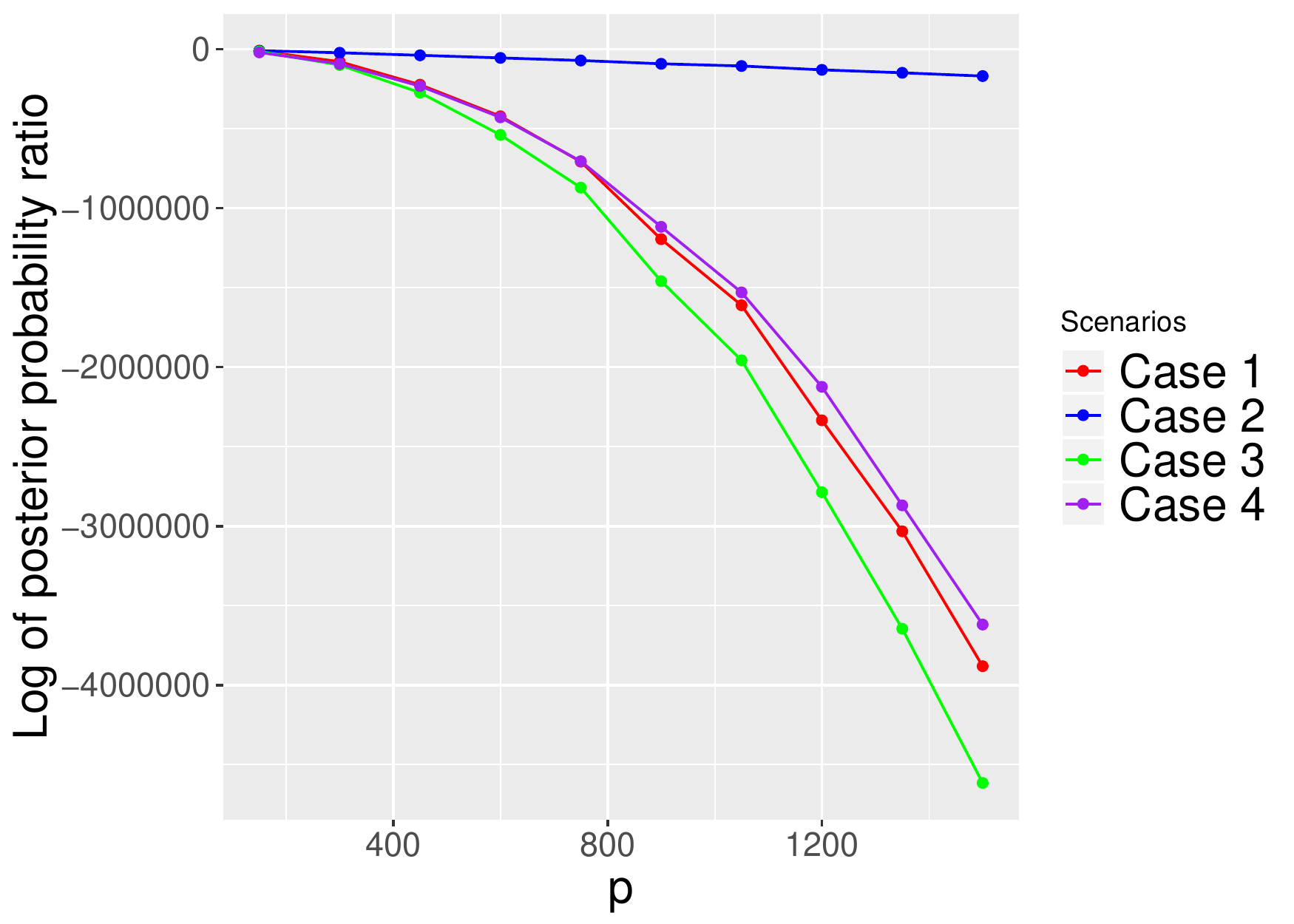}}
	\end{subfigure}

	\begin{subfigure}
		{\includegraphics[width=51.5mm]{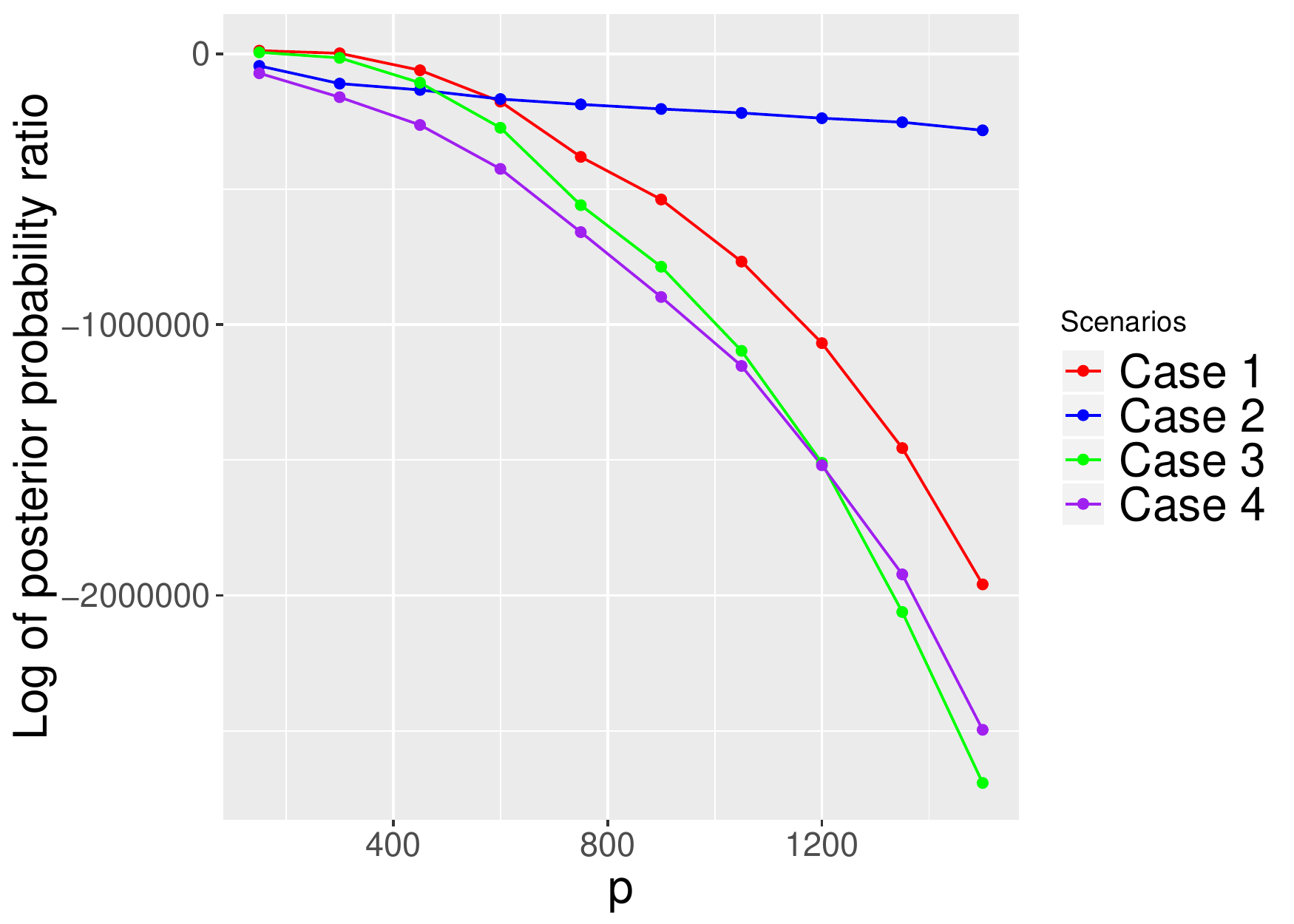}}
	\end{subfigure}
	\qquad
	\begin{subfigure}
		{\includegraphics[width=51.5mm]{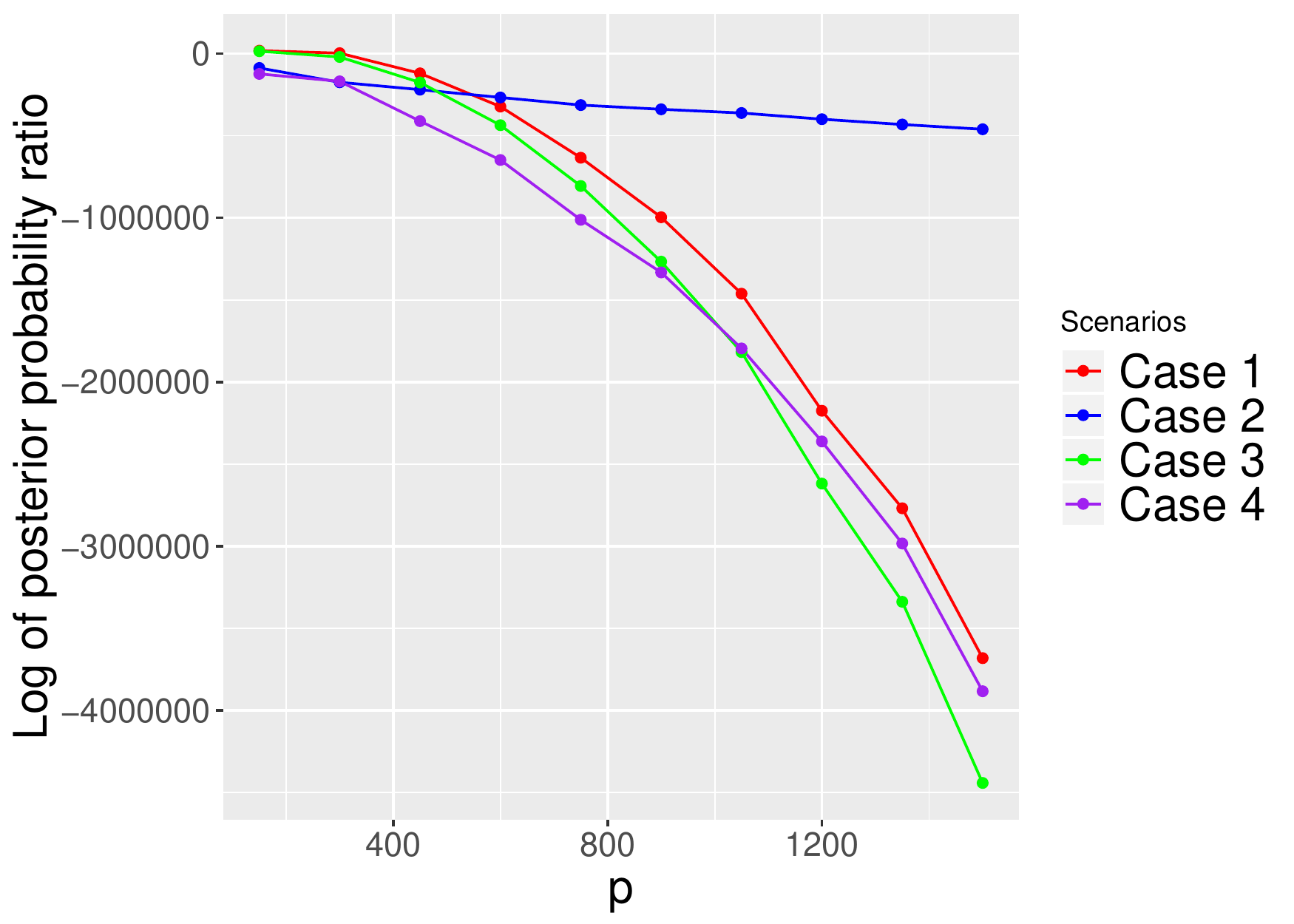}}
	\end{subfigure}
	\caption{Log of posterior probability ratio for $Z$ and $Z_0$ for various choices of the ``non-true" model $Z$. Here $Z_0$ denotes the true underlying model indicator. Left: $3\%$ sparsity; right: $5\%$ sparsity; top: beta-mixture prior; bottom: multiplicative prior.}
	\label{posterior_ratio_plot}
\end{figure}
\noindent
The log of the posterior probability ratio for various cases under two different sparsity settings and our two different priors are provided in Figure \ref{posterior_ratio_plot}. As expected the log of the posterior probability ratio decreases to large negative numbers as $n$ becomes large in all four cases and in both sparsity settings and under both sparsity priors, thereby providing a numerical illustration of 
Theorem \ref{thm4}. 

We would like to point out that in \citep*{CKG:2017}, the log of posterior ratios are almost all positive real numbers, when $p \le 1500$ and the expected value of non-zero entries in each column of $L_0$ does not exceed 3, which indicates the hierarchical model with DAG-Wishart distribution and the Erdos-Renyi type of prior over graphs only performs better with really higher dimension and much more sparse settings. In particular, this leads to one potential drawback of using the DAG-Wishart distribution coupled with the Erdos-Renyi type of prior on the Cholesky space, as in real applications, extremely high-dimensional and sparse data sets are not very commonly seen, while our spike and slab Cholesky prior with the beta-mixture or multiplicative prior is more adaptable and diverse in that aspect.
\subsection{Simulation II: Illustration of Model Selection} \label{sec:simulation:model:selection}
\noindent
In this section, we perform a simulation experiment to illustrate the potential advantages of using our Bayesian model
selection approach. We consider $5$ values of $p$ ranging from 
$300$ to $1500$, with $n = p/3$. For each fixed $p$, the Cholesky factor $L_0$ of the true concentration matrix, and the corresponding dataset, are generated by the same mechanism as in Section \ref{sec:illustration:ratio:consistency}. Then, we 
perform model selection on the Cholesky factor using the four procedures outlined below. 
\begin{enumerate}
	\item {\it Lasso-DAG with quantile based tuning}: We implement the Lasso-DAG approach in \citep{Shojaie:Michailidis:2010} by choosinf penalty parameters (separate for each variable $i$) given by $\lambda_i = 
	2 n^{-\frac 1 2}Z^*_{\frac{0.1}{2p(i-1)}}$, where $Z_q^*$ denotes the $(1-q)^{th}$ quantile of the standard normal distribution. 
	This choice is justified in \citep{Shojaie:Michailidis:2010} based on asymptotic considerations.
	
	\item {\it ESC Metropolis-Hastings algorithm}: We implement the Rao-Blackwellized Metropolis-Hastings algorithm for the ESC prior introduced in \citep*{LLL:2018} for exploring the space of the Cholesky factor. The hyperparameters and the initial states are taken as suggested in \citep{LLL:2018}. Each MCMC chain for each row of the Cholesky factor runs for 5000 iterations with a burn-in period of 2000. All the active components in $L$ with inclusion probability larger than 0.5 are selected. We would like to point out that since the Metropolis-Hastings algorithm needs to be executed for each row of $L$, the procedure could be extremely time consuming, especially in higher dimensions. 
	
	\item {\it DAG-Wishart log-score path search}: The hierarchical DAG-Wishart prior \citep{CKG:2017} also gives us the closed form to calculate the marginal posterior up to a constant. In particular,
	$$\pi({\mathscr{D}}|\bm{Y}) = \frac{\pi({\mathscr{D}})}{\pi(\bm{Y})(\sqrt{2\pi})^n} \frac{z_{\mathscr{D}}(U+nS,n+\bm{\alpha}(\mathscr D))}{z_{\mathscr{D}}(U,\bm{\alpha}(\mathscr D))},$$ where $z_{\mathscr{D}}(\cdot, \cdot)$ is the normalized constant in the DAG-Wishart distrbution and $$
	\pi (\mathscr D) = \prod_{(i,j):1 \leq i < j \leq p} q^{\gamma_{ij}} \left(1-q 
	\right)^{1-\gamma_{ij}} = \prod_{i=1}^{p-1} q^{\nu_i(\mathscr{D})} (1-q)^{p-i-
		\nu_i(\mathscr{D})}. 
	$$ with $q = e^{-\eta_nn}$, where $\eta_n = d_n (\frac{\log p_n}{n})^{\frac{1/2}{1+k/2}}$. Follow the simulation procedures in previous work \citep{CKG:2017}. We set the hyperparameters as $U = I_p$ and $\alpha_i(\mathscr D) = \nu_i(\mathscr D) + 10$ for 
	$i = 1,2,\ldots,p$ and generate candidate graphs by thresholding the modified Cholesky factor of $(S + 0.5 I)^{-1}$ ($S$ is the sample 
	covariance matrix) on a grid from 0.1 to 0.5 by 0.0001 to get a sequence of $4000$ graphs. The log posterior probabilities are computed for all candidate graphs, and the graph with the highest probability is chosen. As we discussed previously, we will see in Figure \ref{heat_2} that for the previous DAG-Wishart model, we always end up choosing the most sparse estimator, since the graph obtained at the thresholding value 0.5 always has the highest log posterior score. Hence, we observe that the choice $q = e^{-\eta_nn}$ though could guarantee the model selection consistency, makes the posterior stuck in very small size models and we are not able to detect the true model. 
	
	\item {\it Spike and slab Cholesky with beta-mixture prior/multiplicative prior}: For our Bayesian approach with spike and slab Cholesky prior and beta-mixture/multiplicative prior on the sparsity pattern of $L$, we adopt the similar procedure as DAG-Wishart log-score path search method. We construct two candidate sets as follows. 
	\begin{enumerate}
		\item All the Cholesky factors with respect to the graphs on the solution paths for Lasso-DAG, CSCS and DAG-Wishart are included in our Cholesky factor candidate set. 
		\item To increase the search 
		range, we also generate additional graphs by thresholding the modified Cholesky factor of $(S + 0.5 I)^{-1}$ ($S$ is the sample 
		covariance matrix) on a grid from 0.1 to 0.5 by 0.0001  to get a sequence of $4000$ additional Cholesky factors, and include them in the candidate set. We then search around 
		all the above candidates using Shotgun Stochastic Search Algorithm in \citep{Shin.M:2015} to generate even more candidate Cholesky factors. In particular, the authors in \citep{Shin.M:2015} claim that the simplified algorithm can significantly lessen the simulation runtime and increase the model selection performance. 
	\end{enumerate}
	The log posterior probabilities are computed for all Cholesky factors in the candidate sets using (\ref{posterior_propto}), and the one with the highest probability is 
	chosen. In Figure \ref{heat_2}, we plot the log of marginal posterior densities under the spike and slab Cholesky prior and the multiplicative/beta-mixture prior for all the Cholesky factors under different thresholding values compared with the marginal posteriors under previous DAG-Wishart model. Unlike the DAG-Wishart distribution always favor the most sparse Cholesky factor corresponding to the largest thresholding value, we observe the maximum log posterior score occurs in the middle of the curve for our proposed models, which leads to the significant improvement of the model selection results shown in Table \ref{model_selection_table1} and Table \ref{model_selection_table2}. 
\end{enumerate}
\begin{figure}[htbp]
	\centering
	\begin{subfigure}[$(n,p) = (100,300)$]
		{\label{fig:a}\includegraphics[width=35mm]{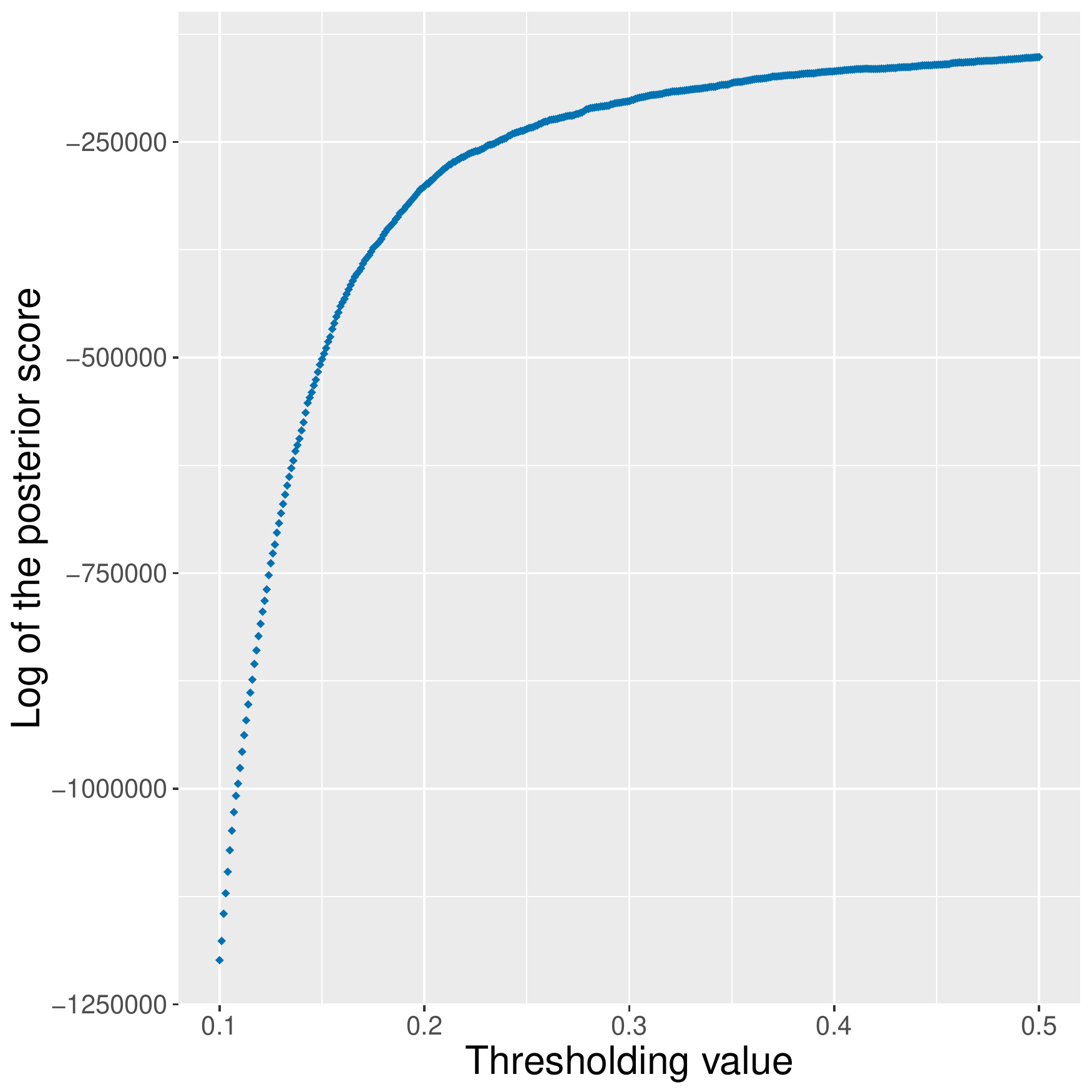}}
	\end{subfigure}
	\quad
	\begin{subfigure}[$(n,p) = (200,600)$]
		{\label{fig:b}\includegraphics[width=35mm]{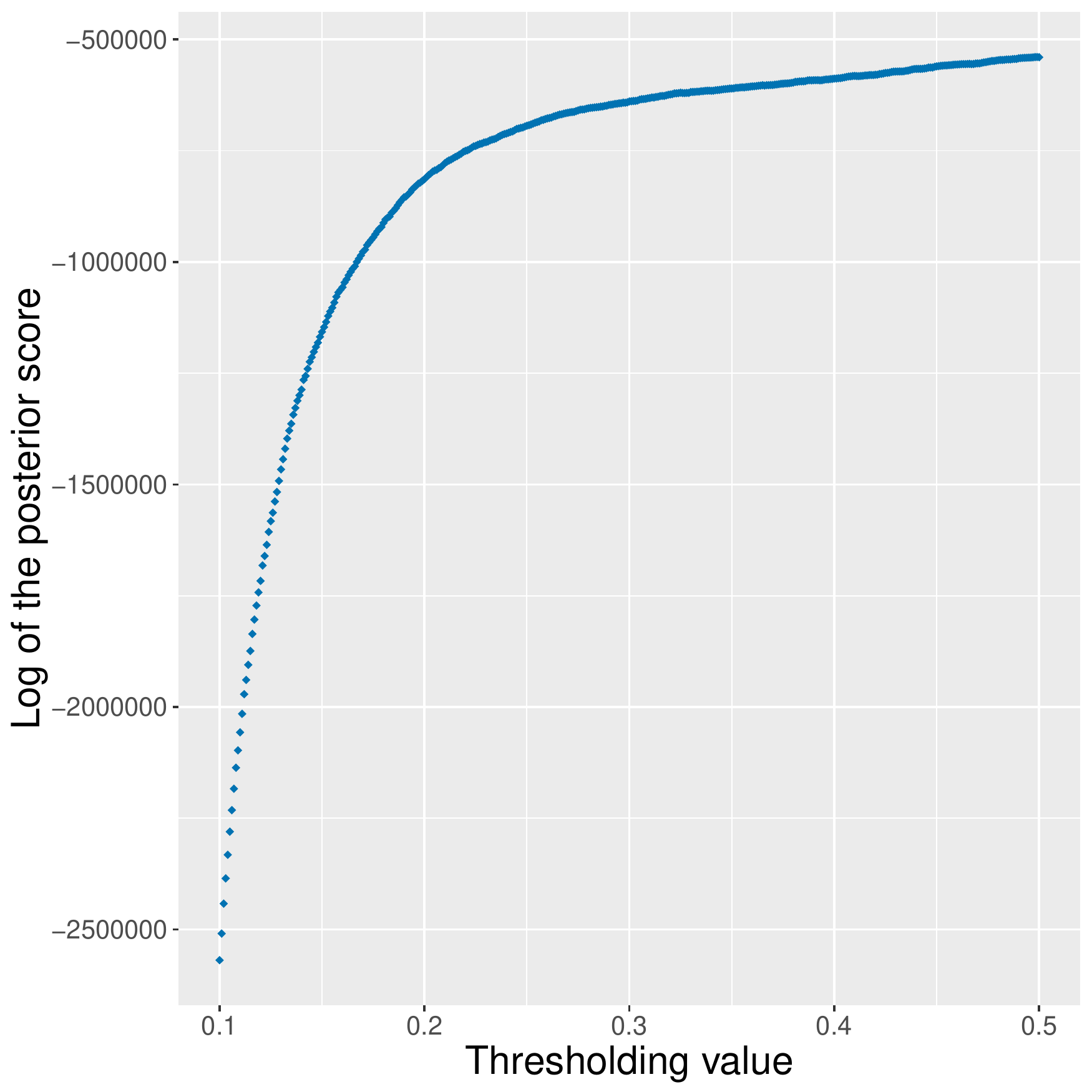}}
	\end{subfigure}
	\quad
	\begin{subfigure}[$(n,p) = (300,900)$]
		{\label{fig:c}\includegraphics[width=35mm]{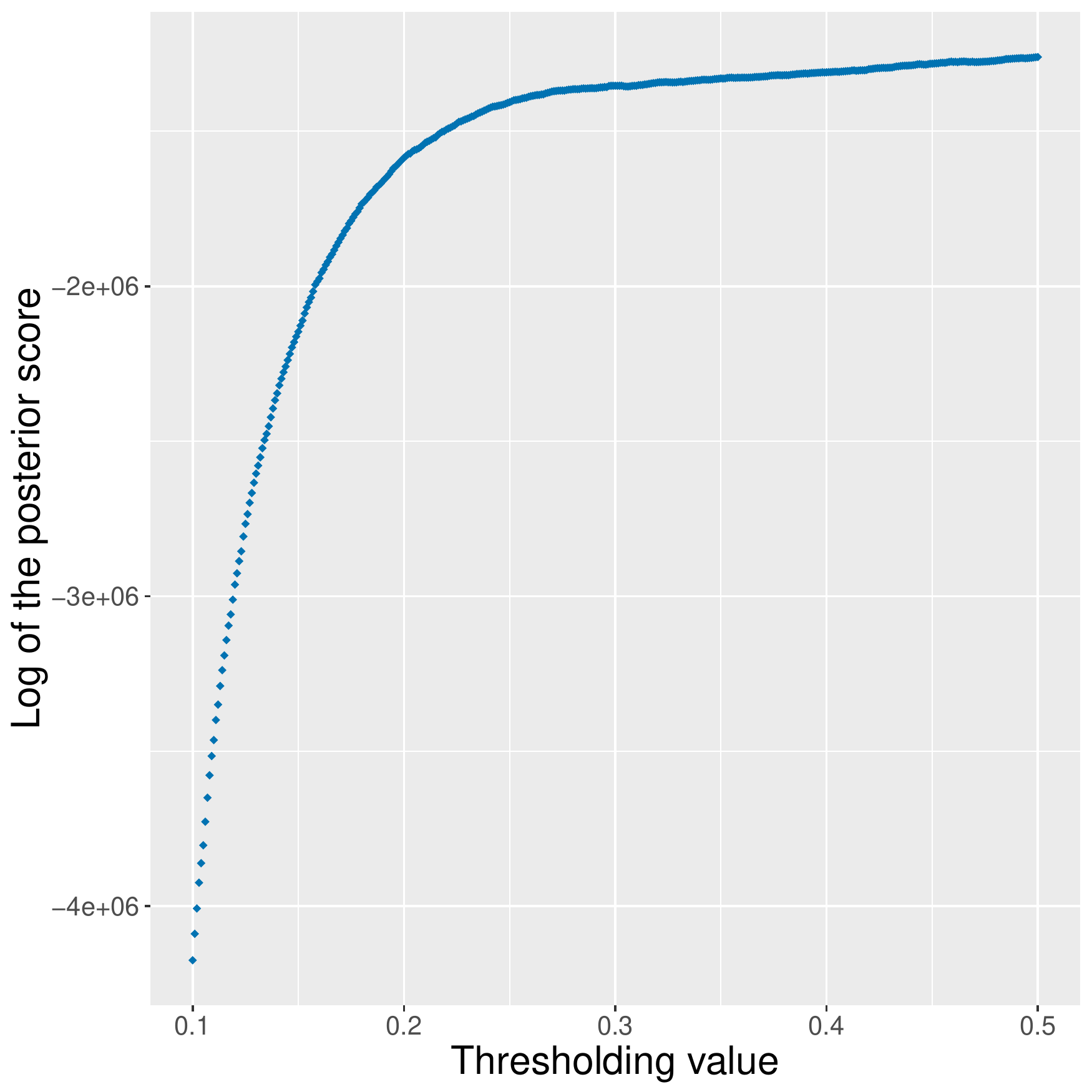}}
	\end{subfigure}%

	\begin{subfigure}[$(n,p) = (100,300)$]
		{\label{fig:d}\includegraphics[width=35mm]{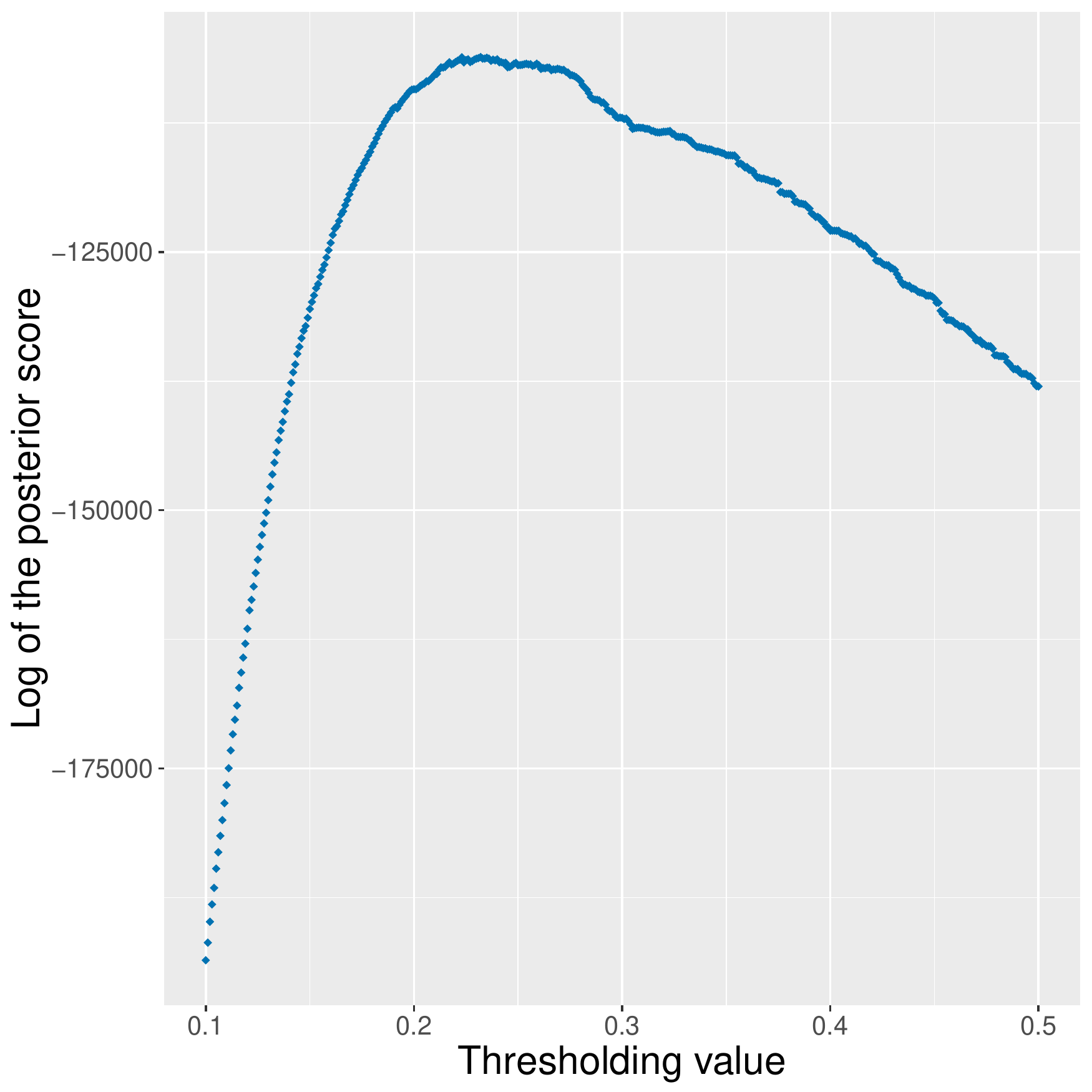}}
	\end{subfigure}
	\quad
	\begin{subfigure}[$(n,p) = (200,600)$]
		{\includegraphics[width=35mm]{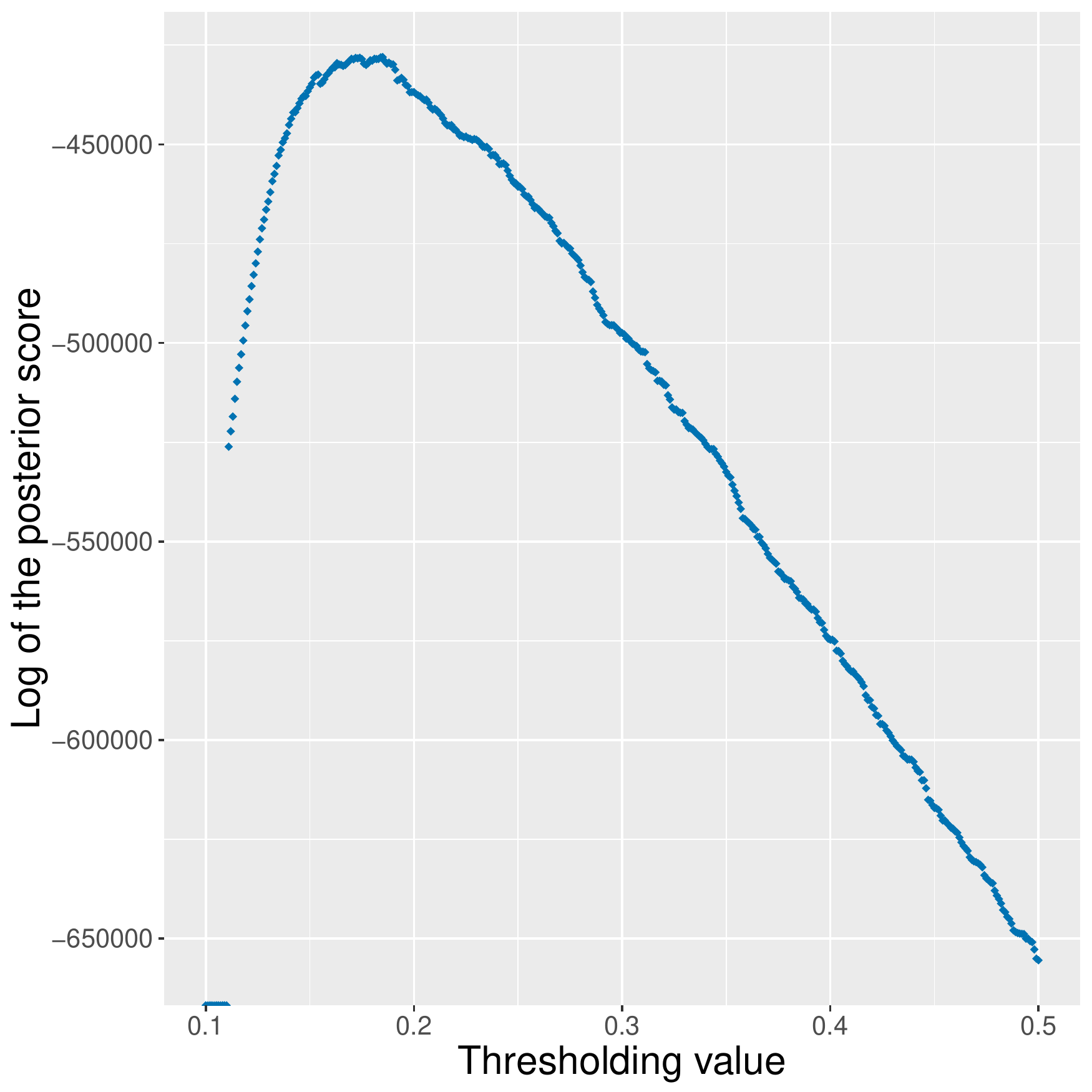}}
	\end{subfigure}%
	\quad
	\begin{subfigure}[$(n,p) = (300,900)$]
		{\includegraphics[width=35mm]{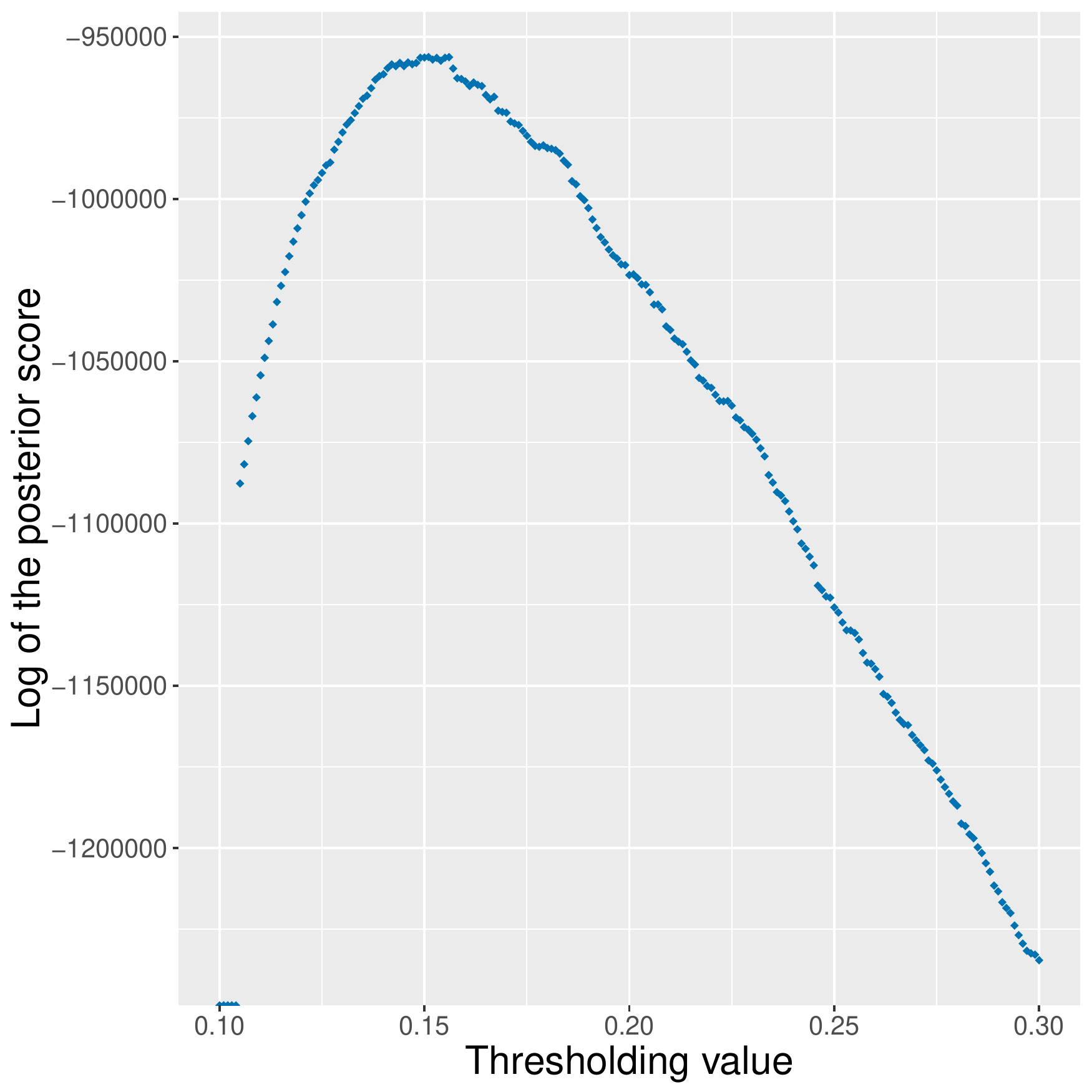}}
	\end{subfigure}%

	\begin{subfigure}[$(n,p) = (100,300)$]
	{\label{fig:d}\includegraphics[width=35mm]{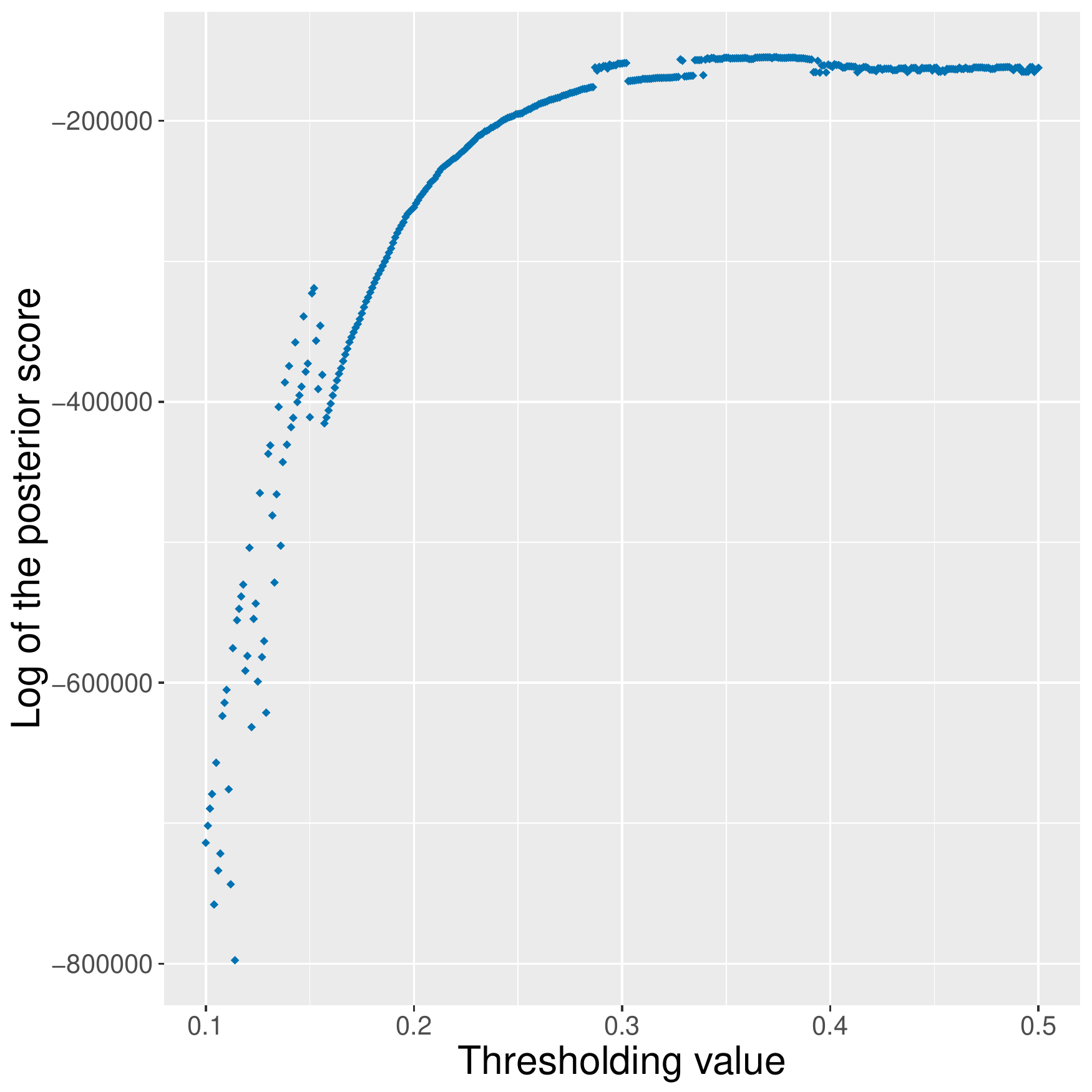}}
\end{subfigure}
\quad
\begin{subfigure}[$(n,p) = (200,600)$]
	{\includegraphics[width=35mm]{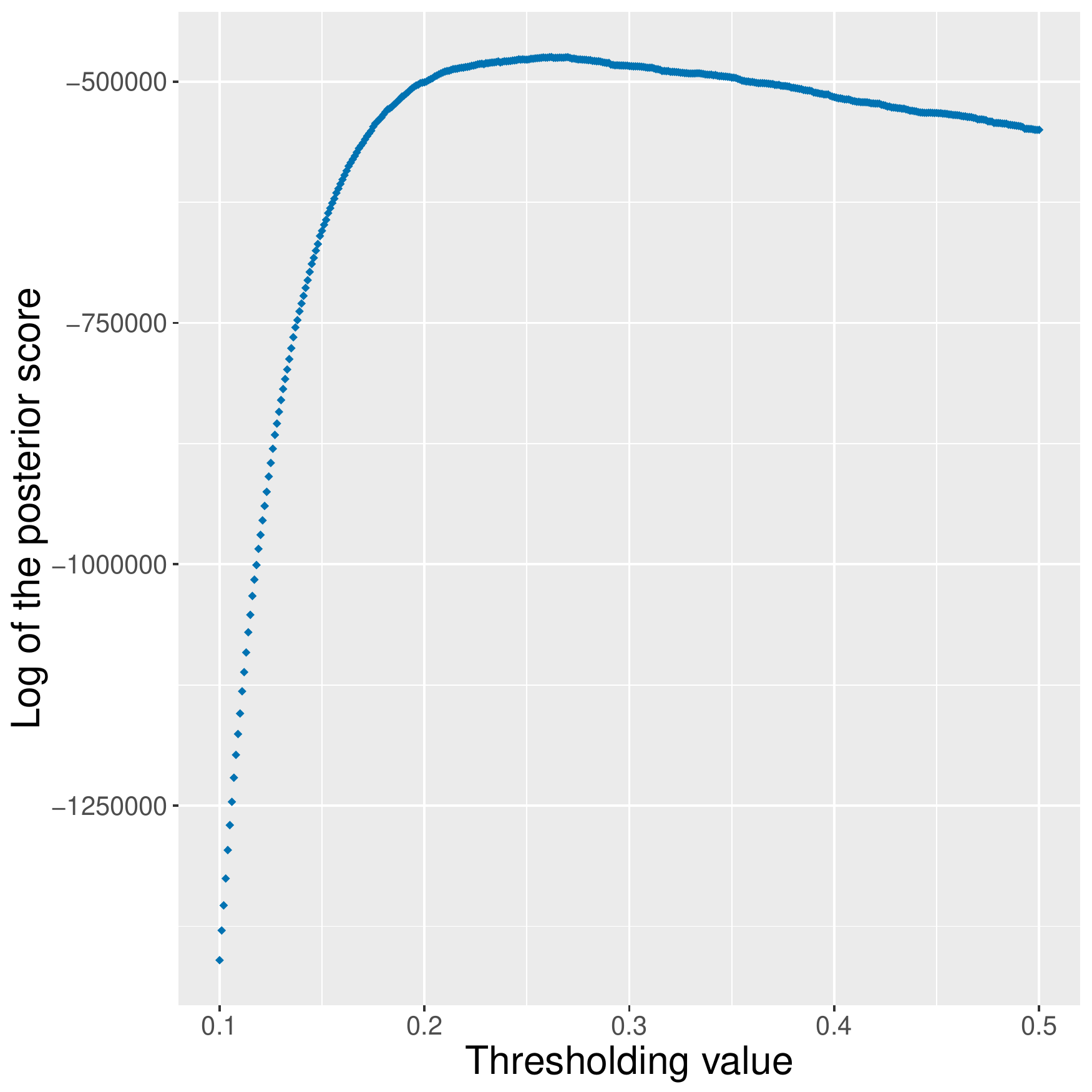}}
\end{subfigure}%
\quad
\begin{subfigure}[$(n,p) = (300,900)$]
	{\includegraphics[width=35mm]{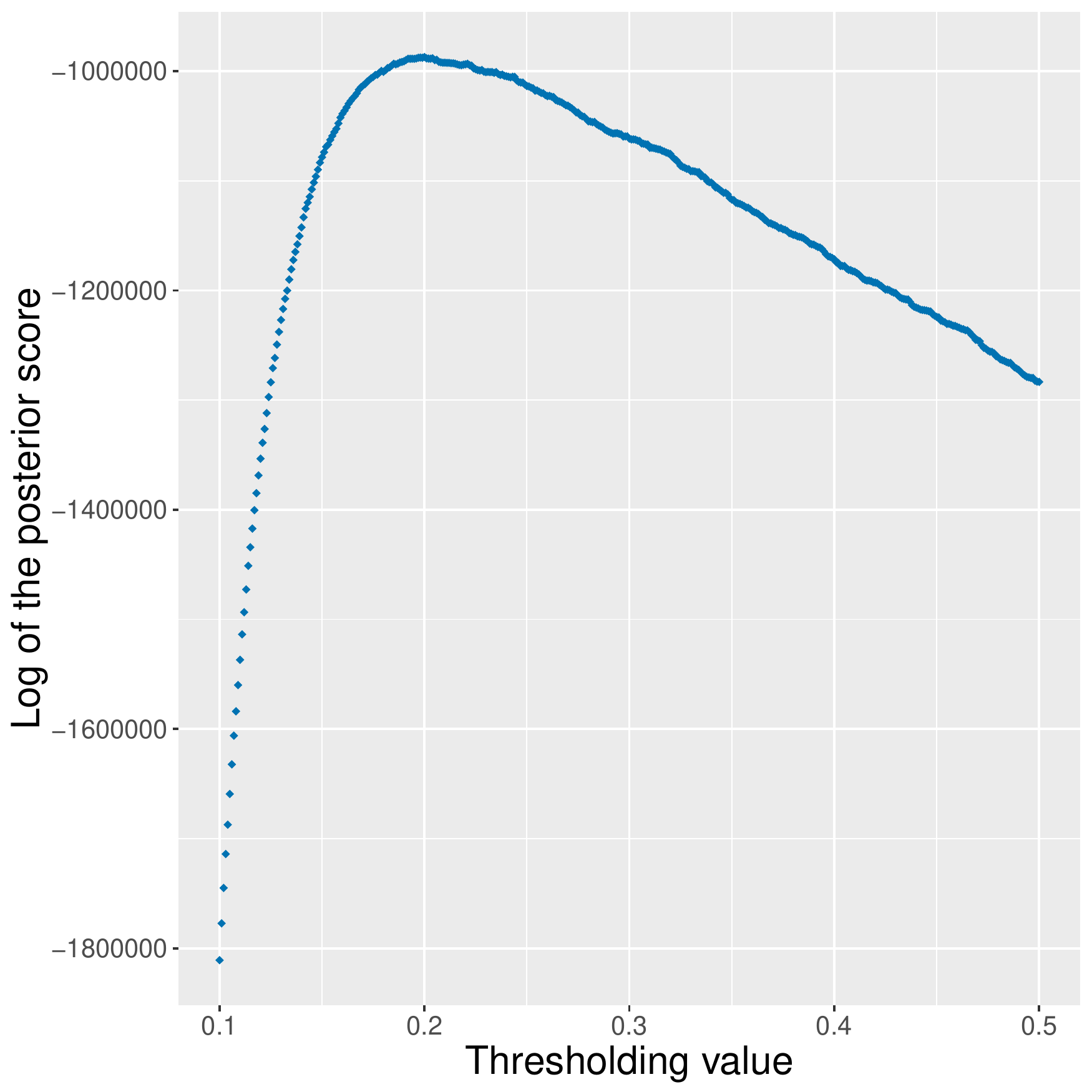}}
\end{subfigure}
	\caption{Log of posterior vs thresholding values under different priors. Top: DAG-Wishart; middle: Spike and slab Cholesky with beta-mixture prior; bottom: Spike and slab Cholesky with multiplicative prior.}
	\label{heat_2}
\end{figure}
\noindent
The model selection performance of these four methods is then compared using several different measures of structure such as 
positive predictive value, true positive rate and mathews correlation coefficient (average over $20$ independent repetitions). 
Positive Predictive Value (PPV) represents the proportion of true non-zero entries among all the entries detected by the given procedure, 
True Positive Rate (TPR) measures the proportion of true non-zero entries detected by the given procedure among all the non-zero entries 
from the true model. PPV and TPR are defined as $$\mbox{PPV} = \frac{\text{TP}}{\text{TP + FP}}, \quad \mbox{TPR} = \frac{\text{TP}}{\text{TP + FN}}.$$Mathews correlation Coefficient (MCC) is commonly used to assess the performance of binary classification methods and is defined as $$ \text{MCC} = \frac{\text{TP} \times \text{TN} - \text{FP} \times \text{FN}} {\sqrt{\text{(FP + TN)} \times \text{(TP+FN)} \times\text{(TN+FP)}\text{(TN+FN)}}},$$ where TP, TN, FP and FN correspond to true positive, true negative, false positive and false negative, respectively. Note that the value of MCC ranges from -1 to 1 with larger values corresponding to better fits (-1 and 1 represent worst and best fits, respectively). Similar to MCC, one would also like the PPV and TPR values to be as close to $1$ as possible. 
The results are provided in Table \ref{model_selection_table1} and Table \ref{model_selection_table2}, corresponding to different true sparsity levels. In Figure \ref{heat_3}, we draw the heatmap comparison between the true $L_0$ and estimated $L$ using our Bayesian spike and slab Cholesky approach under two different sparsity levels when $(n,p) = (100,300)$.
\begin{table}[htbp]
	\small
	\centering
	\scalebox{0.65}{
		\begin{tabular}{ccccccccccccccccc}
			\hline
			\multicolumn{2}{c}{\multirow{2}{*}{}}     & \multicolumn{3}{c}{Lasso-DAG}             &\multicolumn{3}{c}{ESC}             & \multicolumn{3}{c}{DAG-W}      &\multicolumn{3}{c}{SSC-B}         &\multicolumn{3}{c}{SSC-M}                \\
			$p$                   & $n$                & PPV           & TPR          & MCC          & PPV        & TPR        & MCC        & PPV       & TPR       & MCC     & PPV        & TPR        & MCC & PPV        & TPR        & MCC              \\ \hline
			300  & 100 & 0.2  & 0.2  & 0.19 &0.17  &0.43 &0.26  & 0.99 & 0.3  & 0.55 & 0.73 &0.85 &0.78  & 0.98 & 0.69 & 0.82 \\
			600  & 200 & 0.15 & 0.18 & 0.16 &0.15 &0.52 &0.27 & 0.99 & 0.31 & 0.55 &0.69 &0.92 &0.79 & 0.89 & 0.82  & 0.85  \\
			900  & 300 & 0.15 & 0.20  & 0.17 &0.12  &0.54  &0.24 & 1    & 0.33 & 0.57 &0.62 &0.93 &0.76 & 0.83 & 0.87 & 0.84 \\
			1200 & 400 & 0.11 & 0.17 & 0.14 &0.08  &0.52  &0.21  & 1    & 0.33 & 0.58 &0.61 &0.94 &0.76 & 0.78 & 0.90 & 0.84 \\
			1500 & 500 & 0.12 & 0.21 & 0.16 &0.06  &0.45 &0.20  & 1    & 0.33 & 0.58 &0.56 &0.96 &0.73 & 0.71 & 0.93 & 0.81 \\ \hline
		\end{tabular}
	}
	\caption{Model selection performance table with sparsity 3\%. DAG-W: DAG-Wishart log-score path search; SSC-B: Spike and slab Cholesky with beta-mixture prior; SSC-M: Spike and slab Cholesky with multiplicative  prior.}
	\label{model_selection_table1}
\end{table}

\begin{table}[htbp]
	\small
	\centering
	\scalebox{0.65}{
		\begin{tabular}{ccccccccccccccccc}
			\hline
		\multicolumn{2}{c}{\multirow{2}{*}{}}     & \multicolumn{3}{c}{Lasso-DAG}             &\multicolumn{3}{c}{ESC}             & \multicolumn{3}{c}{DAG-W}      &\multicolumn{3}{c}{SSC-B}         &\multicolumn{3}{c}{SSC-M}                \\
			$p$                   & $n$                 & PPV        & TPR        & MCC       & PPV           & TPR          & MCC          & PPV        & TPR        & MCC        & PPV       & TPR       & MCC       & PPV        & TPR        & MCC                   \\ \hline
			300 & 100 & 0.19 & 0.1  & 0.13 &0.14 &0.33 &0.19 & 0.99 & 0.3  & 0.54 &0.66 &0.81 &0.73 & 0.99 & 0.43 & 0.65 \\
			450 & 150 & 0.12 & 0.09 & 0.1  &0.11  &0.35 &0.18 & 1    & 0.29 & 0.53 &0.63 &0.86 &0.73 & 0.93 & 0.72 & 0.82 \\
			600 & 200 & 0.12 & 0.09 & 0.1  &0.10 &0.38 &0.18  & 1    & 0.3  & 0.55 &0.57 &0.89 &0.71  & 0.87 & 0.80 & 0.83 \\
			750 & 250 & 0.09 & 0.08 & 0.08 &0.08 &0.36 &0.16  & 1    & 0.31 & 0.55 &0.59 &0.9 &0.72  & 0.80 & 0.86  & 0.83 \\
			900 & 300 & 0.11 & 0.09 & 0.09 &0.05 &0.31 &0.13  & 0.99 & 0.31 & 0.55 &0.56 &0.92 &0.72  & 0.77 & 0.87 & 0.82 \\ \hline
		\end{tabular}
	}
	\caption{Model selection performance table with sparsity 5\%}
	\label{model_selection_table2}
\end{table}
It is clear that our hierarchical fully Bayesian approach with beta-mixture prior and multiplicative prior outperforms the penalized likelihood approaches, the Bayesian DAG-Wishart and ESC approach based on almost all measures. 
The PPV values for our Bayesian spike and slab Cholesky approach 
are all above $0.55$, while the ones for the penalized likelihood approach and ESC are below $0.2$. Though the PPV for the DAG-Wishart approach is almost 1, it is actually a consequence of the maximized log score occurring at the most sparse model. Hence, The precision (PPV) for the DAG-Wishart method is rather high, as the resulting $L$ is extremely sparse and all the remaining non-zero entries are the true elements in $L_0$. The TPR values for the proposed approaches are almost all beyond $0.70$, while the ones for the penalized likelihood approaches are all below $0.27$. Now again under this measure, as a result of the final sparse estimator, DAG-Wishart Bayesian approach performs very poorly compared to the spike and slab approach with beta-mixture/multiplicative prior. For the most comprehensive measure of MCC, our fully Bayesian approach outperforms all the other three methods under all the cases of $(n,p)$ and two different sparsity levels.
	\begin{figure} [h]
		\centering
	\includegraphics[width=100mm]{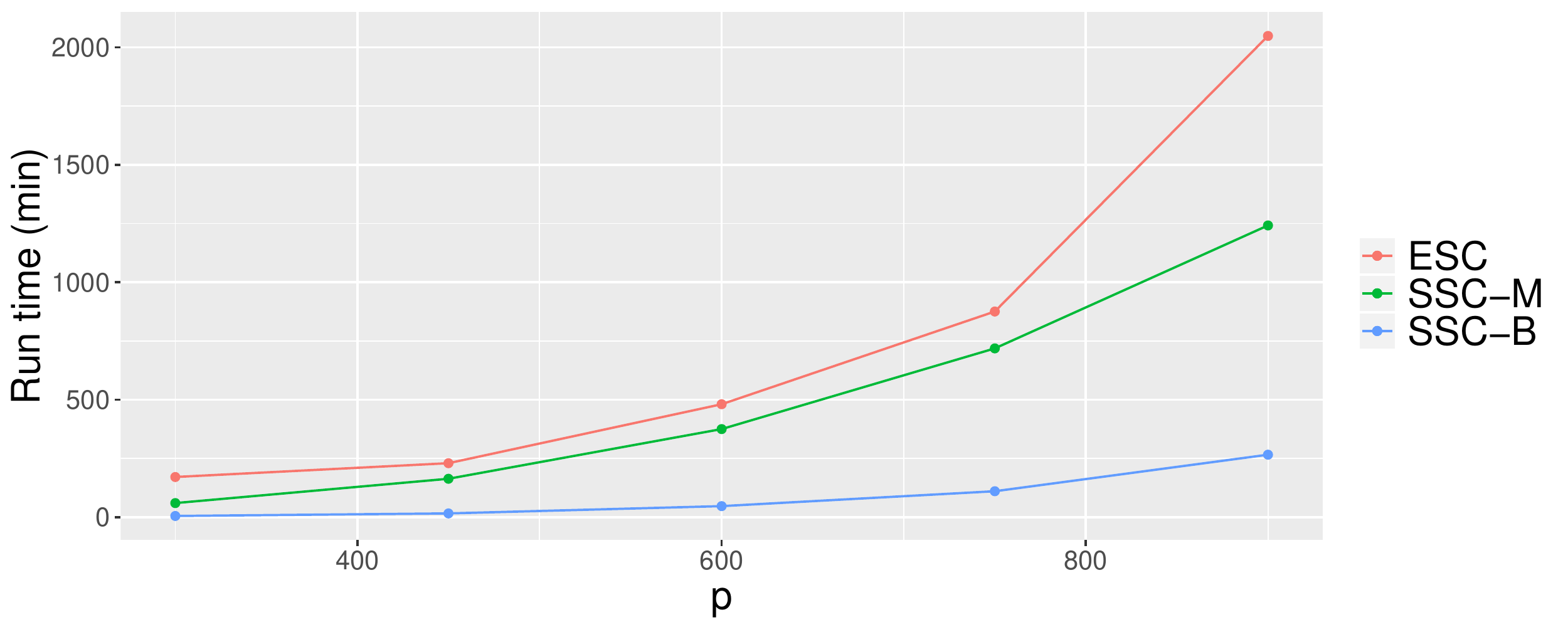}
	\caption{Run time comparison.}
	\label{run_time}
\end{figure}

It is also meaningful to compare the computational runtime between different methods. In Figure \ref{run_time}, we plot the run time comparison between our spike and slab Cholesky with beta-mixture prior/multiplicative prior and ESC. Since the marginal posterior is available in closed form (up to a constant) for the SSC with beta-mixture prior, we can see that the run time for SSC-B via thresholding coupled with stochastic search is significantly lessened compared to the MCMC approach. The computational cost of ESC is extremely expensive in the sense that it requires not only additional run time, but also larger memory (more than 30GB when $p > 900$). On the other hand, for the multiplicative prior, though the model selection performance is almost the best among all the competitors, with the extra step of the Laplace approximation for calculating each posterior probability, the computational burden is quite extensive as $p$ increases.
\begin{figure}[htbp]
	\centering
	\begin{subfigure}[True $L_0$ with sparsity $3\%$]
		{\includegraphics[width=45mm]{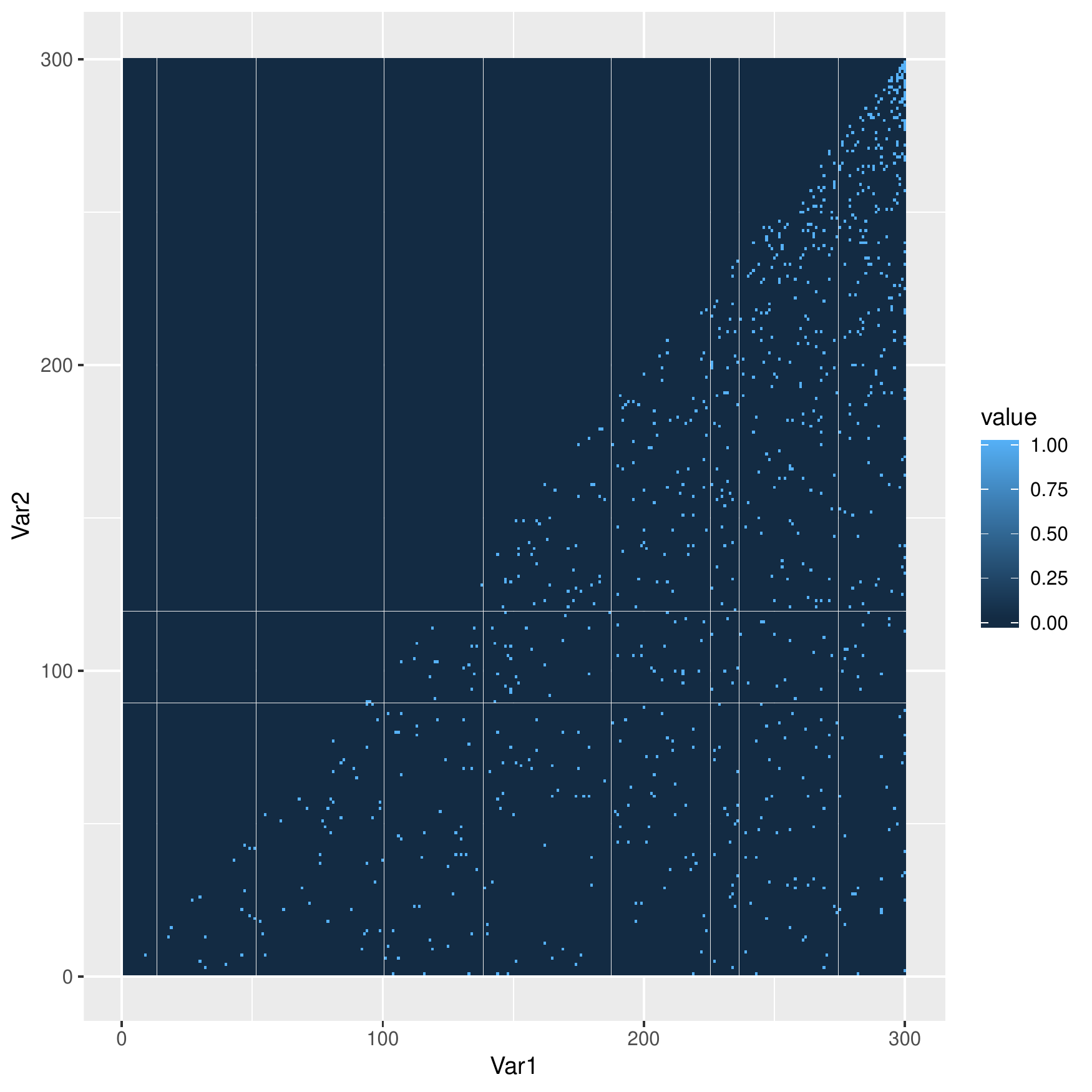}}
	\end{subfigure}
	\qquad
	\begin{subfigure}[Estimated $L$]
		{\includegraphics[width=45mm]{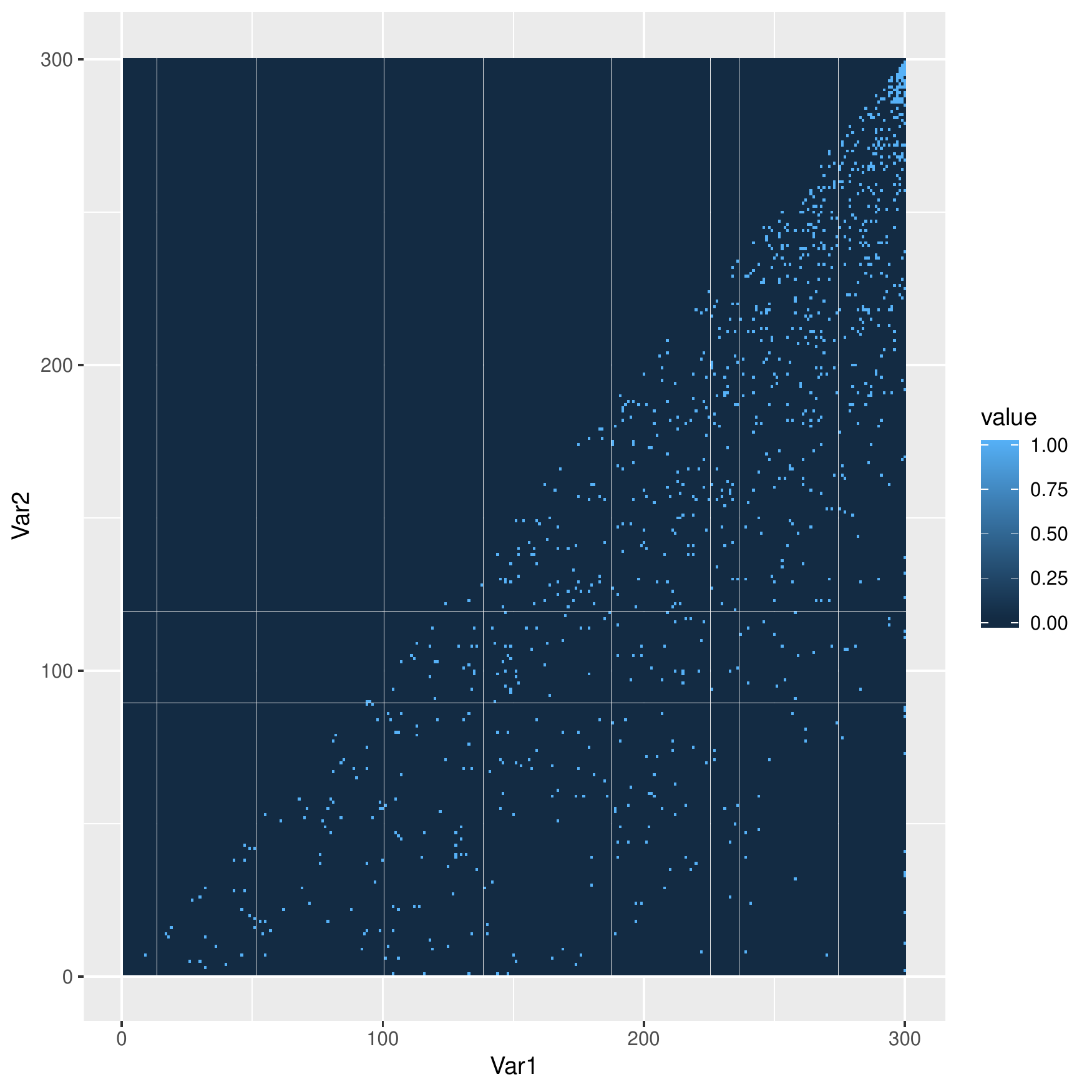}}
	\end{subfigure}%

	\begin{subfigure}[True $L_0$ with sparsity $5\%$]
		{\includegraphics[width=45mm]{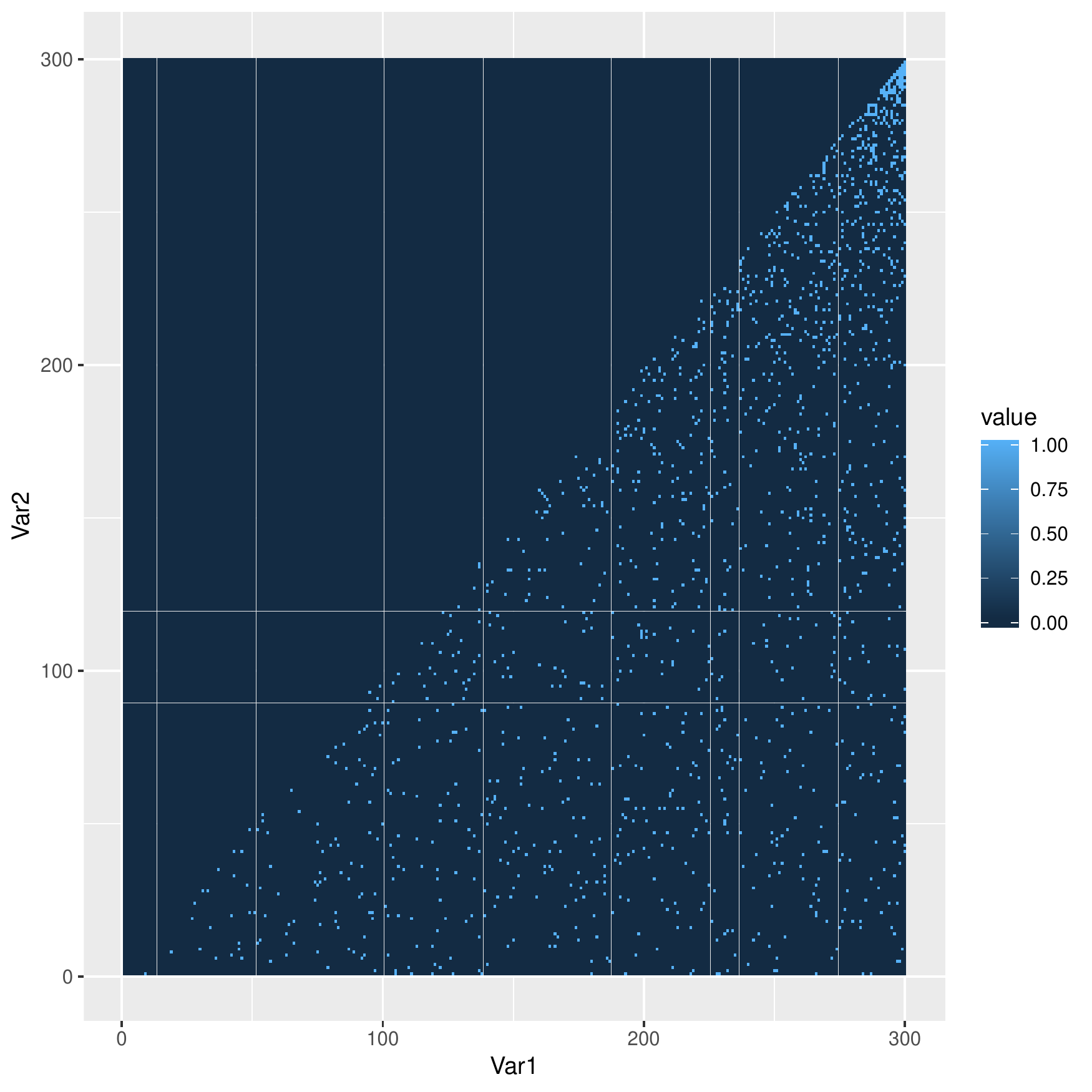}}
	\end{subfigure}
	\qquad
	\begin{subfigure}[Estimated $L$]
		{\includegraphics[width=45mm]{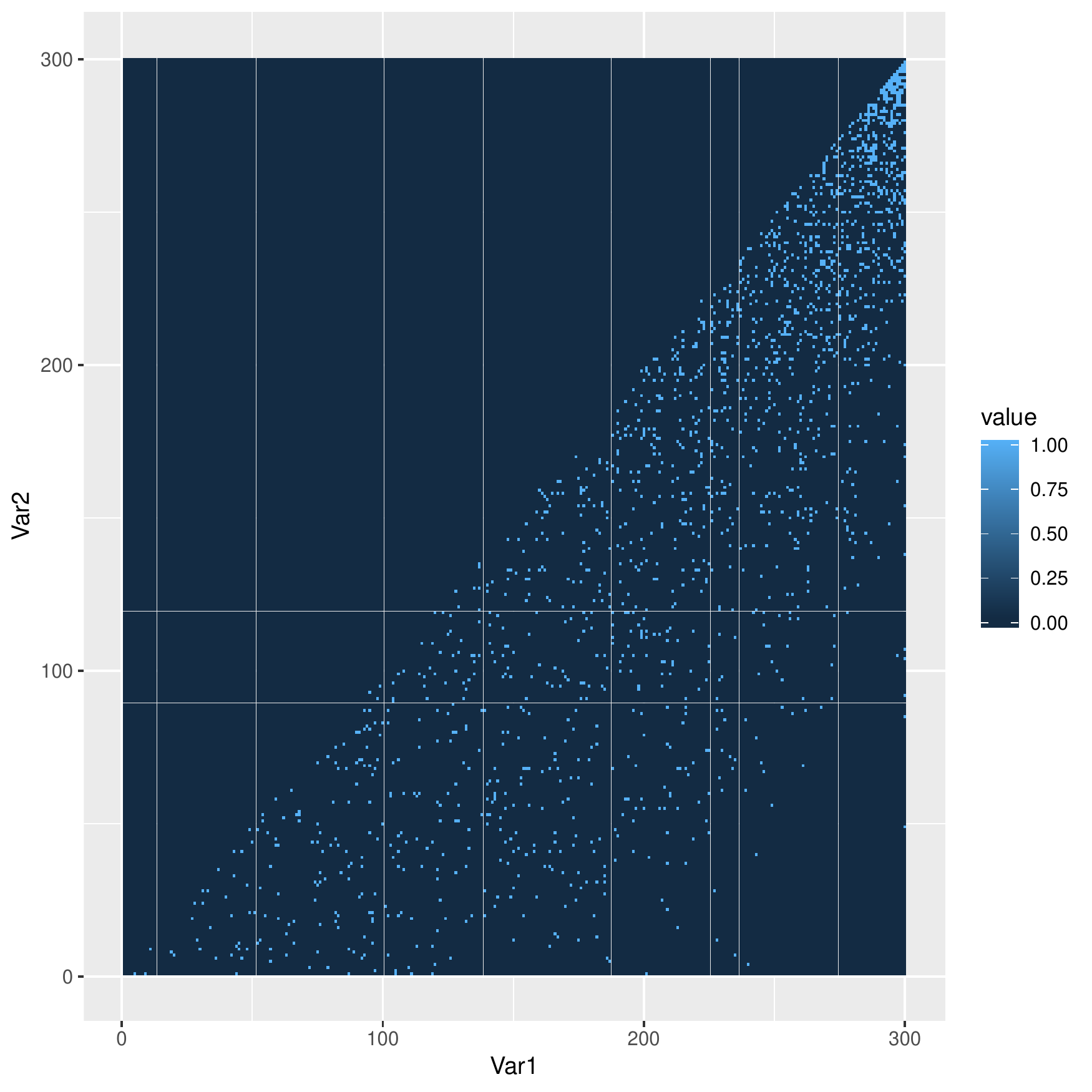}}
	\end{subfigure}
	\caption{Heatmap comparison with $(n, p) = (100, 300)$}
	\label{heat_3}
\end{figure}

Overall, this experiment illustrates that the proposed hierarchical fully Bayesian approach with our spike and slab Cholesky prior and the beta-mixture prior can be used for a broader yet computationally feasible model search, while our spike and slab Cholesky prior with the multiplicative prior though more computationally expensive, can lead to a much more significant improvement in model selection performance for estimating the sparsity pattern of the Cholesky factor and the underlying DAG.

\section{Proofs} \label{sec:modelselectionproofs}
In this section, we take on the task of proving our main results presented in Theorems \ref{thm4} to \ref{thm3}.
\subsection{Proof of Theorem \ref{thm4}} \label{sec:proof_thm4}
The proof of Theorem \ref{thm4} will be broken into several steps. We begin our strong selection consistency proof by first proving the Lemma \ref{graph_ratio_lemma} and Lemma \ref{newlemma1} which give the upper bound for the prior ratio between any ``non-true" model $Z$ and the true model $Z_0$.
\begin{proof} [Proof of Lemma \ref{graph_ratio_lemma}]
	First note that following from model (\ref{model5}) and (\ref{model6}), we have 
	\begin{align} \label{marginal_Z_upper}
	\pi (Z) =& \int \prod_{j = 1}^{p}\pi(\omega_j) \pi(Z | \omega_1, \ldots, \omega_{p})d\omega_1\ldots d\omega_{p} \nonumber\\
=& \int \prod_{1 \le j < k \le p} (\omega_k\omega_j)^{Z_{kj}}(1-\omega_k\omega_j)^{1-Z_{kj}} \prod_{j = 1}^{p}\pi(\omega_j) d\omega_1\ldots d\omega_{p}  \nonumber \\
\le & \int \prod_{1 \le j < k \le p} (\omega_k\omega_j)^{Z_{kj}} \prod_{j = 1}^{p}\pi(\omega_j) d\omega_1\ldots d\omega_{p} \nonumber\\
\le & \prod_{j = 1}^p \int \omega_j^{|Z_j|} \omega_j^{\alpha_1-1} (1-\omega_j)^{\alpha_{2}-1} \frac{\Gamma(\alpha_1+\alpha_2)}{\Gamma(\alpha_1)\Gamma(\alpha_{2})}d\omega_j \nonumber\\
\le&  \prod_{j = 1}^p\frac{\Gamma(\alpha_1 + \alpha_2)\Gamma(\alpha_1 + |{Z}_j|)}{\Gamma(\alpha_1 + \alpha_2 + |{Z}_j|)\Gamma(\alpha_1)}.
	\end{align}
Denote $A_j= \left\{\omega_j: \omega_j < \frac{\alpha_1}{\max\left\{p^{\frac c 2}, d^{\frac c {c-2}}\right\}}\right\}$. Note that on $A_j$, $1 - \omega_i\omega_j > 1 - \frac{\alpha_1^2}{\max\left\{p^{c}, d^{\frac {2c} {c-2}}\right\}}$. Hence, by $c > 2$,
\begin{align*}
\prod_{1 \le j < k \le p}(1-\omega_k\omega_j)^{1-Z_{kj}} \ge& \left(1 - \frac{\alpha_1^2}{\max\left\{p^{c}, d^{\frac {2c} {c-2}}\right\}}\right)^{p^2} \\
\ge & \left(1- \frac{\alpha_1^2}{p^2}\right)^{p^2} \\
\ge & e^{-2\alpha_1^2}, \quad \mbox{for } p \ge \sqrt{2} \alpha_1.
\end{align*}
The last inequality follows from $\frac{\log(1-x)}{x} \ge -2$, for $0 \le x < \frac 1 2$. Hence, for $p \ge \sqrt{2}\alpha_1$, we have
\begin{align} \label{lower_bound_Z_0}
\pi(Z_0) =& \int \pi(Z_0 | \omega_1, \ldots, \omega_{p}) \prod_{j = 1}^p\pi(\omega_j)d\omega_1 \ldots d\omega_{p}\nonumber\\
\ge& e^{-2\alpha_1^2} \prod_{j = 1}^p \int_{A_j} \omega_j^{|{Z_0}_j| + \alpha_1 - 1}(1 - \omega_j)^{\alpha_{2}-1}\frac{\Gamma(\alpha_1+\alpha_2)}{\Gamma(\alpha_1)\Gamma(\alpha_{2})}d\omega_j \nonumber\\
\ge & e^{-2\alpha_1^2} \prod_{j = 1}^p\frac{\Gamma(\alpha_1 + \alpha_2)\Gamma(\alpha_1 + |{Z_0}_j|)}{\Gamma(\alpha_1 + \alpha_2 + |{Z_0}_j|)\Gamma(\alpha_1)}P\left(B_j < \frac{\alpha_1}{\max\left\{p^{\frac c 2}, d^{\frac c {c-2}}\right\}}\right),
\end{align}
where $B_j \sim \mbox{Beta} (\alpha_1 + |{Z_0}_j|, \alpha_2)$. By Markov's inequality and $\alpha_{2} \sim \max\left\{p^c, d^{\frac{2c}{c-2}}\right\}$, where $c> 2$, we have
\begin{align} \label{markov_B_i} 
P\left(B_j < \frac{\alpha_1}{\max\left\{p^{\frac c 2}, d^{\frac c {c-2}}\right\}}\right)
\ge& 1 - \frac{E(B_j)}{\frac{\alpha_1}{\max\left\{p^{\frac c 2}, d^{\frac c {c-2}}\right\}}} \nonumber\\
\ge& 1 - \frac{\alpha_1 + |{Z_0}_j|}{\alpha_1\left(\max\left\{p^{\frac c 2}, d^{\frac c {c-2}}\right\}\right)} \nonumber\\
\ge& e^{-\frac{2(\alpha_1 + |{Z_0}_j|)}{\alpha_1\max\left\{p^{\frac c 2}, d^{\frac c {c-2}}\right\}}},
\end{align}
for $p \ge 4+\frac{4}{\alpha_1}.$ The last inequality follows from 
\begin{align*}
\frac{\alpha_1 + |{Z_0}_j|}{\alpha_1\left(\max\left\{p^{\frac c 2}, d^{\frac c {c-2}}\right\}\right)}  \le& \frac{\alpha_1 + d}{\alpha_1\left(\max\left\{p^{\frac c 2}, d^{\frac c {c-2}}\right\}\right)}  \nonumber\\
\le& \frac 1 p +  \frac{d}{\alpha_1\left(\max\left\{p^{\frac c 2}, d^{\frac c {c-2}}\right\}\right)} \nonumber\\
\le & \frac 1 p + \frac{d}{\alpha_1 p(d^{2/(c-2)})^{c/2-1}} \nonumber\\
\le& \frac 1 p + \frac 1 {\alpha_1 p} \le \frac 1 2,
\end{align*}
for $p \ge 4 + \frac 4 {\alpha_1}$. It then follows by (\ref{lower_bound_Z_0}) and (\ref{markov_B_i}) that
\begin{align} \label{lower_Z_0}
\pi(Z_0) \ge& e^{-2\alpha_1^2} e^{-\frac{2p(\alpha_1 + d)}{\alpha_1\max\left\{p^{\frac c 2}, d^{\frac c {c-2}}\right\}}} \prod_{j = 1}^p\frac{\Gamma(\alpha_1 + \alpha_2)\Gamma(\alpha_1 + |{Z_0}_j|)}{\Gamma(\alpha_1 + \alpha_2 + |{Z_0}_j|)\Gamma(\alpha_1)} \nonumber \\
\ge& e^{-2\alpha_1^2 - 2\alpha_1- \frac 2 {\alpha_{2}}}\prod_{j = 1}^p\frac{\Gamma(\alpha_1 + \alpha_2)\Gamma(\alpha_1 + |{Z_0}_j|)}{\Gamma(\alpha_1 + \alpha_2 + |{Z_0}_j|)\Gamma(\alpha_1)},
\end{align}
for $p \ge 4+\frac{4}{\alpha_1} + 2\sqrt{\alpha_1}$. \\
Therefore, by (\ref{marginal_Z_upper}) and (\ref{lower_Z_0}) that
\begin{align} \label{graph_ratio}
\frac{\pi(Z)}{\pi(Z_0)} \le e^{2\alpha_1^2  + 2\alpha_1 + \frac 2 {\alpha_{2}}} \prod_{j=1}^{p}\frac{B(\alpha_1+ |Z_j|, \alpha_2)}{B(\alpha_1+ |{Z_0}_j|, \alpha_2)},
\end{align}
for $p \ge 4+\frac{4}{\alpha_1} + 2\sqrt{\alpha_1}$.
\end{proof}
\noindent
Next, we prove the result on the upper bound for the marginal posterior ratio that is Lemma \ref{newlemma1}.
\begin{proof}[Proof of Lemma \ref{newlemma1}]
		Next, it follows model (\ref{model1}) to (\ref{model4}) that
	\begin{align} \label{posterior1}
	\begin{split}
	&\pi(Z | \bm Y)\\
	=& \int \frac{\pi(\bm{Y}| Z,(L,D)) \pi \left(L|D, Z\right) \pi(Z) \pi(D)} {\pi(\bm{Y})} dLdD \\
	=& \frac{\pi(Z)}{\pi(\bm{Y})}\int \pi(\bm{Y}| Z,(L,D)) \pi \left(L|D, Z\right) \pi(D) dLdD. 
	\end{split}
	\end{align}
	Note that
	\begin{align}\label{posterior2}
	&\pi(\bm{Y}| Z,(L,D)) \pi \left(L|D, Z\right) \pi(D) \nonumber\\
	= &\prod_{i = 1}^{n}\left((2\pi)^{- \frac p 2}\prod_{j = 1}^{p} d_j^{-\frac12}\exp\left\{-\frac 1 2 \bm Y_i^T(LD^{-1}L^T)\bm Y_i\right\}\right)\nonumber \\ 
	& \times \prod_{j = 1}^{p-1}\prod_{k=j+1}^p\left(N\left(\bm 0, \tau^2 d_j\right) + (1-Z_{kj}) \delta_0(L_{kj})\right) \times \prod_{j=1}^p \pi(d_j)\nonumber \\ 
	\propto &\prod_{j = 1}^{p-1}\left\{d_j^{-\frac n 2 } \exp\left\{-\frac{n \left(L_{Z.j}^\ge\right)^TS_{Z}^{\ge j}L_{Z.j}^\ge}{2d_j}\right\} \right\}d_p^{-\frac n 2}\exp\left\{-\frac{nS_{pp}}{d_p}\right\}\nonumber \\ 
	& \times \prod_{j = 1}^{p-1} \left(d_j\tau^2\right)^{-\frac {|Z_j|}{2}}\exp\left\{-\frac{\left(L_{Z.j}^>\right)^TL_{Z.j}^>}{\tau^2d_j}\right\} \times \prod_{j=1}^p \pi(d_j).
	\end{align}
	It now follows from 
	\begin{align*}
	\left(L_{Z.j}^\ge\right)^TS_{Z}^{\ge j}L_{Z.j}^\ge = \left(1,\left(L_{Z.j}^>\right)^T\right) \times \quad
	\begin{pmatrix} 
	S_{jj} & \left(S_{Z.j}^>\right)^T \\
	S_{Z.j}^> & S_Z^{>j}
	\end{pmatrix}
	\quad
	\times  \left(1,L_{Z.j}^>\right),
	\end{align*}
	that
	\begin{align}
	\begin{split}
	&\exp\left\{-\frac{n \left(L_{Z.j}^\ge\right)^TS_{Z}^{\ge j}L_{Z.j}^\ge}{2d_j} -\frac{\left(L_{Z.j}^>\right)^TL_{Z.j}^>}{\tau^2d_j}\right\}\\
	=& \exp\left\{-\frac{\left(L_{Z.j}^> + \left(\tilde S_Z^{>j}\right)^{-1}\tilde S_{Z.j}^>\right)^T\tilde S_Z^{>j}\left(L_{Z.j}^> + \left(\tilde S_Z^{>j}\right)^{-1}\tilde S_{Z.j}^>\right)}{\frac{2d_j}{n}}\right\}\\
	&\times \exp\left\{-\frac{\tilde S_{jj} - \left(\tilde S_{Z.j}^>\right)^T\left(\tilde S_Z^{>j}\right)^{-1}\tilde S_{Z.j}^>}{\frac{2d_j}{n}} + \frac 1 {2\tau^2d_j}\right\},	 
	\end{split}
	\end{align}
	where $\tilde{S} = S+\frac 1 {n\tau^2} I_p$.\\
If follows from Lemma \ref{graph_ratio_lemma}, (\ref{posterior1}) and (\ref{posterior2}) that integrating out $(L,D)$ gives us
	\begin{align} \label{posterior_propto11}
	&\pi(Z | \bm Y)\nonumber\\
	\propto& \pi(Z)\prod_{j = 1}^{p-1} \frac {1}{(n\tau^2)^{|Z_j|/2}} \left(\frac{n\tilde S_{j|Z_j}}{2} - \frac 1 {2\tau^2} + \lambda_2\right)^{-\frac n 2 - \lambda_1} |\tilde S_Z^{>j}|^{-\frac 1 2}\\ \nonumber
	= & \pi(Z)\prod_{j = 1}^{p-1} \frac {1}{(n\tau^2)^{|Z_j|/2}} \left(\frac{n\tilde S_{j|Z_j}}{2} - \frac 1 {2\tau^2} + \lambda_2\right)^{-\frac n 2 - \lambda_1} \left(|\tilde S_{Z}^{\ge i}|\tilde{S}_{j|Z_j}\right)^{-\frac 1 2},
	\end{align}
	in which $\tilde{S}_{j|Z_j} = \tilde{S}_{jj} - (\tilde{S}_{Z \cdot j}^>)^T 
	(\tilde{S}_Z^{>j})^{-1} \tilde{S}_{Z \cdot j}^>$. 

\noindent
Now note that we are interested in obtaining the posterior ratio. It immediately follows from (\ref{graph_ratio}) that, for $p \ge 4+\frac{4}{\alpha_1} + 2\sqrt{\alpha_1}$, given the data $Y$, the posterior ratio for any $Z$ compared to $Z_0$ can be simplified as
\begin{align} \label{m5_scale}
&\frac{\pi({Z}|\bm{Y})}{\pi({Z}_0|\bm{Y})} \nonumber\\
=& M_1\prod_{j=1}^{p-1} (n\tau^2)^{-\frac{|Z_j| - |{Z_0}_j|}2} \frac{B(\alpha_1+ |Z_j|, \alpha_2)}{B(\alpha_1+ |{Z_0}_j|, \alpha_2)}\nonumber \\
&\times \frac{|\tilde{S}_{Z_0}^{\ge j}|^{\frac12}}{|\tilde{S}_{Z}^{\ge j}|^{\frac12}}\left(\frac{\tilde{S}_{j|{Z_0}_j}}{\tilde{S}_{j|Z_j}}\right)^{\frac 1 2} \left(\frac{\tilde S_{j|{Z_0}_j}- \frac 1 {n\tau_{n,p}^2}+ \frac{2\lambda_2}{n}}{\tilde S_{j|{Z}_j} - \frac 1 {n\tau_{n,p}^2} + \frac{2\lambda_2}{n}}\right)^{\frac n 2 + \lambda_1} \nonumber\\
\triangleq& M_1\times PR^\prime_j(Z,Z_0),
\end{align} where $M_1 = e^{2\alpha_1^2  + 2\alpha_1 + \frac 2 {\alpha_{1}}}$, $\tilde{S} = S+\frac 1 {n\tau_{n,p}^2} I_p$ and $\tilde{S}_{j|Z_j} = \tilde{S}_{jj} - (\tilde{S}_{Z \cdot j}^>)^T 
(\tilde{S}_Z^{>j})^{-1} \tilde{S}_{Z \cdot j}^>$.
\end{proof}
Next, we show that in our setting, the sample and 
population covariance matrices are sufficiently close with high probability. It follows by Lemma A.3 
of \citep{Bickel:Levina:2008} and Hanson-Wright inequality from \citep{rudelson2013} that there exists constants $m_1,m_2$ and $\delta$ 
depending on $\epsilon_{0,n}$ only such that for $1 \le i,j \le p$, we have:
\begin{equation*}
\bar{P}(| S_{ij} - (\Sigma_0)_{ij} | \ge t) \le m_1 \exp\{-m_2n(t\epsilon_{0})^2\}, \, |t| \le \delta.
\end{equation*}
\noindent
By the union-sum inequality, for a large enough $c'$ such that $2-m_2(c')^2/4 < 0$, we 
get that 
\begin{equation} \label{smplbound2}
\bar{P}\left(\Vert {S}-\Sigma_0 \Vert_{\max} \ge c' \sqrt{\frac{\log p}{n}}\right) 
\leq mp^{2-m'c'^2/4} \rightarrow 0. 
\end{equation}
Define the event $C_n$ as 
\begin{equation} \label{smplbound1}
C_n = \left\{\Vert {S}-\Sigma_0 \Vert_{\max} \ge c' \sqrt{\frac{\log p}{n}}\right\}. 
\end{equation}
We now analyze the behavior of $PR^\prime_j(Z,Z_0)$ defined in (\ref{m5}) under different scenarios in a sequence of three lemmas (Lemmas \ref{lm4} - \ref{lm6}). Recall that our 
goal is to find an upper bound for 
$PR^\prime_j(Z,Z_0)$, such that the upper bound converges to $0$ as $n 
\rightarrow \infty$. For all the following analyses, we will restrict ourselves to the event $C_n^c$. 
\begin{lemma} \label{lm4}
	If all the active elements in set ${Z_j}_0$ are contained in the true model ${Z}_j$ denoted as $Z_j \supset {Z_0}_j$, then there exists $N_1$ (not depending on $Z$) such that for $n \geq N_1$ we have for some constant $\kappa > 1$, $PR^\prime_j(Z,Z_0) \le \left(2p\right)^{-\frac c \kappa (|Z_j| - |{Z_0}_j|)}  \rightarrow 0,  \mbox{ as } n \rightarrow \infty.$
\end{lemma}
\begin{proof} [Proof of Lemma \ref{lm4}]
	We begin by simplifying the posterior ratio given in (\ref{m5}). Using the fact that 
	$
	\sqrt{x + \frac{1}{4}} \leq \frac{\Gamma (x+1)}{\Gamma \left( x+\frac{1}{2} \right)} 
	\leq \sqrt{x + \frac{1}{2}} 
	$
	for $x > 0$ (see \citep*{Watson:1959}), it follows from Assumption \ref{assumption:alpha_2_multi}, $|Z_j| > |{Z_0}_j|$, and $1 + x \le e^x$, $1 - x \le e^{-x}$, for $0 \le x \le 1$, that for a large enough constant $M$, 
	and large enough $n$, we have  
	\begin{align} \label{ratio_B_scale}
	&\frac{B(\alpha_1+ |Z_j|, \alpha_2)}{B(\alpha_1+ |{Z_0}_j|, \alpha_2)} \nonumber\\
	=& \frac{\Gamma(|Z_j|+\alpha_1)\Gamma(\alpha_1+\alpha_2+|{Z_0}_j|)}{\Gamma(|{Z_0}_j|+\alpha_1)\Gamma(\alpha_1 +\alpha_{2} + |{Z}_j|)} \nonumber\\
	\le& M\frac{(|Z_j|+\alpha_1)^{|Z_j|+\alpha_1}}{(|{Z_0}_j|+\alpha_1)^{|{Z_0}_j|+\alpha_1}} \frac{(\alpha_1 + \alpha_2 + |{Z_0}_j|)^{\alpha_1 + \alpha_2 + |{Z_0}_j|}}{(\alpha_1 +\alpha_{2} + |{Z}_j|)^{\alpha_1 +\alpha_{2} + |{Z}_j|}} \nonumber\\
	\le&M(|Z_j|+\alpha_1)^{|Z_j|-|{Z_0}_j|}\left(1 + \frac{|Z_j| - |{Z_0}_j|}{|{Z_0}_j|+\alpha_1}\right)^{|{Z_0}_j|+\alpha_1} \nonumber \\
	& \times (\alpha_1 +\alpha_{2} + |{Z_0}_j|)^{-(|Z_j| - |{Z_0}_j|)}\left(1 - \frac{|Z_j| - |{Z_0}_j|}{\alpha_1 + \alpha_2 + |{Z}_j|}\right)^{\alpha_1 + \alpha_2 + |{Z}_j|} \nonumber\\
	\le& (c_1p^c/|Z_j|)^{-(|Z_j| - |{Z_0}_j|)},
	\end{align}
	for some constant $c_1 > 0$. \\
	Next, since $Z_j \supset {Z_0}_j$, we can write 
	$|\tilde{S}_{Z}^{\ge i}| = |\tilde{S}_{Z_0}^{\ge i}| 
	|SC_{{\tilde{S}}_{Z_0}^{\ge i}}|$. Here 
	$SC_{{\tilde{S}}_{Z_0}^{\ge i}}$ is the Schur complement of 
	${\tilde{S}}_{Z_0}^{\ge i}$, defined by 
	$$
	SC_{{\tilde{S}}_{Z_0}^{\ge i}} = D - B^T \left({\tilde{S}}_{Z_0}^{\ge 
		i}\right)^{-1} B 
	$$ 
	
	\noindent
	for appropriate sub matrices $A$ and $B$ of $\tilde{S}_{Z}^{\ge j}$. Since 
	${\tilde{S}}_{Z}^{\ge j} \geq 
	\left(\frac 1 {n\tau_{n,p}^2} I_p\right)_{Z}^{\ge j}$~\footnote{For matrices $A$ and $B$, we say $A \ge B$ if $A - B$ is positive semi-definite}, and 
	$SC_{{\tilde{S}}_{Z_0}^{\ge j}}^{-1}$ is a principal submatrix of 
	$\left( \tilde{S}_{Z}^{\ge j} \right)^{-1}$, the largest eigenvalue of $SC_{{\tilde{S}}_{Z_0}^{\ge j}}^{-1}$ is bounded above by ${n\tau_{n,p}^2}$. Therefore,
	\begin{equation} \label{new7}
	\left(\frac{|\tilde{S}_{Z_0}^{\ge i}|}{|\tilde{S}_{Z}^{\ge j}|} \right)^{\frac12} = {\lvert SC_{\tilde{S}_{Z_0}^{\ge j}}^{-1}\rvert^{1/2}}
	\le \left(\sqrt{{n\tau_{n,p}^2}}\right)^{|Z_j| - |{Z_0}_j|}.
	\end{equation}
	Denote ${S}_{j|Z_j} = {S}_{jj} - ({S}_{Z \cdot j}^>)^T 
	({S}_Z^{>j})^{-1} {S}_{Z \cdot j}^>$. It immediately follows that
	\begin{equation} \label{tilde_s_lowerbound}
	\tilde{S}_{i|{Z}_j} \ge {S}_{i|{Z}_j}.
	\end{equation}
	Since we are restricting ourselves to the event $C_n^c$, it follows by 
	(\ref{smplbound1}) that 
	\begin{equation*}
	|| {S}_{Z_0}^{\ge i} - (\Sigma_0)_{Z_0}^{\ge i} ||_{(2,2)} \le 
	(|{Z_0}_j|+1)c^\prime\sqrt{\frac{\log p}{n}}.
	\end{equation*}
	Therefore,
	\begin{align} \label{pp1}
	\begin{split}
	&|| ({S}_{Z_0}^{\ge i})^{-1} - ((\Sigma_0)_{Z_0}^{\ge i})^{-1} ||_{(2,2)} \\ 
	=& || ({S}_{Z_0}^{\ge i})^{-1}||_{(2,2)} || {S}_{Z_0}^{\ge i} - (\Sigma_0)_{Z_0}^{\ge i}||_{(2,2)} || 
	((\Sigma_0)_{Z_0}^{\ge i})^{-1} ||_{(2,2)} \\
	\le &	(|| ({S}_{Z_0}^{\ge i})^{-1} - ((\Sigma_0)_{Z_0}^{\ge i})^{-1} ||_{(2,2)} + \frac 1 {\epsilon_{0}})(|{Z_0}_j|+1)c^\prime \sqrt{\frac{\log p}{n}}.
	\end{split}
	\end{align}
	Recall $d = \max_{1 \le j \le p-1}|{Z_0}_j|$. By the assumption that $d\sqrt{\frac{\log p}{n}} \rightarrow 0$ and (\ref{pp1}), for large enough $n$, we have 
	\begin{equation} \label{inverse_s_sigma}
	|| ({S}_{Z_0}^{\ge i})^{-1} - ((\Sigma_0)_{Z_0}^{\ge i})^{-1} ||_{(2,2)} \le \frac {4c\prime} {\epsilon_0}d\sqrt{\frac{\log p}{n}} = o(1) \mbox{ and } \frac 1 {{S}_{i|{Z_0}_j}} = \left[({S}_{Z_0}^{\ge i})^{-1}\right]_{ii}
	\ge \frac {\epsilon_{0}} 2.
	\end{equation}
	Note that for any $Z$, $|| \tilde{S}_{Z}^{\ge j} - S_{Z}^{\ge j} ||_{\max} \le \frac 1 {n\tau_{n,p}^2}$ gives us
	$|| \tilde{S}_{Z_0}^{\ge j} - S_{Z_0}^{\ge j}||_{(2,2)} \le (|{Z_0}_j|+1) \frac 1 {n\tau_{n,p}^2}.$
	Therefore,
	\begin{align} \label{pp2}
	\begin{split}
	&|| (\tilde{S}_{Z_0}^{\ge j})^{-1} - (S_{Z_0}^{\ge j})^{-1}||_{(2,2)} \\ 
	=& || (\tilde{S}_{Z_0}^{\ge i})^{-1}||_{(2,2)} || \tilde{S}_{Z_0}^{\ge i} - S_{Z_0}^{\ge i}||_{(2,2)} || 
	(S_{Z_0}^{\ge i})^{-1} ||_{(2,2)} \\
	\le &	(|| (\tilde{S}_{Z_0}^{\ge j})^{-1} - (S_{Z_0}^{\ge j})^{-1} + ||_{(2,2)} +|| ({S}_{Z_0}^{\ge i})^{-1} - ((\Sigma_0)_{Z_0}^{\ge i})^{-1} ||_{(2,2)} +  \frac 1 {\epsilon_{0}})(|{Z_0}_j|+1)\frac 1 {n\tau_{n,p}^2}.
	\end{split}
	\end{align}
	Following from (\ref{inverse_s_sigma}) and $\frac{d}{n\tau_{n,p}^2} \rightarrow 0$, for large enough $n$, (\ref{pp2}) yields
	\begin{equation} \label{inverse_s_stilde}
	|| (\tilde{S}_{Z_0}^{\ge j})^{-1} - (S_{Z_0}^{\ge j})^{-1} ||_{(2,2)} \le \frac{8}{\epsilon_{0}} \frac d {n\tau_{n,p}^2} \mbox{ and }
	\frac 1 {\tilde{S}_{i|{Z_0}_j}} = \left[(\tilde{S}_{Z_0}^{\ge i})^{-1}\right]_{ii}
	\ge \frac {\epsilon_{0}} 4.
	\end{equation}
	Hence, it follow from (\ref{inverse_s_stilde}) and (\ref{inverse_s_sigma}) that, 
	\begin{align}
	|\frac 1 {{S}_{i|{Z_0}_j}} - \frac 1 {\tilde{S}_{i|{Z_0}_j}}| \le \frac{8}{\epsilon_{0}} \frac d {n\tau_{n,p}^2} \mbox{ and } |{{S}_{i|{Z_0}_j}} - {\tilde{S}_{i|{Z_0}_j}}| \le c_1 \frac d {n\tau_{n,p}^2},
	\end{align}
	where $c_1 = 64/\epsilon_{0}^3$ is a constant.\\
	Further note that $n{d_0}_j^{-1}{S}_{i|{Z}_j} \sim \chi^2_{n-|{{Z}_j}|}$ and $n{d_0}_j^{-1}{S}_{i|{Z_0}_j} \stackrel{d}{=} n{d_0}_j^{-1}{S}_{i|{Z}_j} \oplus \chi^2_{|Z_j|-|{{Z_0}_j}|}$ under the true model. Since ${Z_0}_j \subset Z_j$, we get ${S}_{j|{Z_0}_j}  \ge S_{j|Z_j}$ and $\tilde {S}_{j|{Z_0}_j}  \ge \tilde S_{j|Z_j}$. It follows from Lemma 4.1 in \cite{CKG:nonlocal} that
	\begin{align} \label{chisquarebound_S_D}
	\begin{split}
	P\left[\left\lvert n{d_0}_{j}^{-1}{S}_{j|{Z}_j} - (n-|Z_j|)\right\rvert > \sqrt{(n-|Z_j|)\log p} \right] &\le 2p^{-\frac 1 8} \rightarrow 0,	
	\end{split}
	\end{align}
	and
	\begin{align} \label{chisquarebound_S_D_D_0}
	\begin{split}
	P\left[\left\lvert n{d_0}_{j}^{-1}{S}_{j|{Z_0}_j} - n{d_0}_{j}^{-1}{S}_{j|{Z}_j} - (|Z_j| - |{Z_0}_j|)\right\rvert > \sqrt{(|Z_j| - |{Z_0}_j|)\log p} \right] &\le 2p^{-\frac 1 8} \rightarrow 0.
	\end{split}
	\end{align}
	Following from Assumption \ref{assumption:Z}, Assumption \ref{assumption:tau}, Lemma \ref{newlemma1}, (\ref{new7}), (\ref{tilde_s_lowerbound}) and (\ref{inverse_s_stilde}), for larger enough $n > N_1$, we have

	\begin{align} \label{PR_upbound_1_scale}
	&PR^\prime_j(Z,Z_0) \nonumber \\
	\le& (c_1\sqrt{n\tau^2}p^c/|Z_j|)^{-(|Z_j| - |{Z_0}_j|)}\left(1 + \frac{n{d_0}_j^{-1}{S}_{j|{Z_0}_j} - n{d_0}_j^{-1}S_{j|{Z_0}_j} + c_1 \frac d {{d_0}_j\tau_{n,p}^2}}{n{d_0}_j^{-1}S_{j|{Z}_j}}\right)^{\frac 1 2}\nonumber \\ 
	&\times \left(1 + \frac{n{d_0}_j^{-1}{S}_{j|{Z_0}_j} - n{d_0}_j^{-1}S_{j|{Z_0}_j} + c_1 \frac d {{d_0}_j\tau_{n,p}^2}+ \frac{2\lambda_2}{{d_0}_j}}{n{d_0}_j^{-1}S_{j|{Z}_j}+ \frac{2\lambda_2}{{d_0}_j}}\right)^{\frac n 2 + \lambda_1}\nonumber \\ 
	\le & (2p^{-c})^{|Z_j| - |{Z_0}_j|}\left(\sqrt{\frac {\tau^2} n} \log n\right)^{-(|Z_j| - |{Z_0}_j|)} \nonumber\\
	&\times  \exp\left\{\frac{|Z_j| - |{Z_0}_j| + \sqrt{(|Z_j| - |{Z_0}_j|)\log p} + c_1\frac d {\tau_{n,p}^2}}{n - |Z_j| -  \sqrt{(n - |{Z}_j|)\log p}} \times (\frac {n+1} 2 + \lambda_1) \right\}\\ \nonumber
	\le & \left(2p\right)^{-\frac c \kappa (|Z_j| - |{Z_0}_j|)}, \mbox{ for some constant } \kappa > 1.
	\end{align}
	The second inequality follows from $\frac{d}{\tau_{n,p}^2\log p} \rightarrow 0$, as $n \rightarrow \infty$.
\end{proof}
\begin{lemma} \label{lm5}
	If all the active elements in set $Z_j$ are contained in the true model ${Z_0}_j$ denoted as $Z_j \subset {Z_0}_j$, then there exists $N_2$ (not depending on $Z$) such that for $n \geq N_2$ we have $PR^\prime_j(Z,Z_0) \le p^{-\frac{2c}{\kappa}d} \rightarrow 0,  \mbox{ as } n \rightarrow \infty.$
\end{lemma}
\begin{proof} [Proof of Lemma \ref{lm2}]
	Now we move to discuss the scenario when $Z_j$ is a subset of ${Z_0}_j$, i.e., $Z_j \subset {Z_0}_j$. 
	By the similar arguments in (\ref{ratio_B_scale}), it follows from Assumption \ref{assumption:alpha_2} and $|Z_j| < |{Z_0}_j|$, that for a large enough constant $c_1$ 
	and large enough $n$, we have  
	\begin{align} \label{ratio_B_subset}
		&\frac{B(\alpha_1+ |Z_j|, \alpha_2)}{B(\alpha_1+ |{Z_0}_j|, \alpha_2)} \nonumber\\
	=& \frac{\Gamma(|Z_j|+\alpha_1)\Gamma(\alpha_1+\alpha_2+|{Z_0}_j|)}{\Gamma(|{Z_0}_j|+\alpha_1)\Gamma(\alpha_1 +\alpha_{2} + |{Z}_j|)} \nonumber\\
	\le& (c_1p^c/d)^{|{Z_0}_j|-|Z_j|}.
	\end{align}
		It follows that $|\tilde{S}_{Z_0}^{\ge i}| = |\tilde{S}_{Z}^{\ge i}| 
	|SC_{{\tilde{S}}_{Z}^{\ge i}}|$, where $SC_{{\tilde{S}}_{Z}^{\ge i}}$ denotes the Schur complement of ${\tilde{S}}_{Z}^{\ge i}$, defined by $SC_{{\tilde{S}}_{Z}^{\ge i}} = \tilde{A} - \tilde{B}^T ({\tilde{S}}_{Z}^{\ge i})^{-1} \tilde{B}$ for appropriate sub matrices $\tilde{A}$ and $\tilde{B}$ of $\tilde{S}_{{Z_0}}^{\ge i}$. Recall that $d$ is the maximum number of nonzero entries among all the columns of $Z_0$. It follows by (\ref{pp1}) that if restrict to $C_n^c$, we have $$||(\tilde{S}_{{Z_0}}^{\ge i})^{-1} - ({(\Sigma_0)}_{{Z_0}}^{\ge i})^{-1}||_{(2,2)} \le \frac{4c'}{\epsilon_{0}} d \sqrt{\frac{\log p}{n}}$$ and $$ ||SC_{{\tilde{S}}_{Z}^{\ge i}}^{-1} - 
	SC_{(\Sigma_0)_{Z}^{\ge i}}^{-1}||_{(2,2)} \le \frac{4c'}{\epsilon_{0}} d \sqrt{\frac{\log p }{n}},$$ for $n > N_2'$, in which $SC_{(\Sigma_0)_{Z}^{\ge i}}$ represents the Schur complement of 
	$(\Sigma_0)_{Z}^{\ge i}$ given by $SC_{(\Sigma_0)_{Z}^{\ge i}} = \bar{A} - \bar{B}^T 
	((\Sigma_0)_{Z}^{\ge i})^{-1} \bar{B}$ for appropriate sub matrices $\bar{A}$ and $\bar{B}$ of 
	$(\Sigma_0)_{Z_0}^{\ge i}$. Hence,
	there exists $N_2''$ such that, for $n > N_2''$, we have
	\begin{align*}
	\left(\frac{|\tilde{S}_{Z_0 }^{\ge i}|}{|\tilde{S}_{Z}^{\ge i}|} \right)^{\frac12} = {|SC_{{\tilde{S}}_{Z}^{\ge i}}^{-1}|^{-\frac 1 2}} 
	&\le {\left(\lambda_{\mbox{min}}\left(SC_{{(\Sigma_0)}_{Z}^{\ge i}}^{-1}\right)- \frac{4c'}{\epsilon_{0}}d\sqrt{\frac{\log p}{n}} \right)^{-\frac{|{Z_0}_j| - |Z_j|}{2}}}\\
	&\le \left(\frac{\epsilon_0}{2} \right)^{-\frac{|{Z_0}_j| - |Z_j|}{2}}.
	\end{align*}
	It follows from $Z_j \subset {Z_0}_j$ that
	$
	\tilde{S}_{j|{Z_0}_j}  \le \tilde{S}_{j|Z_j}. 
	$
	
	\noindent
	Let $K_1 = \frac{4c'}{\epsilon_{0}}$. By (\ref{m5}) and Proposition 5.2 in \citep*{CKG:2017}, it follows 
	that there exists $N_2'''$ such that for $n \geq N_2'''$, we get 
	\begin{align} \label{pp6}
	&PR^\prime_j(Z,Z_0) \nonumber \\ 
	\le& \left(\sqrt{\frac{2n\tau_{n,p}^2}{\epsilon_{0}}} c_1p^{c}/d
	\right)^{|{Z_0}_j| - |Z_j|}  \left (\frac{\frac{1}{(\Sigma_0)_{j|Z_j}}+K_1 d \sqrt{\frac{\log p }{n}} - \frac{1}{n\tau_{n,p}^2}}
	{\frac{1}{(\Sigma_0)_{j|{Z_0}_j}}-K_1 d \sqrt{\frac{\log p}{n}} - \frac{1}{n\tau_{n,p}^2}} \right)^{\frac{n}{2}+\lambda_1} \nonumber \\
	\le& \left(\exp\left\{\frac{d \log\left(\frac{2}{\epsilon_{0}}\right)}{n+2\alpha_1} + 
	\frac{2{\log\left(p^{c} \sqrt{n\tau_{n,p}^2}\right)}\left(|{Z_0}_j| - |Z_j|\right)}{n+2\alpha_1} \right\}\right)^{\frac{n+2\lambda_1}
		{2}} \nonumber \\ 
	&\times \left(1+ \frac{(\frac{1}{(\Sigma_0)_{j|{Z_0}_j}} - \frac{1}{(\Sigma_0)_{j|Z_j}}) - 2K_1 d \sqrt{\frac{\log p}{n}} }{\frac{1}{(\Sigma_0)_{j|Z_j}} + K_1 d \sqrt{\frac{\log p}{n}}} \right)^{- \frac{n+2\lambda_1}{2}} \nonumber \\
	\le& \left(\exp\left\{\frac{d \log\left(\frac{2}{\epsilon_{0}}\right)}{n+2\lambda_1} + 
	\frac{2{\log\left(p^{c}\sqrt{n\tau_{n,p}^2}\right)}\left(|{Z_0}_j| - |Z_j|\right)}{n+2\lambda_1} \right\}\right)^{\frac{n+2\lambda_1}
		{2}} \nonumber \\ 
	&\times \left(1+ \frac{\epsilon_{0} s_n^2 \left(|{Z_0}_j| - |Z_j|\right) - 
		2K_1d\sqrt{\frac{\log p}{n}}}{2/\epsilon_{0}} \right)^{- \frac{n+2\lambda_1}{2}}. 
	\end{align}
	\noindent
	It follows from $\frac{d\log p + d\log(n\tau_{n,p}^2)}{ns_n^2} \rightarrow 0$ and 
	$\frac{d\sqrt{\frac{\log p}{n}}}{s_n^2} \rightarrow 0$ as $n \rightarrow \infty$, and $e^x \leq 1 + 2x$ for $x < \frac{1}{2}$, that there exists $N_2''''$ 
	such that for $n \geq N_2''''$, 
	$$ \frac{\epsilon_{0} s_n^2  \left(|{Z_0}_j| - |Z_j|\right) - 
		2K_1d\sqrt{\frac{\log p}{n}}}{2/\epsilon_{0}}  \ge \frac{\epsilon_{0}s_n^2}{2}
	$$
	and 
	$$
	\exp\left\{\frac{d \log\left(\frac{2}{\epsilon_{0}}\right)}{n+2\lambda_1} + 
	\frac{2{\log\left(p^c\sqrt{n\tau_{n,p}^2}\right)}\left(|{Z_0}_j| - |Z_j|\right)}{n+2\lambda_1} \right\} \leq 1 + \frac{\epsilon_{0}^2 s_n^2}{8}. 
	$$
	\noindent
	Hence, by (\ref{pp6}), we have 
	\begin{align*}
	PR^\prime_j(Z, Z_0 ) \le \left (\frac{1+\frac{\epsilon_{0}^2}{8}s_n^2}{1+\frac{\epsilon_{0}^2}{4}s_n^2} 
	\right)^{\frac{n+2\lambda_1}{2}},
	\end{align*}	
	\noindent
	for $n \geq \max(N_2',N_2'', N_2''',N_2'''')$. Since there exists at least one $(L_0) _{ji}$ ($j +1 \le i \le p$), such that $s_n^2 \le (L_0) _{ji}^2 \le \frac{(\Omega_0)_{jj}}
	{\epsilon_{0}} \le \frac{1}{\epsilon_{0}^2}$,$\epsilon_{0}^2 s_n^2 \le 1$ and $e^{-x} \ge 1-x$ when $x \ge 0$, we have for all $n \ge N_2 \triangleq \max(N_2',N_2'', N_2''',N_2'''')$,
	\begin{align} \label{PR_upbound_subset}
	PR^\prime_j(Z, Z_0) \le \left (1 - \frac{\frac{\epsilon_{0}^2}{8}s_n^2}{1 
		+ \frac{\epsilon_{0}^2}{4}s_n^2} \right)^{\frac{n+2\lambda_1}{2}} &\le 
	\exp\left\{-\left(\frac{\frac{\epsilon_{0}^2}{8}s_n^2}{1 + \frac{\epsilon_{0}^2}{4}
		s_n^2}\right)\left(\frac{n+2\lambda_1}{2}\right)\right\} \nonumber \\
	&\le e^{-\frac{1}{10}\epsilon_{0}^2 s_n^2(\frac{n+2\lambda_1}{2})} \le p^{-\frac{2c}{\kappa}d},
	\end{align}
	following from $\frac{d\log p}{ns_n^2} \rightarrow 0$, as $n \rightarrow \infty$.
\end{proof}
\begin{lemma} \label{lm6}
	If all the active elements in set $Z_j$ are not contained in the true model ${Z_0}_j$ and all the active elements in set ${Z_0}_j$ are not contained in the true model $Z_j$, denoted as ${Z_0}_j \neq Z_j$, ${Z_0}_j \nsubseteq Z_j$, and ${Z_0}_j \nsupseteq Z_j$, then there exists $N_3$ (not depending on $Z$) such that for $n \geq N_3$ we have $PR^\prime_j(Z,Z_0) \le \left(2p\right)^{-\frac c \kappa |Z_j|} \rightarrow 0,  \mbox{ as } n \rightarrow \infty.$
\end{lemma}
\begin{proof} [Proof of Lemma \ref{lm6}]
	Let $Z^*$ be an arbitrary 0-1 matrix satisfying $Z^*_j = Z_j \cap {Z_0}_j$. Immediately we get $pa_i({{\mathscr{D}}^*} ) \subset pa_i({\mathscr{D}} _0)$ and $pa_i({{\mathscr{D}}^*} ) \subset pa_i({\mathscr{D}})$. It follows from (\ref{m5}) that
	\begin{align} \label{pp3_scale}
	PR^\prime_j(Z, Z_0) =&  (n\tau^2)^{-\frac{|Z_j| - |{Z_0}_j|}2} \frac{B(\alpha_1+ |Z_j|, _2)}{B(\alpha_1+ |{Z_0}_j|, \alpha_2)}\frac{|\tilde{S}_{Z_0}^{\ge j}|^{\frac12}}{|\tilde{S}_{Z}^{\ge j}|^{\frac12}}\nonumber\\
	&\times \left(\frac{\tilde{S}_{j|{Z_0}_j}}{\tilde{S}_{j|Z_j}}\right)^{\frac 1 2}\left(\frac{\tilde S_{j|{Z_0}_j}- \frac 1 {n\tau_{n,p}^2} + \frac{2\lambda_2}{n}}{\tilde S_{j|{Z}_j} - \frac 1 {n\tau_{n,p}^2}+ \frac{2\lambda_2}{n}}\right)^{\frac n 2 + \lambda_1} \nonumber \\
	\le&  (n\tau^2)^{-\frac{|Z_j| - |Z^*_j|}2} \frac{B(\alpha_1+ |Z_j|, \alpha_2)}{B(\alpha_1+ |Z^*_j|, \alpha_2)}
	\frac{|\tilde{S}_{Z^*}^{\ge j}|^{\frac12}}{|\tilde{S}_{Z}^{\ge j}|^{\frac12}} \nonumber\\
	&\times \left(\frac{\tilde S_{j|{Z_0}_j}- \frac 1 {n\tau_{n,p}^2}+ \frac{2\lambda_2}{n}}{\tilde S_{j|Z^*_j} - \frac 1 {n\tau_{n,p}^2} + \frac{2\lambda_2}{n}}\right)^{\frac n 2 + \lambda_1} \nonumber\\ 
	&\times (n\tau^2)^{-\frac{|Z_j^*| - |{Z_0}_j|}2} \frac{B(\alpha_1+ |Z^*_j|, \alpha_2)}{B(\alpha_1+ |{Z_0}_j|, \alpha_2)}
	\frac{|\tilde{S}_{Z_0}^{\ge j}|^{\frac12}}{|\tilde{S}_{Z^*}^{\ge j}|^{\frac12}}  \nonumber\\
	&\times  \left(\frac{\tilde S_{j|Z^*_j}- \frac 1 {n\tau_{n,p}^2}+ \frac{2\lambda_2}{n}}{\tilde S_{j|{Z}_j} - \frac 1 {n\tau_{n,p}^2}+ \frac{2\lambda_2}{n}}\right)^{\frac n 2 + \lambda_1} \nonumber \\ 
	\le& PR^\prime_j(Z, Z^*)\times PR^\prime_j(Z^*, Z_0).	
	\end{align}
	\noindent
	Note that $Z^*_j \subset Z_j$. It follows from (\ref{PR_upbound_1_scale}) that 
	\begin{align} \label{PR_neq_1_scale} 
	PR^\prime_j(Z, Z^*) \le \left(2p\right)^{-\frac c \kappa (|Z_j| - |Z^*_j|)}, \mbox{ for some } \kappa > 1 \mbox{ and } n\ge N_4.
	\end{align}
	By (\ref{PR_upbound_subset_scale}) and $Z^*_j \subset {Z_0}_j$, we have
	\begin{equation} 
	PR^\prime_j(Z^*, Z_0)  \le  p^{-\frac{2c}{\kappa}d},
	\mbox{ for } n \ge N_5. 
	\end{equation}
	\noindent
	It follows from (\ref{pp3}) and $|Z^*_j| < d$ that
	\begin{align} \label{PR_upbound_neq_scale}
	PR^\prime_j(Z, Z_0) \le \left(2p\right)^{-\frac c \kappa |Z_j| - |Z^*_j|} p^{-\frac{2c}{\kappa}d} <  \left(2p\right)^{-\frac c \kappa |Z_j|}, \mbox{ for } n \ge N_3 = \max\left\{N_1,N_2\right\}.
	\end{align}
\end{proof}
The result of Theorem \ref{thm4} immediately follows from Lemma \ref{lm4} to Lemma \ref{lm6}, by noting that if $Z \neq Z_0$, then there exists at least one $j$, such that $Z_j \neq {Z_0}_j$. It follows that if we restrict to 
$C_n^c$, then 
\begin{equation} \label{pp6.1}
\max_{Z \neq Z_0} \frac{\pi(Z|\bm{Y})}
{\pi(Z_0 |\bm{Y})} \leq \max_{Z \neq Z_0}
\prod_{j=1}^{p}PR_j(Z,Z_0) \rightarrow 0, \quad \mbox{as } n \rightarrow \infty,
\end{equation}
which completes our proof of Theorem \ref{thm4}.

\subsection{Proof of Theorem \ref{thm5}} \label{sec:proof_thm5}
We now move on to the proof of Theorem \ref{thm5}. By Lemmas \ref{lm4} - \ref{lm6}, it follows that if we 
restrict to $C_n^c$, then for large enough constant $N > N_3$, we have
\begin{align} \label{thm3_proof}
&\frac{1 - \pi(Z_0 | {\bm Y})}{\pi(Z_0 |{\bm Y})}  \nonumber\\
=&
\sum_{Z \neq Z_0} \frac{\pi(Z|\bm Y)}{\pi(Z_0|\bm Y)} \nonumber \\
\le&  \sum_{j = 1}^{p-1} \sum_{Z_j \neq {Z_0}_j}\frac{\pi(Z|Y)}{\pi(Z_0|\bm Y)} \nonumber \\
\le& \sum_{j = 1}^{p-1} \left(\sum_{Z_j \subset {Z_0}_j}\frac{\pi(Z|\bm Y)}{\pi(Z_0|\bm Y)} + \sum_{Z_j \supset {Z_0}_j}\frac{\pi(Z|\bm Y)}{\pi(Z_0|\bm Y)} + \sum_{Z_j \nsubseteq {Z_0}_j}\frac{\pi(Z|\bm Y)}{\pi(Z_0|\bm Y)} \right) \nonumber \\
\le&  \sum_{j = 1}^{p-1} \bigg(\sum_{\lvert Z_j\rvert = 1}^{ \lvert{Z_0}_j\rvert - 1}\binom{\lvert{Z_0}_j\rvert}{\lvert Z_j \rvert }  p^{-\frac{2c}{\kappa}d} + \sum_{\lvert Z_j\rvert = \lvert {Z_0}_j\rvert}^{R_{n}} \binom{p - \lvert{Z_0}_j\rvert}{\lvert Z_j \rvert - \lvert {Z_0}_j \rvert}  \left(2p\right)^{-\frac c \kappa (|Z_j| - |Z^*_j|)}\nonumber\\
& \qquad + \sum_{\lvert Z_j\rvert = 1 }^{R_{n}} \binom{p}{\lvert Z_j \rvert}  \left(2p\right)^{-\frac c \kappa |Z_j|} \bigg).
\end{align}
Further note that  the upper bound of the binomial coefficient satisfies$\binom {p} {k} \le p^{k}$, for any $1 \le k \le p$. It follows that when $c > 2\kappa$ for some $\kappa > 1$, $$\frac{1 - \pi(Z_0 | {\bm Y})}{\pi(Z_0 |{\bm Y})} \rightarrow 0,\quad \mbox{as } n \rightarrow \infty.$$
Therefore, ${\pi(Z_0 |{\bm Y})} \rightarrow 1$, as $n \rightarrow \infty,$
which completes our proof of the strong model selection result in Theorem \ref{thm5}.

\subsection{Proof of Theorems \ref{thm1} and \ref{thm3}} \label{sec:proof_thm13}
The proof of Theorem \ref{thm1} will also be broken into several steps. We begin proving our posterior ratio consistency result by first proving the Lemma \ref{newlemma} which gives the closed form of the marginal posterior density up to a constant.
\begin{proof} [Proof of Lemma \ref{newlemma}]
	Note that following from model (\ref{model5}) and (\ref{model6}), under the beta-mixture prior, we have 
	\begin{align} \label{marginal_Z}
	\pi (Z) =& \int \pi(q) \prod_{(j,k):1 \leq j < k \leq p} q^{Z_{kj}} \left(1-q 
	\right)^{1-Z_{kj}} dq \nonumber\\
	\propto & \int  \prod_{j=1}^{p-1} q^{\alpha_1 + |Z_j| - 1} (1-q)^{\alpha_2 + p-j-
		|Z_j| - 1} dq \nonumber \\
	\propto& B\left(\alpha_1(p-1)+ \sum_{j = 1}^{p-1}|Z_j|, \alpha_2(p-1) + \frac{p(p-1)}{2}-
	\sum_{j = 1}^{p-1}|Z_j|\right),
	\end{align}
	where 
	\begin{align*}
	&B\left(\alpha_1(p-1)+ \sum_{j = 1}^{p-1}|Z_j|, \alpha_2(p-1) + \frac{p(p-1)}{2}-
	\sum_{j = 1}^{p-1}|Z_j|\right) \\
	=& \frac{\Gamma(\alpha_1(p-1)+ \sum_{j = 1}^{p-1}|Z_j|)\Gamma(\alpha_2(p-1) + \frac{p(p-1)}{2}-
		\sum_{j = 1}^{p-1}|Z_j|)}{\Gamma((\alpha_1 + \alpha_2)(p-1) + \frac{p(p-1)}{2})}.
	\end{align*}
	Similar to the argument in (\ref{posterior1}) and (\ref{posterior2}), integrating out $(L,D)$ gives us
	\begin{align} \label{posterior_propto1}
	&\pi(Z | \bm Y)\nonumber\\
	\propto& B\left(\alpha_1(p-1)+ \sum_{j = 1}^{p-1}|Z_j|, \alpha_2(p-1) + \frac{p(p-1)}{2}-
	\sum_{j = 1}^{p-1}|Z_j|\right) \nonumber\\
	&\times \prod_{j = 1}^{p-1}  \left(\frac{n\tilde S_{j|Z_j}}{2}- \frac 1 {2\tau^2} + \lambda_2\right)^{-\frac n 2 - \lambda_1}  \frac{|\tilde S_Z^{>j}|^{-\frac 1 2}}{(n\tau^2)^{|Z_j|/2}}\nonumber\\ 
	= & B\left(\alpha_1(p-1)+ \sum_{j = 1}^{p-1}|Z_j|, \alpha_2(p-1) + \frac{p(p-1)}{2}-
	\sum_{j = 1}^{p-1}|Z_j|\right) \nonumber\\
	&\times \prod_{j = 1}^{p-1} \left(\frac{n\tilde S_{j|Z_j}}{2} - \frac 1 {2\tau^2} + \lambda_2\right)^{-\frac n 2 - \lambda_1} \frac{\left(|\tilde S_{Z}^{\ge i}|\tilde{S}_{j|Z_j}\right)^{-\frac 1 2}}{(n\tau^2)^{|Z_j|/2}},
	\end{align}
	in which $\tilde{S}_{j|Z_j} = \tilde{S}_{jj} - (\tilde{S}_{Z \cdot j}^>)^T 
	(\tilde{S}_Z^{>j})^{-1} \tilde{S}_{Z \cdot j}^>$. 
\end{proof}
Now we are interested in obtaining the posterior ratio. It immediately follows from Lemma \ref{newlemma} that, given the data $\bm Y$, the posterior ratio for any $Z$ compared to $Z_0$ can be simplified as
\begin{align} \label{m5}
&\frac{\pi({Z}|\bm{Y})}{\pi({Z}_0|\bm{Y})} \nonumber\\
=& \frac{B\left(\alpha_1(p-1)+ \sum_{j = 1}^{p-1}|Z_j|, \alpha_2(p-1) + \frac{p(p-1)}{2}-
	\sum_{j = 1}^{p-1}|Z_j|\right)}{B\left(\alpha_1(p-1)+ \sum_{j = 1}^{p-1}|{Z_0}_j|, \alpha_2(p-1) + \frac{p(p-1)}{2}-
	\sum_{j = 1}^{p-1}|{{Z_0}_j|}\right)}\nonumber \\
&\times  \prod_{j=1}^{p-1} (n\tau^2)^{-\frac{|Z_j| - |{Z_0}_j|}2} \frac{|\tilde{S}_{Z_0}^{\ge j}|^{\frac12}}{|\tilde{S}_{Z}^{\ge j}|^{\frac12}}\left(\frac{\tilde{S}_{j|{Z_0}_j}}{\tilde{S}_{j|Z_j}}\right)^{\frac 1 2} \left(\frac{\tilde S_{j|{Z_0}_j}- \frac 1 {n\tau_{n,p}^2}+ \lambda_2}{\tilde S_{j|{Z}_j} - \frac 1 {n\tau_{n,p}^2}+ \lambda_2}\right)^{\frac n 2 + \lambda_1}.
\end{align} 
We begin by simplifying the posterior ratio given in (\ref{m5}). Using the fact that 
$
\sqrt{x + \frac{1}{4}} \leq \frac{\Gamma (x+1)}{\Gamma \left( x+\frac{1}{2} \right)} 
\leq \sqrt{x + \frac{1}{2}} 
$
for $x > 0$ (see \citep*{Watson:1959}), it follows from Assumption \ref{assumption:Z}, and $1 + x \le e^x$, $1 - x \le e^{-x}$, for $0 \le x \le 1$, that for a large enough constant $M$, 
and large enough $n$, we have  
\begin{align} \label{ratio_B}
&\frac{B\left(\alpha_1(p-1)+ \sum_{j = 1}^{p-1}|Z_j|, \alpha_2(p-1) + \frac{p(p-1)}{2}-
	\sum_{j = 1}^{p-1}|Z_j|\right)}{B\left(\alpha_1(p-1)+ \sum_{j = 1}^{p-1}|{Z_0}_j|, \alpha_2(p-1) + \frac{p(p-1)}{2}-
	\sum_{j = 1}^{p-1}|{{Z_0}_j|}\right)} \nonumber\\
=& \frac{\Gamma(\alpha_1(p-1)+ \sum_{j = 1}^{p-1}|Z_j|)\Gamma(\alpha_2(p-1) + \frac{p(p-1)}{2}-
	\sum_{j = 1}^{p-1}|Z_j|)}{\Gamma(\alpha_1(p-1)+ \sum_{j = 1}^{p-1}|{Z_0}_j|)\Gamma(\alpha_2(p-1) + \frac{p(p-1)}{2}-
	\sum_{j = 1}^{p-1}|{Z_0}_j|)} \nonumber\\
\le& \prod_{j = 1}^{p-1}M^{\lvert|Z_j| - |{Z_0}_j|\rvert}\left(\frac{\alpha_2 + p/2 - \frac{\sum_{j = 1}^{p-1}|Z_j|}{p-1}}{\alpha_1+\frac{\sum_{j = 1}^{p-1}|Z_j|}{p-1}}\right)^{-(|Z_j| - |{Z_0}_j|)},
\end{align}
for some constant $M > 0$. \\
Therefore, the posterior ratio in (\ref{m5}) can be bounded above by
\begin{align}
&\frac{\pi({Z}|\bm{Y})}{\pi({Z}_0|\bm{Y})} \nonumber\\
\le &  \prod_{j = 1}^{p-1}M^{\lvert|Z_j| - |{Z_0}_j|\rvert}\left(\frac{\alpha_1+\frac{\sum_{j = 1}^{p-1}|Z_j|}{p-1}}{\alpha_2 + p/2 - \frac{\sum_{j = 1}^{p-1}|Z_j|}{p-1}}\right)^{-(|Z_j| - |{Z_0}_j|)} (n\tau^2)^{-\frac{|Z_j| - |{Z_0}_j|}2}  \nonumber \\
&\times \frac{|\tilde{S}_{Z_0}^{\ge j}|^{\frac12}}{|\tilde{S}_{Z}^{\ge j}|^{\frac12}}\left(\frac{\tilde{S}_{j|{Z_0}_j}}{\tilde{S}_{j|Z_j}}\right)^{\frac 1 2} \left(\frac{\tilde S_{j|{Z_0}_j}- \frac 1 {n\tau_{n,p}^2}+ \lambda_2}{\tilde S_{j|{Z}_j} - \frac 1 {n\tau_{n,p}^2}+ \lambda_2}\right)^{\frac n 2 + \lambda_1}\nonumber \\
\triangleq& PR_j(Z,Z_0).
\end{align}
We now analyze the behavior of $PR_j(Z,Z_0)$ defined in (\ref{m5}) under different scenarios in a sequence of three lemmas (Lemmas \ref{lm1} - \ref{lm3}). Recall that our 
goal is to find an upper bound for 
$PR_j(Z,Z_0)$, such that the upper bound converges to $0$ as $n 
\rightarrow \infty$. For all the following analyses, we will restrict ourselves to the event $C_n^c$. 
\begin{lemma} \label{lm1}
	If all the active elements in set ${Z_j}_0$ are contained in the true model ${Z}_j$ denoted as $Z_j \supset {Z_0}_j$, then there exists $N_4$ (not depending on $Z$) such that for $n \geq N_4$ we have for some constant $\kappa > 1$, $PR_j(Z,Z_0) \le \left(2p\right)^{-\frac {\max\{c,1\}} \kappa (|Z_j| - |{Z_0}_j|)}  \rightarrow 0,  \mbox{ as } n \rightarrow \infty.$
\end{lemma}
\begin{proof} [Proof of Lemma \ref{lm1}]
	We begin by simplifying the posterior ratio given in (\ref{m5}). It follows from Assumption \ref{assumption:alpha_2},$|Z_j| > |{Z_0}_j|$,  that for a large enough constant $M$, 
	and large enough $n$, we have  
	\begin{align} \label{ratio_B_lemma1}
	&M^{\lvert|Z_j| - |{Z_0}_j|\rvert}\left(\frac{\alpha_2 + p/2 - \frac{\sum_{j = 1}^{p-1}|Z_j|}{p-1}}{\alpha_1+\frac{\sum_{j = 1}^{p-1}|Z_j|}{p-1}}\right)^{-(|Z_j| - |{Z_0}_j|)}  \nonumber\\
	\le& (c_1p^{\max\{c,1\}}/n)^{-(|Z_j| - |{Z_0}_j|)},
	\end{align}
	for some constant $c_1 > 0$.\\
	Following the similar arguments leading up to (\ref{PR_upbound_1}), by Assumption \ref{assumption:Z}, Assumption \ref{assumption:tau}, for larger enough $n \ge N_4$, we have
	\begin{align} \label{PR_upbound_1}
	&PR_j(Z,Z_0) \nonumber \\
	\le& (c_1\sqrt{n\tau^2}p^{\max\{c,1\}}/n)^{-(|Z_j| - |{Z_0}_j|)}\left(1 + \frac{n{d_0}_j^{-1}{S}_{j|{Z_0}_j} - n{d_0}_j^{-1}S_{j|{Z_0}_j} + c_1 \frac d {{d_0}_j\tau_{n,p}^2}}{n{d_0}_j^{-1}S_{j|{Z}_j}}\right)^{\frac 1 2}\nonumber \\ 
	&\times \left(1 + \frac{n{d_0}_j^{-1}{S}_{j|{Z_0}_j} - n{d_0}_j^{-1}S_{j|{Z_0}_j} + c_1 \frac d {{d_0}_j\tau_{n,p}^2}}{n{d_0}_j^{-1}S_{j|{Z}_j}}\right)^{\frac n 2 + \lambda_1}\nonumber \\ 
	\le & (2p^{-{\max\{c,1\}}})^{|Z_j| - |{Z_0}_j|}\left(\sqrt{\frac {\tau^2} n} \log n\right)^{-(|Z_j| - |{Z_0}_j|)} \nonumber\\
	&\times  \exp\left\{\frac{|Z_j| - |{Z_0}_j| + \sqrt{(|Z_j| - |{Z_0}_j|)\log p} + c_1\frac d {\tau_{n,p}^2}}{n - |Z_j| -  \sqrt{(n - |{Z}_j|)\log p}} \times (\frac {n+1} 2 + \lambda_1) \right\}\\ \nonumber
	\le & \left(2p\right)^{-\frac {\max\{c,1\}} \kappa (|Z_j| - |{Z_0}_j|)}, \mbox{ for some constant } \kappa > 1.
	\end{align}
	The second inequality follows from $\frac{d}{\tau_{n,p}^2\log p} \rightarrow 0$, as $n \rightarrow \infty$.
\end{proof}
\begin{lemma} \label{lm2}
	If all the active elements in set $Z_j$ are contained in the true model ${Z_0}_j$ denoted as $Z_j \subset {Z_0}_j$, then there exists $N_5$ (not depending on $Z$) such that for $n \geq N_5$, we have $PR_j(Z,Z_0) \le p^{-\frac{2c}{\kappa}d} \rightarrow 0,  \mbox{ as } n \rightarrow \infty.$
\end{lemma}
\begin{proof} [Proof of Lemma \ref{lm2}]
	Now we move to discuss the scenario when $Z_j$ is a subset of ${Z_0}_j$, i.e., $Z_j \subset {Z_0}_j$. 
	By the similar arguments in (\ref{ratio_B}), it follows from Assumption \ref{assumption:alpha_2} and $|Z_j| < |{Z_0}_j|$, that for a large enough constant $c_1$ 
	and large enough $n$, we have  
	\begin{align} \label{ratio_B_lemma2}
	&M^{\lvert|Z_j| - |{Z_0}_j|\rvert}\left(\frac{\alpha_2 + p/2 - \frac{\sum_{j = 1}^{p-1}|Z_j|}{p-1}}{\alpha_1+\frac{\sum_{j = 1}^{p-1}|Z_j|}{p-1}}\right)^{-(|Z_j| - |{Z_0}_j|)}  \nonumber\\
	\le& (c_1p^{\max\{c,1\}})^{|{Z_0}_j| - |{Z}_j|},
	\end{align}

	It follows from the similar arguments leading up to (\ref{PR_upbound_subset}) that there exists $N_5 > 0$, such that for $n \ge N_5$,
	\begin{align} \label{PR_upbound_subset_scale}
	PR_j(Z, Z_0)  \le p^{-\frac{2c}{\kappa}d}.
	\end{align}
\end{proof}
\begin{lemma} \label{lm3}
	If all the active elements in set $Z_j$ are not contained in the true model ${Z_0}_j$ and all the active elements in set ${Z_0}_j$ are not contained in the true model $Z_j$, denoted as ${Z_0}_j \neq Z_j$, ${Z_0}_j \nsubseteq Z_j$, and ${Z_0}_j \nsupseteq Z_j$, then there exists $N_6$ (not depending on $Z$) such that for $n \geq N_6$ we have $PR_j(Z,Z_0) \le\left(2p\right)^{-\frac {\max\{c,1\}} \kappa (|Z_j| - |Z^*_j|)-\frac{2c}{\kappa}d} \rightarrow 0,  \mbox{ as } n \rightarrow \infty.$
\end{lemma}
\begin{proof} [Proof of Lemma \ref{lm3}]
	Let $Z^*$ be an arbitrary 0-1 matrix satisfying $Z^*_j = Z_j \cap {Z_0}_j$. Immediately we get $pa_i({{\mathscr{D}}^*} ) \subset pa_i({\mathscr{D}} _0)$ and $pa_i({{\mathscr{D}}^*} ) \subset pa_i({\mathscr{D}})$. It follows from (\ref{m5}) that
	\begin{align} \label{pp3}
	PR_j(Z, Z_0) \le&  M^{\lvert|Z_j| - |{Z_0}_j|\rvert}\left(\frac{\alpha_1+\frac{\sum_{j = 1}^{p-1}|Z_j|}{p-1}}{\alpha_2 + p/2 - \frac{\sum_{j = 1}^{p-1}|Z_j|}{p-1}}\right)^{-(|Z_j| - |{Z_0}_j|)} (n\tau^2)^{-\frac{|Z_j| - |{Z_0}_j|}2}  \nonumber \\
	&\times \frac{|\tilde{S}_{Z_0}^{\ge j}|^{\frac12}}{|\tilde{S}_{Z}^{\ge j}|^{\frac12}}\left(\frac{\tilde{S}_{j|{Z_0}_j}}{\tilde{S}_{j|Z_j}}\right)^{\frac 1 2} \left(\frac{\tilde S_{j|{Z_0}_j}- \frac 1 {n\tau_{n,p}^2}+ \lambda_2}{\tilde S_{j|{Z}_j} - \frac 1 {n\tau_{n,p}^2}+ \lambda_2}\right)^{\frac n 2 + \lambda_1} \nonumber \\
	\le& PR_j(Z, Z^*)\times PR_j(Z^*, Z_0).	
	\end{align}
	\noindent
	Note that $Z^*_j \subset Z_j$. It follows from (\ref{PR_upbound_1}) that 
	\begin{align} \label{PR_neq_1} 
	PR_j(Z, Z^*) \le \left(2p\right)^{-\frac {\max\{c,1\}} \kappa (|Z_j| - |Z^*_j|)}, \mbox{ for some } \kappa > 1 \mbox{ and } n\ge N_4.
	\end{align}
	By (\ref{PR_upbound_subset}) and $Z^*_j \subset {Z_0}_j$, we have
	\begin{equation} 
	PR_j(Z^*, Z_0)  \le  p^{-\frac{2c}{\kappa}d},
	\mbox{ for } n \ge N_5. 
	\end{equation}
	\noindent
	It follows from (\ref{pp3}) and $|Z^*_j| < d$ that
	\begin{align} \label{PR_upbound_neq}
	PR_j(Z, Z_0) \le \left(2p\right)^{-\frac {\max\{c,1\}} \kappa (|Z_j| - |Z^*_j|)-\frac{2c}{\kappa}d}, \mbox{ for } n > N_6 = \max\left\{N_4,N_5\right\}.
	\end{align}
\end{proof}
	For any $Z \neq Z_0$, it follows that there exists at least one $1 \le i \le p -1$, such that $Z_j \neq {Z_0}_j$. Hence, by (\ref{PR_upbound_1}), (\ref{PR_upbound_subset}) and (\ref{PR_upbound_neq}), we have
	\begin{align}
	PR_j(Z, Z_0) \rightarrow 0, \mbox{ as } n \rightarrow \infty.
	\end{align}	
\noindent
The results of Theorem \ref{thm1} and \ref{thm3} can be immediately obtained from Lemma \ref{lm1} to Lemma \ref{lm3} by following the same arguments leading up to (\ref{thm3_proof}).

\section{Discussion} \label{sec:discussion}
In this paper, we investigate the theoretical consistency properties for the high-dimensional sparse DAG models based on the spike and slab prior introduced on the Cholesky parameter and appropriate multiplicative and beta-mixture priors on the indicator probabilities. We establish both posterior ratio consistency and the strong model selection consistency under more general conditions than those in the existing literature. In particular, our consistency result requires much more relaxed conditions on the dimensionality and sparsity. In addition, rather than treating $q$ as a constant and controlling its rate, by either putting an extra layer prior on $q$ or placing the multiplicative prior over the space of $Z$, we avoid the potential issues of the model being stuck in rather sparse space. Finally, the simulation study shows that not only the proposed models yield desired asymptotic consistency, in the same time can also give a better model selection performance.
\bibliographystyle{plainnat}
\bibliography{references}

\begin{thebibliography}{29}
\providecommand{\natexlab}[1]{#1}
\providecommand{\url}[1]{\texttt{#1}}
\expandafter\ifx\csname urlstyle\endcsname\relax
  \providecommand{\doi}[1]{doi: #1}\else
  \providecommand{\doi}{doi: \begingroup \urlstyle{rm}\Url}\fi

\bibitem[Aragam et~al.(2015)Aragam, Amini, and Zhou]{AAZ:2015}
B.~Aragam, A.~Amini, and Q.~Zhou.
\newblock Learning directed acyclic graphs with penalized neighbourhood
  regression.
\newblock \emph{https://arxiv.org/abs/1511.08963}, 2015.

\bibitem[Banerjee and Ghosal(2014)]{Banerjee:Ghosal:2014}
S.~Banerjee and S.~Ghosal.
\newblock Posterior convergence rates for estimating large precision matrices
  using graphical models.
\newblock \emph{Electronic Journal of Statistics}, 8:\penalty0 2111--2137,
  2014.

\bibitem[Banerjee and Ghosal(2015)]{Banerjee:Ghosal:2015}
S.~Banerjee and S.~Ghosal.
\newblock Bayesian structure learning in graphical models.
\newblock \emph{Journal of Multivariate Analysis}, 136:\penalty0 147--162,
  2015.

\bibitem[Ben-David et~al.(2016)Ben-David, Li, Massam, and
  Rajaratnam]{BLMR:2016}
E.~Ben-David, T.~Li, H.~Massam, and B.~Rajaratnam.
\newblock High dimensional bayesian inference for gaussian directed acyclic
  graph models.
\newblock \emph{Technical Report}, http://arxiv.org/abs/1109.4371, 2016.

\bibitem[Bickel and Levina(2008{\natexlab{a}})]{Bickel:Levina:2008}
P.~J. Bickel and E.~Levina.
\newblock Regularized estimation of large covariance matrices.
\newblock \emph{Ann. Statist.}, 36:\penalty0 199--227, 2008{\natexlab{a}}.

\bibitem[Bickel and Levina(2008{\natexlab{b}})]{bickel:2008:thres}
Peter~J. Bickel and Elizaveta Levina.
\newblock Covariance regularization by thresholding.
\newblock \emph{Ann. Statist.}, 36\penalty0 (6):\penalty0 2577--2604, 12
  2008{\natexlab{b}}.
\newblock \doi{10.1214/08-AOS600}.
\newblock URL \url{https://doi.org/10.1214/08-AOS600}.

\bibitem[Cao et~al.(2018)Cao, Khare, and Ghosh]{CKG:nonlocal}
X.~Cao, K.~Khare, and M.~Ghosh.
\newblock High-dimensional posterior consistency for hierarchical non-local
  priors in regression.
\newblock \emph{https://arxiv.org/abs/1709.06607}, 2018.

\bibitem[Cao et~al.(2019)Cao, Khare, and Ghosh]{CKG:2017}
X.~Cao, K.~Khare, and M.~Ghosh.
\newblock Posterior graph selection and estimation consistency for
  high-dimensional bayesian dag models.
\newblock \emph{Ann. Statist.}, 47\penalty0 (1):\penalty0 319--348, 02 2019.

\bibitem[Carvalho and Scott(2009)]{Carvalho:Scott:2009}
C.~M. Carvalho and J.~G. Scott.
\newblock {Objective Bayesian model selection in Gaussian graphical models}.
\newblock \emph{Biometrika}, 96\penalty0 (3):\penalty0 497--512, 05 2009.

\bibitem[El~Karoui(2008)]{ElKaroui:2008}
N.~El~Karoui.
\newblock Spectrum estimation for large dimensional covariance matrices using
  random matrix theory.
\newblock \emph{Annals of Statistics}, 36:\penalty0 2757--2790, 2008.

\bibitem[El~Karoui(2007)]{elkaroui:2007}
Noureddine El~Karoui.
\newblock Tracy–widom limit for the largest eigenvalue of a large class of
  complex sample covariance matrices.
\newblock \emph{Ann. Probab.}, 35\penalty0 (2):\penalty0 663--714, 03 2007.
\newblock \doi{10.1214/009117906000000917}.

\bibitem[Huang et~al.(2006)Huang, Liu, Pourahmadi, and Liu]{HLPL:2006}
J.~Huang, N.~Liu, M.~Pourahmadi, and L.~Liu.
\newblock Covariance selection and estimation via penalised normal likelihood.
\newblock \emph{Biometrika}, 93:\penalty0 85--98, 2006.

\bibitem[Johnson and Rossell(2012)]{Johnson:Rossell:2012}
V.~Johnson and D.~Rossell.
\newblock Bayesian model selection in high-dimensional settings.
\newblock \emph{J. Amer. Statist. Assoc}, 107\penalty0 (498):\penalty0
  649--660, 2012.

\bibitem[Khare et~al.(2017)Khare, Oh, Rahman, and Rajaratnam]{KORR:2017}
K.~Khare, S.~Oh, S.~Rahman, and B.~Rajaratnam.
\newblock A convex framework for high-dimensional sparse cholesky based
  covariance estimation in gaussian dag models.
\newblock \emph{Preprint, Department of Statisics, University of Florida},
  2017.

\bibitem[Lee and Lee(2017)]{LL:posterior}
K.~Lee and J.~Lee.
\newblock Estimating large precision matrices via modified cholesky
  decomposition.
\newblock \emph{https://arxiv.org/abs/1707.01143}, 2017.

\bibitem[Lee et~al.(2018)Lee, Lee, and Lin]{LLL:2018}
Kyoungjae Lee, Jaeyong Lee, and Lizhen Lin.
\newblock Minimax posterior convergence rates and model selection consistency
  in high-dimensional dag models based on sparse cholesky factors.
\newblock \emph{Ann. Statist., to appear}, 2018.

\bibitem[Narisetty and He(2014)]{Narisetty:He:2014}
N.~Narisetty and X.~He.
\newblock Bayesian variable selection with shrinking and diffusing priors.
\newblock \emph{Ann. Statist.}, 42:\penalty0 789--817, 2014.

\bibitem[Pourahmadi(2007)]{Pourahmadi:2007}
M.~Pourahmadi.
\newblock Cholesky decompositions and estimation of a covariance matrix:
  Orthogonality of variance--correlation parameters.
\newblock \emph{Biometrika}, 94:\penalty0 1006--1013, 2007.

\bibitem[Rothman et~al.(2010)Rothman, Levina, and Zhu]{RLZ:2010}
A.~J. Rothman, E.~Levina, and J.~Zhu.
\newblock A new approach to cholesky-based covariance regularization in high
  dimensions.
\newblock \emph{Biometrika}, 97:\penalty0 539--550, 2010.

\bibitem[Rudelson and Vershynin(2013)]{rudelson2013}
Mark Rudelson and Roman Vershynin.
\newblock Hanson-wright inequality and sub-gaussian concentration.
\newblock \emph{Electronic Communications in Probability}, 18:\penalty0 9 pp.,
  2013.

\bibitem[Rutimann and Buhlmann(2009)]{Rutimann:Buhlmann:2009}
P.~Rutimann and P.~Buhlmann.
\newblock High dimensional sparse covariance estimation via directed acyclic
  graphs.
\newblock \emph{Electronic Journal of Statistics}, 3:\penalty0 1133--1160,
  2009.

\bibitem[Shin et~al.(2018)Shin, Bhattacharya, and Johnson]{Shin.M:2015}
M.~Shin, A.~Bhattacharya, and V.~Johnson.
\newblock Scalable bayesian variable selection using nonlocal prior densities
  in ultrahigh-dimensional settings.
\newblock \emph{Statist. Sinica}, 28:\penalty0 1053--1078, 2018.

\bibitem[Shojaie and Michailidis(2010)]{Shojaie:Michailidis:2010}
A.~Shojaie and G.~Michailidis.
\newblock Penalized likelihood methods for estimation of sparse
  high-dimensional directed acyclic graphs.
\newblock \emph{Biometrika}, 97:\penalty0 519--538, 2010.

\bibitem[Tan et~al.(2017)Tan, Jasra, De~Iorio, and Ebbels]{Tan:2017}
Linda S.~L. Tan, Ajay Jasra, Maria De~Iorio, and Timothy M.~D. Ebbels.
\newblock Bayesian inference for multiple gaussian graphical models with
  application to metabolic association networks.
\newblock \emph{Ann. Appl. Stat.}, 11\penalty0 (4):\penalty0 2222--2251, 12
  2017.

\bibitem[Watson(1959)]{Watson:1959}
G.N. Watson.
\newblock A note on gamma functions.
\newblock \emph{Proc. Edinburgh Math. Soc.}, 11:\penalty0 7--9, 1959.

\bibitem[Xiang et~al.(2015)Xiang, Khare, and Ghosh]{XKG:2015}
R.~Xiang, K.~Khare, and M.~Ghosh.
\newblock High dimensional posterior convergence rates for decomposable
  graphical models.
\newblock \emph{Electronic Journal of Statistics}, 9:\penalty0 2828--2854,
  2015.

\bibitem[Xu and Ghosh(2015)]{Xu:Ghosh:2015}
Xiaofan Xu and Malay Ghosh.
\newblock Bayesian variable selection and estimation for group lasso.
\newblock \emph{Bayesian Anal.}, 10\penalty0 (4):\penalty0 909--936, 12 2015.
\newblock \doi{10.1214/14-BA929}.

\bibitem[Yang et~al.(2016)Yang, Wainwright, and Jordan]{yang:2016}
Yun Yang, Martin~J. Wainwright, and Michael~I. Jordan.
\newblock On the computational complexity of high-dimensional bayesian variable
  selection.
\newblock \emph{Ann. Statist.}, 44\penalty0 (6):\penalty0 2497--2532, 12 2016.
\newblock \doi{10.1214/15-AOS1417}.
\newblock URL \url{https://doi.org/10.1214/15-AOS1417}.

\bibitem[Yu and Bien(2016)]{Yu:Bien:2016}
G.~Yu and J.~Bien.
\newblock Learning local dependence in ordered data.
\newblock \emph{arXiv:1604.07451}, 2016.

\end{thebibliography}
\end{document}